\documentclass[a4paper,american,11pt]{amsart}

\usepackage{babel}
\usepackage{csquotes}

\usepackage{amsmath}
\usepackage{amssymb}
\usepackage{amsfonts}
\usepackage{amscd}

\usepackage{graphicx}
\usepackage{caption}
\usepackage{subcaption}
\usepackage{wrapfig}
\usepackage{pdflscape}
\usepackage{float}

\input xy
\xyoption{all}

\usepackage{listings}
\usepackage[cmyk,table]{xcolor}
\definecolor{cobalt}{rgb}{0.0, 0.28, 0.67}

\usepackage{tikz-cd}
\usepackage{tikz}
\usetikzlibrary{arrows.meta,positioning,fit}
\usetikzlibrary{fadings}
\usetikzlibrary{decorations.markings}
\usetikzlibrary{patterns}
\usetikzlibrary{3d}

\usepackage{pgfplots}
\pgfplotsset{compat=1.16}

\usepackage{tkz-fct}
\usepackage{tikz-qtree}

\usepackage[maxnames=5]{biblatex}
\usepackage{hyperref}
\hypersetup{
  colorlinks=true,
  linkcolor=cobalt,
  citecolor=cobalt,
  urlcolor=cobalt
}

\theoremstyle{plain}
\newtheorem{theorem}{Theorem}[section]
\newtheorem{proposition}[theorem]{Proposition}
\newtheorem{lemma}[theorem]{Lemma}
\newtheorem{corollary}[theorem]{Corollary}

\theoremstyle{definition}
\newtheorem{definition}[theorem]{Definition}
\newtheorem{example}[theorem]{Example}
\newtheorem{remark}[theorem]{Remark}

\newcommand{\N}{\mathbb N}
\newcommand{\Z}{\mathbb Z}
\newcommand{\Q}{\mathbb Q}
\newcommand{\R}{\mathbb R}
\newcommand{\C}{\mathbb C}

\let\ZZ\Z

\let\RR\R

\newcommand{\ind}{\mathbf 1}
\newcommand{\dd}{\,d}
\newcommand{\floor}[1]{\left\lfloor #1\right\rfloor}
\newcommand{\ceil}[1]{\left\lceil #1\right\rceil}

\DeclareMathOperator*{\Res}{Res}
\DeclareMathOperator{\Area}{Area}
\DeclareMathOperator{\Length}{Length}

\DeclareMathOperator{\TV}{TV}

\begin{document}

\title[Residues of a tropical zeta function]
{Residues of a Tropical Zeta Function for Convex Domains}

\author{Nikita Kalinin \and Ernesto Lupercio \and Mikhail Shkolnikov}

\date{June 23, 2026}

\begin{abstract}
For a compact convex domain \(\Omega\subset\mathbb{R}^n\), we introduce the
\(\mathrm{SL}(n,\mathbb{Z})\)-invariant tropical zeta function
\[
Z_\Omega(s)=\int_\Omega \rho_\Omega(x)^{\,s-n}\,dx,
\]
where \(\rho_\Omega\) is the minimum of the supporting affine functions with
primitive integral gradients. Its singularities encode geometric properties
of \(\partial\Omega\).

In dimension \(2\), every compact convex domain has a canonical
rational-slope minimal model \(\widehat\Omega\), from which it is obtained
by unimodular corner cuts. After multiplication by \(s(s-1)\),
\(Z_\Omega(s)\) differs from the negative of the Dirichlet series of the
cut sizes by an explicit holomorphic term determined by
\(\widehat\Omega\). On each boundary arc, this series is also given by the
coefficients of the Legendre dual in Hata's Schauder basis. Thus the study
of \(Z_\Omega\) becomes a boundary problem in convex geometry, Farey
arithmetic, and tropical optics.

For rational-slope polygons, \(Z_\Omega\) extends meromorphically to
\(\mathbb{C}\), with simple poles at \(s=1\) and \(s=0\). Their residues
are the lattice perimeter and, when defined, the negative self-intersection
of the canonical class of the associated toric surface. For a strictly
convex planar domain with \(C^3\) boundary and everywhere nonvanishing
curvature, the pole at \(s=1\) disappears, while primitive supporting
directions produce a simple pole at \(s=\frac23\):
\[
\operatorname*{Res}_{s=2/3}Z_\Omega(s)
=
\frac{3^{5/2}}{2^{5/3}\pi^3}
\Gamma\!\left(\frac13\right)^3
\operatorname{Length}_{\mathrm{equiaffine}}(\partial\Omega).
\]
Hence the first smooth singularity converts the arithmetic
\(\mathrm{SL}(2,\mathbb{Z})\)-symmetry into the
\(\mathrm{SL}(2,\mathbb{R})\)-invariant equiaffine geometry of the
boundary. We also continue the boundary series to
\(\Re(s)>\frac35\) and derive, by a Tauberian argument, the small-time
asymptotic for the lattice perimeter of the tropical wave fronts
\(\Omega_t=\{x\in\Omega:\rho_\Omega(x)\geq t\}\).

Finally, we study the parabolic domain
\[
L=\{(x,y)\in\mathbb{R}^2:
\sqrt{1-|x|}+\sqrt{1-|y|}\geq1\}.
\]
Its tropical zeta function is explicitly expressible through Witten's
\(\mathrm{SU}(3)\) zeta function, and \(L\) is the deterministic limit
shape for convex lattice polygons in the square. The same Farey cells carry
the cubic weights \(m_{p,q}=pq(p+q)\) governing \(Z_L(s)\). In the exact
head--tail decomposition of the averaged lattice-point count, the apparent
\(N^{2/3}\) contribution from the pole at \(s=\frac23\) cancels against
explicit arithmetic corrections from the rational tropical coefficients.
The first surviving term is
\[
N^{-3/2}\sum_{n\leq N}
\left(
|nL\cap\mathbb{Z}^2|-\frac{10}{3}n^2
\right)
\longrightarrow
\frac{32}{3}\zeta\!\left(-\frac12\right)
\frac{\zeta(2)}{\zeta(3)}.
\]
Thus, for \(L\), one Farey decomposition determines the Witten
\(\mathrm{SU}(3)\) expression for \(Z_L\), its equiaffine residue, and the
first surviving averaged lattice-point fluctuation.
\end{abstract}

\maketitle
\newpage
\tableofcontents

\section{Introduction}

\subsection{Broader view}\label{sec1.1}

The Gauss circle problem \cite{gauss1826nexu} is an archetypal problem at
the interface of analytic number theory and planar geometry. Although it has
been studied intensively for more than two centuries, it remains open, despite
substantial progress obtained by a wide range of methods and ideas
\cite{BerndtKimZaharescu2018}. One aim of the present article is to introduce
a new tool into this circle of questions, arising from recent developments in
what may be called ``non-algebraic tropical geometry'' or ``tropical optics.''
Namely, to a compact convex domain \(\Omega\) one can canonically associate a
function \(Z_\Omega(s)\) of one complex variable, which we call its tropical
zeta function. Just as the zeros of the Riemann zeta function govern fine
asymptotic properties of the distribution of prime numbers
\cite{Riemann1859,MazurStein2016}, we expect the poles of the tropical zeta
function to play a central role in obtaining sharp bounds on the error term in
lattice-point counting for dilates of the domain.

For a planar convex domain, the main term in lattice-point counting is
quadratic and equals the area of the domain. A general bound for the error
term is linear in the dilation parameter \(R\). The coefficient of this linear
term is determined by the straight segments of the boundary, and in the
strictly convex case one can do better. Assuming that the boundary is
sufficiently smooth and has everywhere nonvanishing curvature, van der Corput
proved \cite{vanDerCorput1920} about a century ago that the error term is at
most of order \(R^{2/3}\); see also
\cite{IvicKratzelKuehleitnerNowak2004,Huxley2000}. In terms of tropical zeta
functions, the first bound corresponds to the existence of a rightmost pole at
\(s=1\) for domains whose boundary contains rational straight segments, with
residue equal to the lattice perimeter. For domains with sufficiently smooth
boundary, the next pole, moving from right to left, appears at
\(s=\frac23\), and the corresponding residue is universally proportional to
the equiaffine arc length of the boundary. This is the main result of the paper
and the technically deepest part of the argument.

The definition of the tropical zeta function considered in this paper
originates in the mathematical study of self-organized criticality. More
precisely, finite-point perturbations of the maximal stable state in the
Abelian sandpile model give rise to tropical analytic curves and to the
tropical power series defining them
\cite{KalininShkolnikov2015TropicalCurvesSandpileModels,
KalininShkolnikov2020SandpileSolitons,
KalininShkolnikov2018IntroductionTropicalSeries,
KalininShkolnikov2016TropicalCurvesSandpiles}. One such series,
corresponding to the action of a point perturbation at the distinguished
central point, is the tropical distance function of a convex domain, while the
associated tropical analytic curve is the so-called tropical caustic
\cite{MikhalkinShkolnikov2023WaveFrontsCaustics,
Shkolnikov2025PlanarTropicalCaustics}. The convex domain itself appears as
the domain of convergence of this tropical series, which defines a positive,
concave, piecewise-linear function on its interior. Integrating its \(s\)-th
power over the domain then produces a function of the complex variable \(s\),
which we call the tropical zeta function of the domain. The central conjecture
for this class of functions is that they admit meromorphic continuation to the
entire complex plane.

One elementary case in which the conjecture holds is that of polygons whose
sides have rational slopes. Such polygons are naturally viewed as moment
polytopes of compact symplectic toric surfaces. From the algebro-geometric
point of view, these surfaces may be obtained from a minimal model by a
sequence of blow-ups at torus-fixed points. From the symplectic point of view,
each blow-up comes with a natural ``size'', measured by the symplectic area of
the corresponding exceptional divisor. In this setting, our zeta function may
be interpreted as a Dirichlet generating series for the sizes of these
blow-ups, with the symplectic structure encoded through tropical optics and
with a particular minimal model and blow-up chain chosen canonically. From
this perspective, one sees that the corresponding zeta function extends
meromorphically to the whole complex plane. Moreover, it has only two poles,
both reflecting basic geometric information about the toric surface: the
residue at \(s=1\) again corresponds to the symplectic area of the
anticanonical class, while the residue at \(s=0\) is the negative of its
self-intersection. The same, though infinite, triangle-cutting mechanism also
appears in the algebro-geometric story
\cite{KramerPippich2015Snapshot,BurgosGilKramerKuehn2016} behind the
universal elliptic curve over the moduli space \(A_1\) of elliptic curves.

Beyond the algebraic polygonal case, there is essentially one domain for which
we can currently write the tropical zeta function explicitly in terms of known
special functions. This domain arises as the limit shape obtained from
concentration of measure for the uniform distribution on lattice polygons
contained in a fixed square, as the lattice mesh tends to zero
\cite{Vershik1994,Barany1995,Barany2008}. For this particular domain, the
tropical zeta function is
\[
\left(8-2^{2-s}\frac{\zeta_{\mathrm{SU}(3)}(s)}{\zeta(3s)}\right)
\frac{1}{s(s-1)},
\]
where \(\zeta_{\mathrm{SU}(3)}\) denotes Witten's \(\mathrm{SU}(3)\) zeta
function, arising in the context of two-dimensional Yang--Mills theory, and
\(\zeta\) is the Riemann zeta function. The pole structure of
\(\zeta_{\mathrm{SU}(3)}\) was described by Romik \cite{Romik2017}, who
showed that all of its poles are simple and occur at \(s=\frac23\) and at
\(s=\frac12-k\) for \(k=0,1,2,\dots\).

The same special domain also provides a test case for the relation between
tropical zeta functions and lattice-point counting. The tropical zeta function
naturally organizes the Farey triangles by the cubic weight
\[
m_{p,q}=pq(p+q),
\]
and the averaged lattice discrepancy is subtler than the residue at
\(s=\frac23\) alone. In Section~\ref{sec_6} we show that the apparent bulk
contribution cancels after an exact head--tail decomposition of the Farey
cells. The first surviving term comes from a second layer of Farey-edge
effects, and gives the asymptotic
\[
N^{-3/2}\sum_{n\le N}
\left(|nL\cap\Z^2|-\frac{10}{3}n^2\right)
\longrightarrow
\frac{32}{3}\zeta\!\left(-\frac12\right)
\frac{\zeta(2)}{\zeta(3)}.
\]

and, for \(\psi(x)=\{x\}-\frac12\), gives
\begin{equation*}
\lim_{N\to\infty}N^{-3/2}
\sum_{1\le i<j\le N}\psi(2\sqrt{ij})
=
\frac43\zeta\!\left(-\frac12\right)
\left(1-2\frac{\zeta(2)}{\zeta(3)}\right).
\end{equation*}

This lattice-counting theorem also suggests a broader program connected with
the Gauss circle problem. The parabolic domain \(L\) and the Euclidean disk
have the same tropical caustic combinatorics: the same Farey tree of primitive
directions and the same directions of caustic edges occur in both cases. What
changes are the sizes assigned to these edges. For \(L\), these sizes are
governed by the rational cubic weights \(m_{p,q}=pq(p+q)\), leading to the
\(\mathrm{SU}(3)\) zeta function and to the exact averaged discrepancy
computed in Section~\ref{sec_6}. For the disk, one expects the same
combinatorial decomposition to carry quadratic-irrational size data, governed
by square roots, continued fractions, and Diophantine approximation.

This raises a natural intermediate question between the classical Gauss circle
problem and smoothed lattice-point estimates: to understand the averaged circle
discrepancy through the tropical caustic decomposition and determine whether
the apparent bulk term cancels as it does for \(L\). A positive answer would
indicate that the pointwise Gauss error is governed by a second layer of
caustic arithmetic. In this sense, Section~\ref{sec_6} provides an exactly
solvable model for a tropical-geometric approach to one of the classical
problems of lattice-point theory.

We also emphasize that, although the main body of the present article is
devoted to the two-dimensional case, the underlying construction is meaningful
in arbitrary dimensions. One may also hope that the convexity assumption can
eventually be weakened, or perhaps removed entirely. In such a framework, it
is conceivable that tropical zeta functions could be attached to a broader
class of arithmetico-geometric objects and would define a homomorphism from a
suitable motivic ring to the field of meromorphic functions, somewhat in the
spirit of the Hodge--Deligne polynomial in birational geometry
\cite{DeligneHodgeII,DeligneHodgeIII,DanilovKhovanskii1987} and of Igusa
zeta functions in arithmetic geometry \cite{Igusa1974,Igusa2000}. Broader
continuation questions, as well as higher-dimensional analogues of the affine
residue, remain open.

\subsection{Main results}\label{ss:mainresults}

Let \(\Omega\subset \R^2\) be a compact convex domain with nonempty interior,
and let
\[
h_\Omega(u)=\min_{x\in\Omega}\langle u,x\rangle,\qquad u\in\R^2,
\]
be its lower support function. We define the associated tropical
distance-to-the-boundary function by
\[
\rho_\Omega(x)
=
\min_{u\in\Z^2_{\mathrm{prim}}}
\bigl(\langle u,x\rangle-h_\Omega(u)\bigr),
\qquad x\in\Omega^\circ,
\]
where \(\Z^2_{\mathrm{prim}}\) denotes the primitive lattice directions. This
function is concave, positive, and piecewise linear on the interior
\(\Omega^\circ\), and extends continuously by zero to the boundary
\(\partial\Omega\).

\begin{definition}\label{def_zeta}
The tropical zeta function of \(\Omega\) is defined by
\[
Z_\Omega(s)=\int_\Omega \rho_\Omega(x)^{\,s-2}\,dV,
\]
where \(dV\) denotes Lebesgue measure, and \(\Re(s)\) is initially assumed to
be sufficiently large.\footnote{More precisely, the integral converges at least
for \(\Re(s)>1\).}
\end{definition}

Thus one associates to \(\Omega\) an \(\mathrm{SL}(2,\Z)\)-invariant zeta
function built from the canonical tropical distance function.

The function \(Z_\Omega(s)\) admits a natural Mellin interpretation. If
\[
\Omega_t:=\{x\in\Omega:\rho_\Omega(x)\ge t\}
\]
denotes the inward tropical wave front at time \(t\), and if \(P_\Omega(t)\)
denotes its lattice perimeter, then
\[
Z_\Omega(s)=\int_0^{m_\Omega} t^{s-2}P_\Omega(t)\,dt,
\qquad
m_\Omega:=\max_\Omega \rho_\Omega.
\]
Thus the tropical zeta function is the Mellin transform of the
lattice-perimeter profile of the tropical wave fronts. In particular, the
singularities of \(Z_\Omega(s)\) are governed by the short-time asymptotics
of this boundary evolution.

Our first result concerns the polygonal case. If \(\Omega\) is a polygon whose
sides have rational slopes, then \(Z_\Omega(s)\) extends meromorphically to
the whole complex plane, and its rightmost pole occurs at \(s=1\). The
residue at this pole is the lattice length of \(\partial\Omega\). Thus, for
rational polygons, the leading singularity records the lattice-visible part of
the boundary.

Our main result concerns the opposite regime, namely smooth convex domains
with everywhere nonvanishing curvature. Let \(\Omega\subset\R^2\) be a compact
convex domain with \(C^3\)-smooth boundary and everywhere nonvanishing
curvature. Then the pole at \(s=1\) disappears, and the leading singularity
moves to \(s=\tfrac23\). More precisely, we prove that \(Z_\Omega(s)\) admits
a meromorphic continuation to a neighborhood of \(s=\tfrac23\), with a simple
pole there, and
\[
\Res_{s=2/3} Z_\Omega(s)
=
\frac{3^{5/2}}{2^{5/3}\pi^3}\Gamma\!\left(\frac13\right)^3
\Length_{\mathrm{equiaffine}}(\partial\Omega),
\]
where \(\Gamma(\cdot)\) is Euler's gamma function.

Equivalently,
\[
\Res_{s=2/3} Z_\Omega(s)
=
\frac{3^{5/2}}{2^{5/3}\pi^3}\Gamma\!\left(\frac13\right)^3
\int_{\partial\Omega}\kappa^{1/3}\,ds.
\]

\begin{figure}[p]
\centering

\begin{tikzpicture}[
  >=Latex,
  font=\small,
  main/.style={
    draw,
    rounded corners=3pt,
    align=center,
    inner sep=7pt,
    text width=0.78\textwidth
  },
  final/.style={
    draw,
    very thick,
    rounded corners=4pt,
    align=center,
    inner sep=9pt,
    text width=0.80\textwidth
  },
  note/.style={
    draw,
    rounded corners=3pt,
    align=left,
    inner sep=6pt,
    text width=0.36\textwidth
  },
  arrow/.style={->, thick},
  dashedarrow/.style={->, thick, dashed}
]

\node[main] (boundary) {
\textbf{Boundary reduction}\\[2pt]
Theorem~3.4 gives, near \(s=\frac23\),
\[
s(s-1)Z_\Omega(s)
=
-F_{\partial\Omega}(s)+H_{\widehat\Omega}(s),
\]
where \(H_{\widehat\Omega}\) is holomorphic.
};

\node[main, below=9mm of boundary] (arcs) {
\textbf{Arc decomposition}\\[2pt]
The boundary is decomposed into finitely many arcs \(\Gamma\), and
\[
F_{\partial\Omega}(s)=\sum_{\Gamma} F_\Gamma(s).
\]
Thus the residue is computed arc by arc.
};

\node[main, below=9mm of arcs] (hata) {
\textbf{Hata--Farey identification}\\[2pt]
For the Legendre-dual function \(\widetilde g(u)=g^*(-u)\),
\[
\operatorname*{Res}_{s=2/3}F_\Gamma(s)
=
\operatorname*{Res}_{s=2/3}Z_{\widetilde g}(s).
\]
This uses Lemma~4.15 and Corollary~4.16.
};

\node[main, below=9mm of hata] (farey) {
\textbf{Local Farey residue theorem}\\[2pt]
Theorem~4.7 gives
\[
\operatorname*{Res}_{s=2/3}Z_f(s)
=
\frac{\sqrt3\,\Gamma(\frac13)^3}{2^{2/3}\pi^3}
\int_0^1 |f''(u)|^{2/3}\,du .
\]
};

\node[main, below=9mm of farey] (legendre) {
\textbf{Legendre duality and equiaffine length}\\[2pt]
For \(f=\widetilde g\), Lemma~4.17 gives
\[
\int_0^1(\widetilde g''(u))^{2/3}\,du
=
\operatorname{Length}_{\mathrm{equiaffine}}(\Gamma).
\]
};

\node[final, below=10mm of legendre] (finalres) {
\textbf{Residue formula}\\[2pt]
Summing over the arcs of \(\partial\Omega\) yields
\[
\operatorname*{Res}_{s=2/3} Z_\Omega(s)
=
\frac{3^{5/2}}{2^{5/3}\pi^3}
\Gamma\!\left(\frac13\right)^3
\operatorname{Length}_{\mathrm{equiaffine}}(\partial\Omega).
\]
};

\draw[arrow] (boundary) -- (arcs);
\draw[arrow] (arcs) -- (hata);
\draw[arrow] (hata) -- (farey);
\draw[arrow] (farey) -- (legendre);
\draw[arrow] (legendre) -- (finalres);

\node[note, right=8mm of farey] (ingredients) {
\textbf{Analytic inputs to Theorem~4.7}\\[3pt]
\(\bullet\) Lemma~4.10: mean-value formula for \(T_I(f)\).\\[2pt]
\(\bullet\) Lemma~4.11: endpoint replacement.\\[2pt]
\(\bullet\) Lemma~4.12: parametrization by coprime pairs.\\[2pt]
\(\bullet\) Proposition~4.14: the leading term is the same as if the Farey pairs were equidistributed.
};

\draw[dashedarrow] (ingredients.west) -- (farey.east);

\end{tikzpicture}

\caption{Dependency structure for the residue formula at \(s=\frac23\).
The interior tropical zeta function is first reduced to the boundary
Dirichlet series by Theorem~3.4.  The boundary is then decomposed into
arcs, each arc is identified with a Farey--Hata zeta function through
Legendre duality, and Theorem~4.7 computes the local residue.  Lemma~4.17
converts the resulting integral into equiaffine arc length, and summing over
the arcs gives the stated residue of \(Z_\Omega(s)\).}
\label{fig:residue-dependency}
\end{figure}

Thus, in the smooth strictly convex case with everywhere nonvanishing
curvature, the leading singularity of the tropical zeta function is governed
by equiaffine geometry.

The half-plane continuation to \(\Re(s)>\frac35\), combined with a Tauberian
argument, gives the short-time asymptotic of the lattice perimeter
\(P_\Omega(t)\) of the tropical wave front. In particular, as \(t\to0^+\), the
leading term is of order \(t^{1/3}\). In the special parabolic model this
asymptotic can be sharpened: the next term is proportional to \(t^{1/2}\),
and, assuming the Riemann hypothesis, the following term is of order at most
\(t^{5/6}\). We expect this behavior to be universal for sufficiently smooth
convex domains with everywhere nonvanishing curvature.

The analytic extension and the computation of the residue are based on a
reduction from the interior Mellin integral to a boundary Dirichlet series. In
dimension two, one associates to \(\partial\Omega\) a boundary zeta series
\[
F_{\partial\Omega}(s)
=
\sum_{\substack{a,b,c,d\in\Z\\ ad-bc=1}}
f_{\partial\Omega}(a,b,c,d)^s,
\]
where the summands are determined by neighboring primitive directions and the
corresponding support defects. If \(\widehat\Omega\) denotes the minimal model
of \(\Omega\), then one has an exact identity
\[
s(s-1)Z_\Omega(s)
=
-
F_{\partial\Omega}(s)+H_{\widehat\Omega}(s),
\]
where \(H_{\widehat\Omega}(s)\) is an explicitly computable holomorphic
function coming from the symplectic minimal model of \(\Omega\). Thus the
singular behavior of \(Z_\Omega(s)\) near \(s=\frac{2}{3}\) is entirely
governed by the boundary series.

A model case is provided by an arc of a parabola. For the corresponding
parabolic domain, the boundary zeta series can be computed explicitly and
turns out to be the primitive Mordell--Tornheim series, or equivalently,
Witten's \(\mathrm{SU}(3)\) zeta function after the natural normalization. This model
already contains the pole at \(s=\tfrac23\) and explains the appearance of the
same transcendental constant as in the general residue formula.

The parabolic case is the simplest model because the second derivative of the
Legendre dual is constant; consequently, the support defect is governed only
by the universal arithmetic weight attached to a Farey pair, and the boundary
zeta series reduces to a primitive Mordell--Tornheim series. For a general
smooth convex arc, the second derivative varies along the slope parameter, so
one is led to a Farey sum with a nonconstant weight. The main analytic task is
therefore to separate this slowly varying geometric factor from the arithmetic
summation. We do this by replacing the exact defect by an endpoint model and
then studying the resulting weighted Farey series using Fej\'er approximation
together with power-saving estimates for incomplete Kloosterman sums. This
yields the residue as an averaged slope-space density
\[
\bigl((g^\ast)''(\alpha)\bigr)^{2/3}\,d\alpha,
\]
and Legendre duality converts this density into the equiaffine arc length
element
\[
\bigl((g^\ast)''(\alpha)\bigr)^{2/3}\,d\alpha
=
(g''(x))^{1/3}\,dx
=
\kappa^{1/3}\,ds.
\]
Thus the explicit parabolic computation extends to the general case after
decoupling the variable curvature from the Farey arithmetic.

The same special parabolic domain also yields a main result in lattice-point
counting. Let
\[
L=\{(x,y)\in\R^2:\sqrt{1-|x|}+\sqrt{1-|y|}\ge 1\},
\]
and put
\[
N_L(n)=|nL\cap\Z^2|,
\qquad
E_L(n)=N_L(n)-\frac{10}{3}n^2.
\]
In Section~\ref{sec_6} we prove the averaged discrepancy formula
\[
N^{-3/2}\sum_{n\le N}E_L(n)
\longrightarrow
\frac{32}{3}\zeta\!\left(-\frac12\right)\frac{\zeta(2)}{\zeta(3)}.
\]

Among zeta-type constructions attached to convex domains, the present one is
especially well suited to arithmetic questions: it is \(\mathrm{SL}(n,\Z)\)-
invariant, so it respects the natural symmetry of the lattice, and in the
planar smooth convex case with everywhere nonvanishing curvature its leading
nontrivial residue is the equiaffine length of the boundary. This is exactly
the type of affine invariant that is known to control the behavior of rational
and lattice points near convex curves; see
\cite{Petrov2006,HowardTrifonov2022,Howard2023}. In this sense, the tropical
zeta function captures the arithmetic geometry of the boundary in affine
terms.

The paper is organized as follows. Figure~\ref{fig:structureofpaper}
summarizes the main tropical-zeta structure, including the lattice-counting
application in Section~\ref{sec_6}. In Section~\ref{sec1}, we define the
tropical zeta function in integral form and derive its Mellin representation in
terms of tropical wave fronts. In Section~\ref{sec2}, we specialize to
dimension \(2\), introduce the boundary zeta series, and prove the
integral--boundary identity. Routine computations, together with background
material on symplectic minimal models and the residues at \(s=0\) and \(s=1\)
in the polygonal case, are collected in Appendix~\ref{app1}. Essential toric
and symplectic background, as well as the Tauberian passage from the boundary
series to the short-time asymptotic of the tropical wave-front lattice
perimeter, are contained in Appendix~\ref{app2}. In Section~\ref{sec3}, we
analyze the resulting Farey-type Dirichlet series, prove meromorphic
continuation near \(s=\frac23\), and compute the residue; the analytic
estimates, including the half-plane continuation to \(\Re(s)>\frac35\), are
deferred to Appendix~\ref{app3}. Section~\ref{sec4} treats the parabolic model
and identifies its boundary series with the primitive Mordell--Tornheim series,
equivalently Witten's \(\mathrm{SU}(3)\) zeta function. Section~\ref{sec_6}
applies the same Farey-cell structure to the averaged lattice-point
discrepancy of the special parabolic domain \(L\), proving an explicit
\(\sqrt N\)-asymptotic. Appendix~\ref{app:section6} contains the detailed
proof, including the head--tail decomposition, the phase identity, the three
second-layer contributions, and the final algebra. Finally,
Appendix~\ref{sec_survey} surveys several other zeta functions attached to
convex domains and compares them with the tropical zeta function studied here.

\clearpage
\begin{landscape}
\thispagestyle{empty}
\begingroup



\definecolor{residuecolor}{RGB}{181,116,58}
\definecolor{wavecolor}{RGB}{63,126,139}
\definecolor{countcolor}{RGB}{103,99,151}
\definecolor{bridgecolor}{RGB}{42,135,105}
\definecolor{contextcolor}{RGB}{112,128,112}


\tikzset{
  country/.style={
    line width=.9pt,
    rounded corners=10pt
  },
  wave country/.style={
    country,
    draw=wavecolor!80!black,
    fill=wavecolor!10
  },
  residue country/.style={
    country,
    draw=residuecolor!80!black,
    fill=residuecolor!10
  },
  count country/.style={
    country,
    draw=countcolor!80!black,
    fill=countcolor!10
  },
  bridge country/.style={
    country,
    draw=bridgecolor!82!black,
    fill=bridgecolor!12
  },
  context country/.style={
    country,
    draw=contextcolor!72!black,
    fill=contextcolor!8
  },
  country title/.style={
    font=\sffamily\bfseries\large,
    align=center
  },
  box/.style={
    draw=black!22,
    fill=white,
    rounded corners=4pt,
    inner xsep=6pt,
    inner ysep=5pt,
    align=center,
    font=\sffamily\scriptsize,
    text width=3.55cm
  },
  widebox/.style={
    box,
    text width=3.90cm
  },
  sidebox/.style={
    draw=black!14,
    fill=white,
    rounded corners=3pt,
    inner xsep=4pt,
    inner ysep=4pt,
    align=center,
    font=\sffamily\scriptsize,
    text width=1.80cm
  },
  bridgebox/.style={
    draw=bridgecolor!35,
    fill=white,
    rounded corners=4pt,
    inner xsep=7pt,
    inner ysep=6pt,
    align=center,
    font=\sffamily\scriptsize,
    text width=5.00cm
  },
  contextnote/.style={
    align=center,
    font=\sffamily\scriptsize,
    text width=4.30cm
  },
  route/.style={
    -{Stealth[length=2.2mm,width=1.4mm]},
    line width=.85pt,
    draw=black!65
  },
  support/.style={
    -{Stealth[length=1.7mm,width=1.05mm]},
    dashed,
    line width=.62pt,
    draw=black!43
  },
  contextlink/.style={
    -{Stealth[length=1.5mm,width=.95mm]},
    dotted,
    line width=.60pt,
    draw=black!38
  }
}


\begin{figure}[p]
\centering
\vspace*{-12mm}

\makebox[\linewidth][c]{%
\begin{tikzpicture}[
  x=.92cm,
  y=.92cm,
  line join=round,
  line cap=round
]


\path[wave country]
  (0.30,6.50) rectangle (8.15,15.10);

\path[residue country]
  (8.65,6.50) rectangle (18.35,15.10);

\path[count country]
  (18.85,6.50) rectangle (26.70,15.10);

\path[bridge country]
  (5.00,2.55) rectangle (22.00,6.00);

\path[context country]
  (0.30,0.20) rectangle (26.70,2.20);


\node[
  country title,
  text=wavecolor!45!black
] at (4.22,14.50)
{Tropical wave-front asymptotics};

\node[
  country title,
  text=residuecolor!45!black
] at (13.50,14.50)
{Residues of the tropical zeta function};

\node[
  country title,
  text=countcolor!45!black
] at (22.78,14.50)
{Lattice-point counting on \(L\)};

\node[
  draw=bridgecolor!75!black,
  fill=white,
  rounded corners=4pt,
  inner xsep=10pt,
  inner ysep=4pt,
  font=\sffamily\bfseries\normalsize,
  text=bridgecolor!45!black
] at (13.50,5.55)
{Section 5: the parabolic bridge};

\node[
  font=\sffamily\bfseries\normalsize,
  text=contextcolor!40!black
] at (13.50,1.82)
{THE CONTEXT};


\node[widebox] (continuation) at (4.22,12.55) {
  \textbf{Boundary-series continuation}\\[3pt]
  \(F_{\partial\Omega}(s)\) for
  \(\Re(s)>\frac35\)\\
  Appendix D.5
};

\node[widebox] (tauberian) at (4.22,10.15) {
  \textbf{Triangle counting and Tauberian step}\\[3pt]
  \(
    \displaystyle
    N_\Omega^{\mathrm{cut}}(t)
    \sim
    \frac32r_\Omega t^{-2/3}
  \)\\
  Appendices B.4--B.5
};

\node[widebox] (wavefinal) at (4.22,7.75) {
  \textbf{Wave-front perimeter}\\[3pt]
  \(
    \displaystyle
    \Length_{\Z}(\partial\Omega_t)
    =
    \Res_{s=2/3}Z_\Omega(s)t^{1/3}
    +
    o(t^{1/3})
  \)
};

\draw[route] (continuation) -- (tauberian);
\draw[route] (tauberian) -- (wavefinal);


\node[widebox] (mellin) at (13.50,12.65) {
  \textbf{Tropical distance and Mellin form}\\
  Sections 1.2 and 2
};

\node[widebox] (reduction) at (13.50,10.25) {
  \textbf{Boundary and Farey--Hata reduction}\\
  Sections 3 and 4\\[3pt]
  \(
    s(s-1)Z_\Omega
    =
    -F_{\partial\Omega}+H_{\widehat\Omega}
  \)
};

\node[widebox] (residue) at (13.50,7.85) {
  \textbf{Smooth residue at \(s=\frac23\)}\\[3pt]
  \(
    \displaystyle
    \Res_{s=2/3}Z_\Omega
    =
    C_{\mathrm{aff}}\,
    \Length_{\mathrm{equiaffine}}(\partial\Omega)
  \)
};

\draw[route] (mellin) -- (reduction);
\draw[route] (reduction) -- (residue);

\node[sidebox] (polygonal) at (9.92,9.25) {
  \textbf{Polygonal poles}\\
  \(s=1\): perimeter\\
  \(s=0\): \(-K^2\)\\
  Appendix A
};

\node[sidebox] (appendixD) at (17.08,9.25) {
  \textbf{Appendix D}\\
  Fej\'er approximation\\
  Kloosterman bounds\\
  \(\Re(s)>\frac35\)
};

\draw[support]
  (reduction.west)
  to[out=190,in=15]
  (polygonal.east);

\draw[support]
  (appendixD.west)
  to[out=165,in=-10]
  (reduction.east);

\draw[route]
  (residue.west)
  to[out=185,in=-5]
  (continuation.east);


\node[bridgebox] (parabolic) at (9.70,3.85) {
  \textbf{5.1--5.4: Explicit parabolic model}\\[4pt]
  The domain \(L\), the weights
  \(m_{p,q}=pq(p+q)\), and the
  Mordell--Tornheim/\(\mathrm{SU}(3)\) series
};

\node[bridgebox] (broader) at (17.30,3.85) {
  \textbf{5.5--5.6: Broader appearances}\\[4pt]
  Moduli-space volume and
  concentration of measure
};

\draw[route]
  (residue.south)
  to[out=-90,in=90]
  (parabolic.north);

\draw[contextlink]
  (parabolic.east)
  --
  (broader.west);


\node[widebox] (scale) at (22.00,12.55) {
  \textbf{6.1: Tropical-zeta scale}\\[3pt]
  \(m_{p,q}=pq(p+q)\)
};

\node[widebox] (headtail) at (22.00,10.10) {
  \textbf{6.2--6.3: Head--tail mechanism}\\[3pt]
  \(
    E_L
    =
    (H_0+T_1)+H_{12}+R_{T_0}+1
  \)\\
  Three second-layer packages
};

\node[widebox] (countfinal) at (22.00,7.70) {
  \textbf{Averaged discrepancy}\\[3pt]
  \(
    \displaystyle
    N^{-3/2}\sum_{n\le N}E_L(n)
    \longrightarrow
    \frac{32}{3}
    \zeta\!\left(-\frac12\right)
    \frac{\zeta(2)}{\zeta(3)}
  \)
};

\draw[route]
  (parabolic.east)
  to[out=10,in=-135]
  (scale.south west);

\draw[route] (scale) -- (headtail);
\draw[route] (headtail) -- (countfinal);

\node[sidebox] (future) at (25.55,10.10) {
  \textbf{6.4--6.5}\\
  column comparison\\
  sawtooth conjectures\\
  disk program\\
  QFT analogy
};

\draw[contextlink]
  (countfinal.east)
  to[out=5,in=-105]
  (future.south);


\node[contextnote] at (3.70,.82) {
  \textbf{Introduction 1.1}\\
  Gauss circle problem,
  sandpiles, tropical optics
};

\node[contextnote] at (10.20,.82) {
  \textbf{Appendix A}\\
  minimal models,
  corner cuts, polygonal residues
};

\node[contextnote] at (16.80,.82) {
  \textbf{Appendix B.1--B.3}\\
  toric and symplectic geometry,
  wave fronts, caustics
};

\node[contextnote] at (23.30,.82) {
  \textbf{Appendix C}\\
  gauge, distance, tube,
  and spectral zeta functions
};

\end{tikzpicture}%
}

\caption[Topographic structure of the paper]{
Topographic structure of the paper. The central country develops the tropical
zeta function, its boundary Farey--Hata reduction, and the residue theorem.
The left and right countries lead respectively to tropical wave-front
asymptotics and to lattice-point counting on \(L\). Section~5 forms the
parabolic bridge between the general theory and the explicit arithmetic
model, while the lower country records the geometric and zeta-theoretic
context.
}
\label{fig:structureofpaper}

\end{figure}
\endgroup
\end{landscape}
\clearpage

\section{The tropical zeta function in integral form}\label{sec1}

In this section, we define the zeta function of a convex body by an integral
over the body itself. In dimension \(2\), this integral presentation is
equivalent, up to an explicit holomorphic correction term, to a boundary
Dirichlet series built from primitive lattice directions and Farey neighbors.

\subsection{The tropical distance function}

Let \(\Omega\subset \R^n\) be a compact convex body with nonempty interior.
For
\[
u\in \R^n
\]
let
\[
h_\Omega(u):=\min_{x\in \Omega}\langle u,x\rangle
\]
be the lower support function of \(\Omega\). Denote by
\[
\Z^n_{\mathrm{prim}}
:=
\{u\in \Z^n\setminus\{0\}:\gcd(u_1,\dots,u_n)=1\}
\]
the set of primitive lattice directions.

Since
\[
\langle mu,x\rangle-h_\Omega(mu)
=
m\bigl(\langle u,x\rangle-h_\Omega(u)\bigr)
\qquad (m\in \N),
\]
it suffices to minimize over primitive lattice vectors.

\begin{definition}
For \(x\in \Omega^\circ\), define the \emph{tropical distance function} of
\(x\) by
\[
\rho_\Omega(x)
:=
\min_{u\in \Z^n_{\mathrm{prim}}}
\bigl(\langle u,x\rangle-h_\Omega(u)\bigr).
\]
\end{definition}

Indeed, if \(x\in \Omega^\circ\), then
\[
\langle u,x\rangle-h_\Omega(u) \ge c_x \|u\|
\]
for some \(c_x>0\), so the minimum is attained.

The function \(\rho_\Omega\) is continuous on \(\Omega^\circ\), positive in
the interior, and extends continuously by zero to \(\partial\Omega\). It may
be viewed as an \(\mathrm{SL}(n,\Z)\)-invariant analogue of a
distance-to-the-boundary function, built from primitive lattice support
functions rather than Euclidean normals.

\begin{definition}
The \emph{tropical zeta function, in integral form,} of \(\Omega\) is
\[
Z_\Omega(s):=\int_\Omega \rho_\Omega(x)^{\,s-n}\,dV,
\]
initially defined for \(\Re(s)\) sufficiently large.
\end{definition}

The function \(Z_\Omega\) inherits \(\mathrm{SL}(n,\Z)\)-invariance and
translation invariance from the affine-lattice definition of
\(\rho_\Omega\); that is,
\[
Z_{A(\Omega)+v}(s)=Z_{\Omega}(s)
\]
for all \(A\in\mathrm{SL}(n,\Z)\) and \(v\in\R^n\). As a conceptual remark,
we emphasize that the tropical zeta function by itself is not sufficient to
yield an exact formula for the number of lattice points in the domain, since
this number varies under continuous families of translations of the domain.

The normalization \(s-n\) is chosen so that
\[
Z_\Omega(n)=\operatorname{Vol}(\Omega).
\]
More generally, this normalization ensures that \(Z_\Omega\) is homogeneous
with respect to dilations of \(\Omega\) in the sense that
\[
Z_{r\Omega}(s)=r^sZ_\Omega(s)
\]
for every \(r>0\). In particular, a residue of its analytic continuation at a
real point \(d\) is a homogeneous invariant of \(\Omega\) of degree \(d\).

\subsection{Mellin form}

Let
\[
m_\Omega:=\max_{x\in \Omega}\rho_\Omega(x)
\]
and, for \(t\in[0,m_\Omega]\), define
\[
\Omega_t:=\{x\in \Omega:\rho_\Omega(x)\ge t\}.
\]
Then for \(\Re(s)>n\), the layer-cake formula gives
\[
Z_\Omega(s)
=
(s-n)\int_0^{m_\Omega} t^{\,s-n-1}\operatorname{Vol}(\Omega_t)\,dt.
\]
Thus \(Z_\Omega(s)\) is the Mellin transform of the volume profile of the
family \(\{\Omega_t\}_{t\ge0}\) given by the tropical wave-front propagation.

For every \(0<t<m_\Omega\), the set \(\Omega_t\) is a convex polytope with
finitely many facets, each of whose supporting hyperplanes has a primitive
lattice normal. The finiteness follows because, on \(\Omega_t\), only
primitive support functions belonging to a bounded set of lattice directions
can become active.

\begin{definition}\label{def:surfacevolume}
Let \(F\) be an \((n-1)\)-dimensional polytope contained in an affine
hyperplane
\[
H=\{x\in\R^n:\langle u,x\rangle=\lambda\},
\]
where \(u\in\Z^n\) is primitive and \(\lambda\in\R\). Let
\[
\Lambda_u:=u^\perp\cap\Z^n
\]
be the lattice parallel to \(H\). The \emph{lattice volume} of \(F\) is
defined by
\[
\operatorname{Vol}_{\Z}(F)
:=
\frac{\operatorname{Vol}_{n-1}(F)}
{\operatorname{covol}(\Lambda_u)},
\]
where \(\operatorname{Vol}_{n-1}\) denotes Euclidean \((n-1)\)-dimensional volume in \(H\).
\end{definition}

\begin{definition}\label{def:wfsurfacevolume}
For \(t>0\), the \emph{lattice surface volume} of \(\Omega_t\) is the sum of
the lattice volumes of all facets of \(\Omega_t\). We denote it by
\[
P_\Omega(t).
\]
\end{definition}

Since \(\operatorname{covol}(\Lambda_u)\) is equal to the length of the primitive integer
normal vector \(u\), Proposition~\ref{prop:twf_surfacevol} gives, for
\(t>0\),
\[
P_\Omega(t)=-\frac{d}{dt}\operatorname{Vol}(\Omega_t).
\]

\begin{proposition}[Lattice surface volume form of the Mellin transform]
\label{prop:perimeter-mellin}
For \(\Re(s)>n\),
\[
Z_\Omega(s)=\int_0^{m_\Omega} t^{s-n}P_\Omega(t)\,dt.
\]
\end{proposition}

\begin{proof}
Starting from
\[
Z_\Omega(s)
=
(s-n)\int_0^{m_\Omega} t^{s-n-1}\operatorname{Vol}(\Omega_t)\,dt,
\]
integrate by parts:
\[
(s-n)\int_0^{m_\Omega} t^{s-n-1}\operatorname{Vol}(\Omega_t)\,dt
=
\Bigl[t^{s-n}\operatorname{Vol}(\Omega_t)\Bigr]_0^{m_\Omega}
-
\int_0^{m_\Omega} t^{s-n}
\frac{d}{dt}\operatorname{Vol}(\Omega_t)\,dt.
\]
Since \(\operatorname{Vol}(\Omega_{m_\Omega})=0\) and
\[
t^{s-n}\operatorname{Vol}(\Omega_t)\to 0
\qquad\text{as }t\to0^+
\]
for \(\Re(s)>n\), this gives
\[
Z_\Omega(s)
=
\int_0^{m_\Omega} t^{s-n}
\left(-\frac{d}{dt}\operatorname{Vol}(\Omega_t)\right)\,dt.
\]
By Proposition~\ref{prop:twf_surfacevol}, the expression in parentheses is
\(P_\Omega(t)\). This proves the formula.
\end{proof}

In particular, for \(n=2\), \(Z_\Omega(s)\) is the Mellin transform of the
lattice-perimeter profile of the tropical wave fronts \(\Omega_t\).

\begin{remark}
It is natural to ask whether \(Z_\Omega(s)\) satisfies a functional equation,
or more generally whether its Mellin transform structure reflects a duality on
tropical wave fronts. A possible answer may involve a correspondence between
several domains, together with character weights for \(\mathrm{SL}(n,\Z)\) in
a Dirichlet-series formulation. We do not pursue this question here.
\end{remark}

\section{The planar boundary zeta series}\label{sec2}

The following definition makes sense for \(\Omega\) in arbitrary dimension.

\begin{definition}[Minimal model]\label{def:minimalmodel}
For \(\Omega\subset\R^n\), let
\[
m_\Omega:=\max_{x\in \Omega}\rho_\Omega(x),
\qquad
M_\Omega:=\{x\in\Omega:\rho_\Omega(x)=m_\Omega\}.
\]
Let
\[
\mathcal E_\Omega
:=
\left\{
u\in\Z^n_{\mathrm{prim}}:
\exists x\in M_\Omega
\text{ with }
m_\Omega=\langle u,x\rangle-h_\Omega(u)
\right\},
\]
where
\[
h_\Omega(u):=\min_{x\in \Omega}\langle u,x\rangle.
\]
The \emph{minimal model} of \(\Omega\) is
\[
\widehat\Omega
=
\bigcap_{u\in\mathcal E_\Omega}
\{x\in\R^n:\langle u,x\rangle\ge h_\Omega(u)\}.
\]
\end{definition}

Since \(\mathcal E_\Omega\) is a finite set, \(\widehat\Omega\) is a polytope
with rational-slope facets. It is characterized by the properties that
\[
m_\Omega=m_{\widehat\Omega},
\qquad
M_\Omega=M_{\widehat\Omega},
\]
and \(\rho_{\widehat\Omega}\) coincides with \(\rho_\Omega\) in a neighborhood
of \(M_\Omega\). Clearly, \(\Omega\subset\widehat\Omega\).

The terminology ``minimal model'' (or, more accurately, ``symplectic minimal
model,'' used for the first time in \cite{Shkolnikov17Theis}, with the word
``symplectic'' usually suppressed in the present article for brevity) comes
from the analogy with the minimal model program in algebraic geometry,
transplanted into the symplectic-geometric setting, where blow-ups come with
sizes; see Subsection~\ref{ssec:sympltor}. Taking sums, or sums of squares, of
these sizes gives formulas for the differences between the lattice perimeters
or areas of convex domains and those of their minimal models; see
\cite{kalinin2017number,kalinin2019tropical}, where further formulas, for
example for the Euclidean perimeter, are also described. Thus, the main result
of the present paper extends this framework\footnote{In principle, one may ask
whether other geometric invariants of convex domains can be expressed in terms
of symplectic minimal models and sizes of blow-ups.}, allowing one to express
the equiaffine perimeter in symplectic/tropical terms.

From now on, we specialize to the two-dimensional case.

\begin{definition}
A \(\Gamma\)-triangle is a planar region bounded by an embedded planar arc
\(\Gamma\) and two straight segments of rational slope. The primitive direction
vectors of the straight segments form a basis of the lattice, and \(\Gamma\),
together with the segment joining its endpoints, bounds a convex region.
\end{definition}

From the discussion in Appendix~\ref{ssec:tropop}, it follows that the
complement of the interior of \(\Omega\) in \(\widehat\Omega\) is a finite
union of \(\Gamma\)-triangles. In fact, there are at most six of them; see
Appendix~\ref{ss:tropmmzeta}, where all types of minimal models are
enumerated.

The basis condition, also known as unimodularity, implies that, after
composing an \(\mathrm{SL}(2,\Z)\) transformation with a translation, one may
place the \(\Gamma\)-triangle in the first quadrant so that its straight sides
are segments of the coordinate axes. In particular, \(\Gamma\) is then
represented by the graph of a convex function
\[
g:[0,r]\to[0,u],
\qquad
g(0)=u,\quad g(r)=0.
\]

For each primitive vector
\[
(a,b)\in \Z_{\ge0}^2\setminus\{(0,0)\},
\]
let
\[
ax+by=\gamma_{a,b}
\]
be the supporting line to the graph of \(g\) with normal \((a,b)\). If
\[
a,b,c,d\in \Z_{\ge0},
\qquad
ad-bc=1,
\]
define
\[
f_\Gamma(a,b,c,d)
:=
\bigl|\gamma_{a,b}+\gamma_{c,d}-\gamma_{a+c,b+d}\bigr|.
\]

\begin{definition}
The \emph{boundary zeta series} attached to \(\Gamma\) is
\[
F_\Gamma(s)
:=
\sum_{\substack{a,b,c,d\in\Z_{\ge0}\\ ad-bc=1}}
f_\Gamma(a,b,c,d)^s.
\]
For the whole boundary, after normalizing each \(\Gamma\)-triangle separately
as above, we set
\[
F_{\partial\Omega}(s):=\sum_{\Gamma}F_\Gamma(s),
\]
where the sum runs over the finitely many \(\Gamma\)-triangles, namely
\(3\), \(4\), \(5\), or \(6\) of them, forming
\(\widehat\Omega\setminus\Omega^\circ\).
\end{definition}

The integral form of the tropical zeta function and the boundary Dirichlet
series are related by an exact identity.

\begin{theorem}\label{thm:integral-boundary}
Let \(\Omega\subset \R^2\) be a compact convex domain, and let
\(\widehat\Omega\supset\Omega\) be its minimal model. Then for \(\Re(s)>2\),
\[
s(s-1)Z_\Omega(s)
=
-
F_{\partial\Omega}(s)+H_{\widehat\Omega}(s),
\]
where
\[
H_{\widehat\Omega}(s):=s(s-1)Z_{\widehat\Omega}(s)
\]
is a holomorphic function of the form
\[
m_{\widehat\Omega}^{s-1}
\bigl(2\ell_{\widehat\Omega}s
+k_{\widehat\Omega}m_{\widehat\Omega}\bigr),
\]
where the real numbers \(m_{\widehat\Omega}\) and
\(\ell_{\widehat\Omega}\), and the integer \(k_{\widehat\Omega}\), explicitly
depend on the minimal model \(\widehat\Omega\) of \(\Omega\).
\end{theorem}

See the proof on page~\pageref{proof:thm:integral-boundary}; all planar
minimal models are enumerated in Appendix~\ref{ss:tropmmzeta}. More precisely,
by Proposition~\ref{prop:minmodzeta}, \(m_{\widehat\Omega}=m_{\Omega}\) is the
maximal value of \(\rho_\Omega\), \(\ell_{\widehat\Omega}\) is the lattice
length of the locus
\[
M_\Omega=M_{\widehat\Omega}
\]
where this maximum is attained, and \(m_{\widehat\Omega}k_{\widehat\Omega}\)
is the lattice perimeter of \(\widehat\Omega\) minus
\(2\ell_{\widehat\Omega}\), in elementary terms.\footnote{In
algebro-geometric terms, \(k_{\widehat\Omega}\) is the self-intersection of
the canonical class on the toric surface defined by the fan dual to the
polygon \(\widehat\Omega\), and, from the symplectic-geometric perspective,
the lattice perimeter of \(\widehat\Omega\) is the symplectic area of the
anticanonical class on the symplectic toric surface with moment polygon
\(\widehat\Omega\). From the tropical-geometric point of view, the meaning of
\(2\ell_{\widehat\Omega}\) is that of the limit of the modulus, namely the
length of the unique cycle, of the tropical elliptic curve defined by
\(\min(\rho_{\widehat\Omega},m_{\widehat\Omega}-\varepsilon)\) as
\(\varepsilon\to0^+\).}

\begin{remark}
In the special case when \(\widehat\Omega=[0,P]\times[0,Q]\) with
\(P\ge Q>0\), one has four arcs and
\[
H_{\widehat\Omega}(s)
=
8\Bigl(\frac Q2\Bigr)^s
+
2s(P-Q)\Bigl(\frac Q2\Bigr)^{s-1}.
\]
\end{remark}

\subsection{The first singularity and equiaffine perimeter}

We now state the main boundary residue formula.

\begin{theorem}\label{thm:boundary-residue}
Let \(\Gamma\) be a \(C^3\)-smooth convex planar arc with everywhere nonvanishing
curvature. Then \(F_\Gamma(s)\) admits a meromorphic continuation to a
neighborhood of \(s=\frac23\), with a simple pole at \(s=\frac23\), and
\[
\Res_{s=2/3}F_\Gamma(s)
=
\frac{\sqrt3}{4^{1/3}\pi^3}\Gamma\!\left(\frac13\right)^3
\Length_{\mathrm{equiaffine}}(\Gamma).
\]
\end{theorem}

See the proof on page~\pageref{proof:thm:boundary-residue}.

Since \(\partial\Omega\) is decomposed into three to six such arcs, according
to the minimal model, summing Theorem~\ref{thm:boundary-residue} over these
pieces gives
\[
\Res_{s=2/3}F_{\partial\Omega}(s)
=
\frac{\sqrt3}{4^{1/3}\pi^3}\Gamma\!\left(\frac13\right)^3
\Length_{\mathrm{equiaffine}}(\partial\Omega).
\]

Combining this with the identity of Theorem~\ref{thm:integral-boundary},
which for \(\Re(s)>2\) expresses \(Z_\Omega(s)\) as
\[
-\frac{F_{\partial\Omega}(s)}{s(s-1)}
\]
plus a holomorphic term, we obtain a corresponding residue formula for the
interior zeta function \(Z_\Omega(s)\). In particular, \(Z_\Omega(s)\) admits
a meromorphic continuation to a neighborhood of \(s=\frac23\), and
\[
\Res_{s=2/3}Z_\Omega(s)
=
\frac92\,\Res_{s=2/3}F_{\partial\Omega}(s).
\]
Therefore
\[
\Res_{s=2/3}Z_\Omega(s)
=
\frac{9\sqrt3}{2\cdot 4^{1/3}\pi^3}
\Gamma\!\left(\frac13\right)^3
\Length_{\mathrm{equiaffine}}(\partial\Omega).
\]

Thus the first singularity of \(Z_\Omega(s)\), for \(\Omega\) with \(C^3\)-smooth
boundary and everywhere nonvanishing curvature, is governed by the equiaffine
perimeter of \(\partial\Omega\).

\begin{remark}
We conjecture that, in the \(n\)-dimensional case, the first pole of the
tropical zeta function of a sufficiently smooth convex domain with everywhere
nonvanishing curvature occurs at
\[
s=\frac{n(n-1)}{n+2},
\]
and that the corresponding residue is proportional to the equiaffine surface
volume of \(\partial\Omega\). We are currently unable to prove this conjecture.
In particular, the analytic continuation of the zeta function to the left of
\(s=n-1\), where the residue is the boundary lattice volume of \(\Omega\),
remains open.
\end{remark}

We now recall the basic facts about equiaffine arc length used in the statement
of the theorem.

\subsection{Equiaffine arc length}

Let \(\Gamma:[0,1]\to\R^2\) be a \(C^2\) regular curve,
\[
\Gamma(t)=(x(t),y(t)),
\qquad
\dot\Gamma(t)\neq 0.
\]
Its Euclidean curvature is
\[
\kappa(t)=
\frac{|\det(\dot\Gamma(t),\ddot\Gamma(t))|}
{\|\dot\Gamma(t)\|^3},
\]
and its equiaffine arc length element is
\[
ds_{\mathrm{equiaffine}}
=
|\det(\dot\Gamma(t),\ddot\Gamma(t))|^{1/3}\,dt.
\]
Accordingly,
\[
\Length_{\mathrm{equiaffine}}(\Gamma)
=
\int_\Gamma \kappa^{1/3}\,ds
=
\int_0^1 |\det(\dot\Gamma(t),\ddot\Gamma(t))|^{1/3}\,dt.
\]

If \(\Gamma\) is written locally as the graph of a convex \(C^2\)-function
\(y=g(x)\), then
\[
\Length_{\mathrm{equiaffine}}(\Gamma)
=
\int (g''(x))^{1/3}\,dx.
\]

There is also a geometric interpretation in terms of support triangles: if
\(\Gamma\) is partitioned into sufficiently small arcs and \(\Delta_i\)
denotes the support triangle corresponding to the \(i\)-th arc, then
\cite{Blaschke1923,Ludwig1999}
\begin{equation}\label{23affine}
\Length_{\mathrm{equiaffine}}(\Gamma)
=
\lim_{\max_i |t_i-t_{i-1}|\to 0}
\sum_i 2 \bigl(\Area(\Delta_i)\bigr)^{1/3}.
\end{equation}
The formula is local. Equiaffine length is additive under subdivision, so one
may break the arc into very small pieces and ask what a single infinitesimal
piece contributes. After an affine change of coordinates, such a small convex
arc can be written as the graph of a strictly convex \(C^2\) function
\(y=g(x)\), and in these coordinates the equiaffine length is measured by the
integral \(\int (g''(x))^{1/3}\,dx\). The point is that, at sufficiently small
scale, the curve is governed by its osculating parabola, which in affine
geometry plays the role that the osculating circle plays in Euclidean
geometry. For a parabola, one finds directly that the support triangle
determined by the endpoint tangents over an interval of width \(h\) has area
\[
\frac18\,g''(\xi)h^3+o(h^3)
\]
for some intermediate point \(\xi\). Thus its contribution to the sum is
\[
2\,\Area(\Delta)^{1/3}=(g''(\xi))^{1/3}h+o(h),
\]
which is precisely the infinitesimal equiaffine arc length element. Summing
over all subarcs therefore produces a Riemann sum for
\(\int (g'')^{1/3}\,dx\), and letting the mesh tend to zero gives the claimed
limit. In this way, the support triangles provide a polygonal model for affine
length: each small triangle records the local cubic bending of the curve, and
the whole sum recovers the equiaffine measure of the arc.

\section{Farey intervals and the zeta function}\label{sec3}

\begin{definition}[Farey interval]
A \emph{Farey interval} is an interval
\[
I=\left[\frac cd,\,\frac ab\right]\subset[0,1]
\]
with \(b,d\in\N\), \(a,c\in\Z\) satisfying
\[
0\le c\le d,\qquad 0\le a\le b,\qquad
\gcd(a,b)=\gcd(c,d)=1,
\]
and
\[
ad-bc=1.
\]
We write its mediant as
\[
\mu_I:=\frac{a+c}{b+d}.
\]
\end{definition}

\begin{definition}[Hata coefficient and its renormalization]
Let \(f\in C^2([0,1])\). For
\[
I=\left[\frac cd,\frac ab\right]
\]
define
\begin{equation}\label{coeff_hata}
c_I(f)
:=
f\!\left(\frac{a+c}{b+d}\right)
-\frac{b}{b+d}f\!\left(\frac ab\right)
-\frac{d}{b+d}f\!\left(\frac cd\right).
\end{equation}
Define
\[
T_I(f):=(b+d)c_I(f).
\]
\end{definition}

\begin{definition}[Farey zeta function]\label{def:Fareyzeta}
For \(s\in\C\) with \(\Re(s)\) large, define
\[
Z_f(s):=\sum_I |T_I(f)|^s,
\]
where the sum runs over all Farey intervals \(I\subset[0,1]\).
\end{definition}

\begin{remark}[On residues]
\label{rem:residue}
Throughout, \(\Res_{s=2/3}\) denotes the residue of a meromorphic
continuation of the relevant function to a punctured neighborhood of
\(s=\frac23\). Accordingly, whenever we compute a residue, we also prove
that, after subtracting the explicit polar part, the remainder is holomorphic
on some disk
\[
|s-\tfrac23|<\eta.
\]
\end{remark}

\begin{remark}
Removing or adding finitely many intervals, for instance those adjacent to
\(0/1\) or \(1/1\), changes \(Z_f(s)\) by an entire function. Hence it does
not affect the residue at \(s=\frac23\).
\end{remark}

\begin{remark}[The convergence threshold]
By Theorem~5 of \cite{kalinin2024legendre}, if
\(f\in C^3([0,1])\) and \(|f''|>c>0\), then the Dirichlet-type series in
Definition~\ref{def:Fareyzeta} converges for \(s>\frac23\) and diverges at
\(s=\frac23\). The mechanism behind this threshold is simple. The normalized
Hata coefficient \(T_I(f)=(b+d)c_I(f)\) is a second-order defect attached to
the Farey interval
\[
I=\left[\frac cd,\frac ab\right].
\]
As the mean-value formula below makes precise, one has
\[
T_I(f)
=
-\frac{f''(\xi_I)}{2bd(b+d)}
\]
for some \(\xi_I\in(c/d,a/b)\). Thus, when \(|f''|\) is bounded above and
below, the summands of \(Z_f(s)\) are comparable to
\[
\frac{1}{(bd(b+d))^s}.
\]
The convergence problem is therefore governed by the primitive
Mordell--Tornheim-type series
\[
\sum_{\substack{b,d\ge1\\(b,d)=1}}
\frac{1}{(bd(b+d))^s}.
\]
A dyadic block with \(b\asymp d\asymp R\) contains \(\asymp R^2\) coprime
pairs, and each summand has size \(\asymp R^{-3s}\). Hence the diagonal
blocks have size \(\asymp R^{2-3s}\), giving the critical exponent
\(s=\frac23\). The off-diagonal blocks are controlled by the same
Mordell--Tornheim estimate and do not change the threshold. In the parabolic
case the comparison is exact, while in the general \(C^3\) case Legendre
duality and the bounds on \(f''\) give uniform comparison with this arithmetic
model.
\end{remark}

Assume from now on that
\begin{equation}
\label{eq:ass}
f\in C^3([0,1])
\qquad\text{and}\qquad
0<m\le |f''(x)|
\quad\text{for all }x\in[0,1].
\end{equation}

\begin{theorem}[Local residue formula]
\label{thm:res}
Assume \eqref{eq:ass}. Then \(Z_f(s)\) admits a meromorphic continuation to
a neighborhood of \(s=\frac23\); more precisely, for every
\(0<\eta<1/21\) it extends meromorphically to the disk
\[
|s-\tfrac23|<\eta,
\]
with a simple pole at \(s=\frac23\). Moreover,
\[
\Res_{s=2/3} Z_f(s)
=
\frac{\sqrt3\,\Gamma\!\left(\frac13\right)^3}{2^{2/3}\pi^3}
\int_0^1 |f''(v)|^{2/3}\,dv.
\]
\end{theorem}

A stronger continuation statement is proved later using this result in
Appendix~\ref{app3}, Subsection~\ref{ssec:striprefinement}: under the same
assumptions, \(Z_f(s)\) extends meromorphically to the half-plane
\[
\Re(s)>\frac35,
\]
and is holomorphic there except for the same simple pole at \(s=\frac23\).

See the proof on page~\pageref{proof:thm:res}. We first reduce the geometric
quantity \(T_I(f)\) to the local curvature \(f''\) on the interval \(I\). This
leads to an endpoint model \(Z_f^{\mathrm{end}}(s)\), which has the same
residue at \(s=\frac23\). We then reparametrize Farey intervals by coprime
denominator pairs, reorganize the series by the first denominator \(b\), and
show that the inner sums admit an equidistributed asymptotic with a
power-saving error. The pole then comes entirely from the classical Dirichlet
series
\[
\sum_{b\ge1}\varphi(b)b^{-3s}.
\]

Hata associates to every Farey interval \(I\) the piecewise-linear function
\cite{Hata1995Farey}
\[
S_I(x)
:=
\frac{b+d}{2}
\Bigl(
|a-bx|+|c-dx|-|a+c-(b+d)x|
\Bigr).
\]

\begin{lemma}
For every Farey interval
\[
I=\left[\frac cd,\frac ab\right],
\]
the function \(S_I\) satisfies
\[
S_I(x)\ge 0,
\qquad
\operatorname{supp}(S_I)=I,
\qquad
S_I\!\left(\frac cd\right)=S_I\!\left(\frac ab\right)=0,
\qquad
S_I(\mu_I)=1.
\]
Moreover, \(S_I\) is piecewise linear and attains its unique maximum at the
mediant \(\mu_I\).
\end{lemma}

\begin{theorem}[\cite{Hata1995Farey}]
Every continuous function \(f\in C[0,1]\) admits the expansion
\[
f(x)=f(0)+(f(1)-f(0))x+\sum_I c_I(f)\,S_I(x),
\]
where the sum is taken over all Farey intervals in Stern--Brocot order and
converges uniformly on \([0,1]\). The coefficient of \(S_I\) is given by
\eqref{coeff_hata}.
\end{theorem}

For our purposes, the importance of Hata's coefficient is that, after the
normalization
\[
T_I(f)=(b+d)c_I(f),
\]
it behaves like a second-order local quantity attached to the interval \(I\).
Thus the zeta function \(Z_f(s)\) can be expected to be controlled by the
local curvature \(f''\), with a geometric weight depending only on the Farey
denominators \(b,d,b+d\). The next lemma makes this heuristic exact.

\begin{lemma}[Mean-value formula for \(T_I(f)\)]
\label{lemma:exactT}
Let \(I=[c/d,a/b]\) be a Farey interval and let \(f\in C^2([0,1])\). Then
there exists \(\xi_I\in(c/d,a/b)\) such that
\[
T_I(f)=-\frac{f''(\xi_I)}{2bd(b+d)}.
\]
\end{lemma}

The proof is given on page~\pageref{proof:lemma:exactT}.

Thus \(c_I(f)\) is a weighted discrete second derivative of \(f\) along the
Farey interval \(I\). Lemma~\ref{lemma:exactT} shows that each summand of
\(Z_f(s)\) is governed by the value of \(f''\) at some interior point
\(\xi_I\), multiplied by the universal arithmetic weight
\[
(bd(b+d))^{-1}.
\]
To isolate the arithmetic structure, we now replace the interior point
\(\xi_I\) by one endpoint. This produces a simpler model series
\(Z_f^{\mathrm{end}}(s)\), and the next lemma shows that this modification
does not change the residue at \(s=\frac23\).

Define
\[
Z_f^{\mathrm{end}}(s)
:=
2^{-s}
\sum_{I=[c/d,a/b]}
\frac{|f''(a/b)|^s}{(bd(b+d))^s}.
\]
The next lemma shows that replacing \(\xi_I\) by this endpoint does not
change the residue at \(s=\frac23\).

\begin{lemma}[Endpoint replacement]
\label{lem:replace}
There exists \(\eta_0>0\) such that
\[
Z_f(s)-Z_f^{\mathrm{end}}(s)
\]
is holomorphic on the disk
\[
|s-\tfrac23|<\eta_0.
\]
In particular,
\[
\Res_{s=2/3}Z_f(s)=\Res_{s=2/3}Z_f^{\mathrm{end}}(s).
\]
\end{lemma}

The proof is given on page~\pageref{proof:lem:replace}.

\begin{lemma}[Parametrization by modular inverses]
\label{lem:param}
Farey intervals
\[
I=\left[\frac cd,\frac ab\right]
\]
with \(b,d\ge1\) are in bijection with pairs
\[
(b,d)\in\N^2,
\qquad
\gcd(b,d)=1,
\]
via
\[
a\equiv d^{-1}\pmod b,
\qquad
1\le a\le b,
\qquad
c=\frac{ad-1}{b}.
\]
\end{lemma}

\begin{proof}
If \(ad-bc=1\), then \(ad\equiv1\pmod b\), so \(\gcd(b,d)=1\) and
\(a\equiv d^{-1}\pmod b\). Conversely, for coprime \(b,d\), choose \(a\)
with
\[
a\equiv d^{-1}\pmod b,
\qquad
1\le a\le b,
\]
and define
\[
c=\frac{ad-1}{b}\in\Z.
\]
Then \(ad-bc=1\).
\end{proof}

As we explain below, after writing \(d=kb+r\), the dependence on \(k\)
becomes one-dimensional and is absorbed into a universal kernel
\(H_s(r/b)\). Hence, using the parametrization by \((b,d)\), the original
Farey sum is reorganized into outer sums over \(b\) and inner reduced-residue
sums over the inverse of \(r\) modulo \(b\).

For \(\Re(s)>\frac12\) and \(u\in(0,1]\), define
\[
H_s(u):=\sum_{k=0}^\infty (k+u)^{-s}(k+1+u)^{-s}.
\]

For \(\Re(s)>\frac23\), where the original Farey series converges
absolutely, Lemma~\ref{lem:param} and the decomposition
\(d=kb+r\), with \(\gcd(r,b)=1\), give
\begin{equation}
\label{eq:Hsdecomp}
Z_f^{\mathrm{end}}(s)
=
2^{-s}
\sum_{b\ge1}b^{-3s}
\sum_{\substack{1\le r\le b\\(r,b)=1}}
|f''(\overline r/b)|^s\,H_s(r/b),
\end{equation}
where \(\overline r\) denotes the inverse of \(r\) modulo \(b\), represented in
\(\{1,\dots,b\}\); the convention at \(b=1\) affects only a harmless
holomorphic contribution. The individual functions \(H_s(u)\) and the
finite reduced-residue sums occurring below are holomorphic for
\(\Re(s)>\frac12\). Below the line \(\Re(s)=\frac23\), however,
\eqref{eq:Hsdecomp} is used only through the meromorphic continuation
constructed below. Note that the arithmetic kernel depends on \(r/b\),
while the geometric sampling point is \(\overline r/b\), and this is
precisely why both \(r\) and \(\overline r\) appear simultaneously.

\begin{lemma}[Integral of \(H_s\)]
\label{lem:intHs}
For \(\frac12<\Re(s)<1\),
\[
\int_0^1 H_s(u)\,du
=
\int_0^\infty t^{-s}(1+t)^{-s}\,dt
=
\frac{\Gamma(1-s)\Gamma(2s-1)}{\Gamma(s)}.
\]
In particular,
\[
\int_0^1 H_{2/3}(u)\,du
=
\frac{\Gamma\!\left(\frac13\right)^2}{\Gamma\!\left(\frac23\right)}.
\]
\end{lemma}

\begin{proof}
For \(\frac12<\Re(s)<1\), the series is absolutely summable after
integration over \(u\in(0,1)\). Indeed, the terms with \(k\ge1\) are
controlled by
\[
\sum_{k\ge1}k^{-2\Re(s)},
\]
while the \(k=0\) term is integrable near \(u=0\) because
\(\Re(s)<1\). Hence Fubini's theorem gives
\[
\int_0^1 H_s(u)\,du
=
\sum_{k\ge0}\int_0^1
(k+u)^{-s}(k+1+u)^{-s}\,du.
\]
Substituting \(t=k+u\) in each term and observing that the intervals
\([k,k+1]\) partition \([0,\infty)\), we obtain
\[
\int_0^\infty t^{-s}(1+t)^{-s}\,dt,
\]
which is the stated Beta integral.
\end{proof}

The decomposition \eqref{eq:Hsdecomp} reduces the residue problem to
understanding, for each fixed \(b\), the reduced-residue average of the product
\[
H_s(r/b)\,|f''(\overline r/b)|^s.
\]
We denote this inner sum by \(\Sigma_b(s)\). The key analytic step is that
\(\Sigma_b(s)\) admits an asymptotic formula with a power-saving error,
uniformly in \(s\) on compact subsets of
\[
\left\{\frac12<\Re(s)<1\right\}.
\]

Define
\[
\Sigma_b(s)
:=
\sum_{\substack{1\le r\le b\\(r,b)=1}}
H_s(r/b)\,|f''(\overline r/b)|^s.
\]
We prove the following equidistribution estimate for \(\Sigma_b(s)\).

\begin{proposition}
\label{prop:spec}
Fix a compact set
\[
K\subset\left\{\frac12<\Re(s)<1\right\}
\]
and assume \eqref{eq:ass}. Then for every \(\varepsilon>0\) and \(s\in K\)
we have
\[
\Sigma_b(s)
=
\varphi(b)
\left(\int_0^1 H_s(u)\,du\right)
\left(\int_0^1 |f''(v)|^s\,dv\right)
+
O_{f,K,\varepsilon}
\bigl(b^{\frac{1+\Re(s)}{2}+\varepsilon}\bigr),
\]
with implied constant uniform for \(s\in K\).
\end{proposition}

The proof is given on page~\pageref{proof:prop:spec}.

At this point the analytic difficulty is essentially finished. Indeed,
Proposition~\ref{prop:spec} turns \(Z_f^{\mathrm{end}}(s)\) into an explicit
main term involving
\[
\sum_{b\ge1}\varphi(b)b^{-3s},
\]
plus an error which is holomorphic near \(s=\frac23\). Thus the pole and its
residue come from the Euler-factor Dirichlet series for \(\varphi\).

\begin{proof}[Proof of Theorem~\ref{thm:res}]
\label{proof:thm:res}
By Lemma~\ref{lem:replace}, it suffices to compute the residue of
\(Z_f^{\mathrm{end}}(s)\). Fix a disk
\[
D:=\{s:|s-\tfrac23|\le \eta\}
\]
with \(0<\eta<1/21\), and set
\[
\sigma_-:=\frac23-\eta,
\qquad
\sigma_+:=\frac23+\eta.
\]

From \eqref{eq:Hsdecomp} and Proposition~\ref{prop:spec}, applied on a
compact set containing \(D\) but lying in
\[
\left\{\frac12<\Re(s)<1\right\},
\]
we obtain, on the nonempty open set
\[
D^\circ\cap\left\{\Re(s)>\frac23\right\},
\]
the identity
\[
Z_f^{\mathrm{end}}(s)
=
A(s)\sum_{b\ge1}\frac{\varphi(b)}{b^{3s}}
+
E(s),
\]
where
\[
A(s):=
2^{-s}
\left(\int_0^1 H_s(u)\,du\right)
\left(\int_0^1 |f''(v)|^s\,dv\right)
\]
and
\[
E(s):=
2^{-s}\sum_{b\ge1}b^{-3s}R_b(s),
\qquad
|R_b(s)|
\ll_{f,D,\varepsilon}
b^{\frac{1+\Re(s)}{2}+\varepsilon}.
\]

For \(s\in D\),
\[
|b^{-3s}R_b(s)|
\ll_{f,D,\varepsilon}
b^{-3\sigma_-+\frac{1+\sigma_+}{2}+\varepsilon}
=
b^{-\frac76+\frac{7\eta}{2}+\varepsilon}.
\]
Choose \(\varepsilon>0\) so small that
\[
-\frac76+\frac{7\eta}{2}+\varepsilon<-1,
\]
which is possible because \(\eta<1/21\). Hence
\[
\sum_{b\ge1}b^{-7/6+7\eta/2+\varepsilon}
\]
converges, and the series defining \(E(s)\) converges uniformly on \(D\)
by the Weierstrass \(M\)-test. Since each summand is holomorphic in \(s\),
the function \(E(s)\) is holomorphic on \(D^\circ\).

The function \(A(s)\) is holomorphic on
\[
\frac12<\Re(s)<1
\]
by Lemma~\ref{lem:intHs} and dominated convergence. The classical identity
\[
\sum_{n\ge1}\varphi(n)n^{-w}
=
\frac{\zeta(w-1)}{\zeta(w)},
\qquad
\Re(w)>2,
\]
provides the meromorphic continuation of the totient Dirichlet series.
Define, for \(s\in D^\circ\),
\[
\widetilde Z_f^{\mathrm{end}}(s)
:=
A(s)\frac{\zeta(3s-1)}{\zeta(3s)}
+
E(s).
\]
Because \(\eta<1/21\), one has \(\Re(3s)>1\) throughout \(D\), so
\(\zeta(3s)\neq0\) there. Thus
\(\widetilde Z_f^{\mathrm{end}}\) is meromorphic on \(D^\circ\), with at
most a simple pole at \(s=\frac23\).

On
\[
D^\circ\cap\left\{\Re(s)>\frac23\right\},
\]
the original series converges absolutely and
\[
\sum_{b\ge1}\frac{\varphi(b)}{b^{3s}}
=
\frac{\zeta(3s-1)}{\zeta(3s)}.
\]
Consequently,
\[
\widetilde Z_f^{\mathrm{end}}(s)
=
Z_f^{\mathrm{end}}(s)
\]
on this overlap. Hence \(\widetilde Z_f^{\mathrm{end}}\) is the meromorphic
continuation of \(Z_f^{\mathrm{end}}\) to \(D^\circ\).

By Lemma~\ref{lem:replace}, the difference
\(Z_f-Z_f^{\mathrm{end}}\) is holomorphic near \(s=\frac23\).
Therefore \(Z_f\) has the same pole and residue as
\(\widetilde Z_f^{\mathrm{end}}\), and
\[
\Res_{s=2/3} Z_f(s)
=
A(2/3)\cdot\frac{1}{3\zeta(2)}.
\]

Now
\[
A(2/3)
=
2^{-2/3}
\left(\int_0^1 H_{2/3}(u)\,du\right)
\left(\int_0^1 |f''(v)|^{2/3}\,dv\right).
\]
By Lemma~\ref{lem:intHs},
\[
\int_0^1 H_{2/3}(u)\,du
=
\frac{\Gamma\!\left(\frac13\right)^2}
{\Gamma\!\left(\frac23\right)}.
\]
Using
\[
\zeta(2)=\frac{\pi^2}{6},
\qquad
\Gamma\!\left(\frac13\right)\Gamma\!\left(\frac23\right)
=
\frac{2\pi}{\sqrt3},
\]
the constant simplifies to
\[
\frac{\sqrt3\,\Gamma\!\left(\frac13\right)^3}{2^{2/3}\pi^3}.
\]
This proves the theorem.
\end{proof}

\subsection{Hata's basis and Legendre duality}

We now explain that the boundary zeta series is, up to terms holomorphic at
\(s=\frac23\), the zeta series attached to the Hata coefficients of the
Legendre dual. Thus Theorem~\ref{thm:boundary-residue} is a direct
consequence of Theorem~\ref{thm:res} applied to a suitable dual function.

Let \(\Gamma\subset \partial\Omega\) be one of the \(3\), \(4\), \(5\), or \(6\) arcs
cut out by the chosen supporting lines from the minimal model. After an
\(\mathrm{SL}(2,\Z)\) transformation followed by a translation, we may assume
that \(\Gamma\) is the graph of a convex, decreasing \(C^3\) function
\[
y=g(x),
\qquad
x\in[0,A],
\]
and that the parameter
\[
u=-g'(x)
\]
runs through \([0,1]\) along \(\Gamma\).

Define
\[
\widetilde g(u):=g^*(-u),
\qquad
u\in[0,1],
\]
where
\[
g^*(\alpha):=\sup_{x\in[0,A]}(\alpha x-g(x))
\]
is the Legendre transform of \(g\).

A supporting line to \(y=g(x)\) with primitive integer normal
\[
(a,b)\in \Z_{\ge 0}^2\setminus\{(0,0)\}
\]
has equation
\[
ax+by=\gamma_{a,b}.
\]
If \(b>0\), then
\[
y=-\frac abx+\frac{\gamma_{a,b}}{b}.
\]
Its slope is \(-a/b\), and comparison with the Legendre representation
\[
y=\alpha x-g^*(\alpha)
\]
shows that
\[
\widetilde g\!\left(\frac ab\right)
=
g^*\!\left(-\frac ab\right)
=
-\frac{\gamma_{a,b}}{b}.
\]
Thus the values of \(\widetilde g\) at rational points encode the support
data of \(\Omega\).

\begin{lemma}[Identification of the geometric weights]
\label{lem:dual-identification}
Let
\[
I=\left[\frac cd,\frac ab\right]
\]
be a Farey interval, and let
\[
v_I=(a+c,b+d)
\]
be the associated primitive vector. Then
\[
T_I(\widetilde g)
:=
(b+d)c_I(\widetilde g)
=
-\gamma_{a+c,b+d}+\gamma_{a,b}+\gamma_{c,d}.
\]
In particular,
\[
|T_I(\widetilde g)|=f_\Gamma(a,b,c,d).
\]
\end{lemma}

\begin{proof}
By the coefficient formula,
\begin{align*}
c_I(\widetilde g)
&=
\widetilde g\!\left(\frac{a+c}{b+d}\right)
-\frac{b}{b+d}\widetilde g\!\left(\frac ab\right)
-\frac{d}{b+d}\widetilde g\!\left(\frac cd\right) \\
&=
-\frac{\gamma_{a+c,b+d}}{b+d}
+\frac{\gamma_{a,b}}{b+d}
+\frac{\gamma_{c,d}}{b+d}.
\end{align*}
Multiplying by \(b+d\) gives the formula.
\end{proof}

The previous lemma identifies the nonzero summands in the geometric
Dirichlet series attached to \(\Gamma\) with the corresponding summands in
the Hata-coefficient zeta series of the dual function \(\widetilde g\).
Only the finitely many endpoint conventions in the two Stern--Brocot
parametrizations can differ.

\begin{corollary}
\label{cor:Fgamma-is-Zg}
There exists a finite Dirichlet polynomial \(Q_\Gamma(s)\), hence an entire
function, such that in the half-plane of absolute convergence
\[
F_\Gamma(s)
=
Z_{\widetilde g}(s)+Q_\Gamma(s).
\]
Consequently, every meromorphic continuation of
\(Z_{\widetilde g}(s)\) induces one of \(F_\Gamma(s)\), with the same poles
and residues. In particular,
\[
\Res_{s=2/3}F_\Gamma(s)
=
\Res_{s=2/3}Z_{\widetilde g}(s).
\]
If the same endpoint convention is used in both parametrizations, then
\(Q_\Gamma=0\).
\end{corollary}

\begin{proof}
Lemma~\ref{lem:dual-identification} gives, term by term,
\[
|T_I(\widetilde g)|^s
=
f_\Gamma(a,b,c,d)^s
\]
for the Farey interval \(I=[c/d,a/b]\) and its corresponding unimodular
pair. The Farey intervals and the unimodular triangles in the
Stern--Brocot decomposition are indexed by the same adjacent primitive
pairs. Terms whose slopes lie strictly outside the interval swept out by
\(-g'\) contribute zero because the support function is affine on the
corresponding endpoint cones. Thus the only possible discrepancy comes
from the finitely many endpoint conventions. Their contribution is a
finite sum of functions of the form \(q^s\), and therefore defines an
entire function \(Q_\Gamma(s)\).
\end{proof}

It remains to identify the integral appearing in Theorem~\ref{thm:res} with
the equiaffine arc length of \(\Gamma\).

\begin{lemma}[Legendre duality and equiaffine arc length]
\label{lem:Legendre-ea}
The function \(\widetilde g\) satisfies the assumptions of
Theorem~\ref{thm:res}. Moreover,
\[
\int_0^1 (\widetilde g''(u))^{2/3}\,du
=
\int_0^A (g''(x))^{1/3}\,dx
=
\Length_{\mathrm{equiaffine}}(\Gamma).
\]
\end{lemma}

\begin{proof}
Since \(g\in C^3\) is convex and has nonvanishing curvature on a compact
interval, there exist constants \(0<m\le M<\infty\) such that
\[
0<m\le g''(x)\le M
\qquad
(x\in[0,A]).
\]
Standard Legendre duality implies that \(\widetilde g\in C^3([0,1])\) and,
if \(u=-g'(x)\), then
\[
\widetilde g''(u)=\frac{1}{g''(x)}.
\]
Also,
\[
du=-g''(x)\,dx.
\]
Since \(u=-g'(x)\) decreases from \(1\) to \(0\) as \(x\) increases from
\(0\) to \(A\), we have
\[
\int_0^1 (\widetilde g''(u))^{2/3}\,du
=
\int_A^0 (g''(x))^{-2/3}\,(-g''(x))\,dx
=
\int_0^A (g''(x))^{1/3}\,dx.
\]
This proves the claim.
\end{proof}

\begin{proof}[Proof of Theorem~\ref{thm:boundary-residue}]
\label{proof:thm:boundary-residue}
By Corollary~\ref{cor:Fgamma-is-Zg}, the residue of \(F_\Gamma(s)\) at
\(s=\frac23\) is the same as the residue of \(Z_{\widetilde g}(s)\). Applying
Theorem~\ref{thm:res} with \(f=\widetilde g\), we obtain
\[
\Res_{s=2/3}F_\Gamma(s)
=
\frac{\sqrt3\,\Gamma\!\left(\frac13\right)^3}{2^{2/3}\pi^3}
\int_0^1 (\widetilde g''(u))^{2/3}\,du.
\]
By Lemma~\ref{lem:Legendre-ea},
\[
\int_0^1 (\widetilde g''(u))^{2/3}\,du
=
\Length_{\mathrm{equiaffine}}(\Gamma).
\]
Hence
\[
\Res_{s=2/3}F_\Gamma(s)
=
\frac{\sqrt3\,\Gamma\!\left(\frac13\right)^3}{2^{2/3}\pi^3}
\Length_{\mathrm{equiaffine}}(\Gamma).
\]
This proves the theorem.
\end{proof}

\subsection{Geometric meaning of the terms of the boundary zeta series}
\label{sseq:geommeaning}

The summands of the boundary zeta series admit a simple geometric
interpretation in terms of small triangles cut out by supporting lines.

Let \((a,b)\) and \((c,d)\) be a unimodular pair, so that \(ad-bc=1\), and
let \(\Delta_{a,b,c,d}\) be the triangle bounded by the three supporting lines
with primitive normals
\[
(a,b),
\qquad
(c,d),
\qquad
(a+c,b+d).
\]
Then
\[
\Area(\Delta_{a,b,c,d})
=
\frac12
\Bigl(\gamma_{a,b}+\gamma_{c,d}-\gamma_{a+c,b+d}\Bigr)^2.
\]

Indeed, after applying an \(\mathrm{SL}(2,\Z)\) transformation and then a
translation, we may reduce to the case
\[
(a,b)=(1,0),
\qquad
(c,d)=(0,1),
\qquad
\gamma_{a,b}=\gamma_{c,d}=0.
\]
In these coordinates, the three supporting lines are
\[
x=0,
\qquad
y=0,
\qquad
x+y=\gamma_{a+c,b+d}.
\]
Thus \(\Delta_{a,b,c,d}\) is a right isosceles triangle whose legs have length
\[
\bigl|\gamma_{a,b}+\gamma_{c,d}-\gamma_{a+c,b+d}\bigr|.
\]
Its area is therefore exactly
\[
\frac12
\Bigl(\gamma_{a,b}+\gamma_{c,d}-\gamma_{a+c,b+d}\Bigr)^2.
\]

Consequently, the terms in
\[
F_\Gamma(s)
=
\sum_{\substack{a,b,c,d\in\Z_{\ge0}\\ ad-bc=1}}
f_\Gamma(a,b,c,d)^s
\]
are given by
\[
f_\Gamma(a,b,c,d)
=
\bigl|\gamma_{a,b}+\gamma_{c,d}-\gamma_{a+c,b+d}\bigr|
=
\sqrt{2\,\Area(\Delta_{a,b,c,d})}.
\]

Thus the coefficients of Hata's expansion of the Legendre dual have a direct
geometric meaning: they measure the sizes of the support triangles determined
by neighboring primitive normals.

Theorem~\ref{thm:boundary-residue} shows that the first singularity of the
boundary Dirichlet series is governed by the asymptotic distribution of these
small support triangles. Since
\[
f_\Gamma(a,b,c,d)=\sqrt{2\,\Area(\Delta_{a,b,c,d})},
\]
the residue at \(s=\frac23\) may be viewed as encoding a \(1/3\)-power
summation law for the areas of these triangles. This is precisely the exponent
that appears in the formula \eqref{23affine} for the equiaffine length.
\section{Example: The parabolic model}\label{sec4}

In this section we compute explicitly the boundary zeta series for the
parabolic arc
\[
\Gamma_{\mathrm{par}}
:=
\{(x,y)\in \R_{\ge 0}^2:\sqrt{x}+\sqrt{y}=1\},
\]
and then deduce a closed formula for the tropical zeta function of the
symmetric domain
\[
L
=
\Bigl\{
(x,y)\in\R^2:\sqrt{1-|x|}+\sqrt{1-|y|}\ge 1
\Bigr\}.
\]
The resulting Dirichlet series is the primitive Mordell--Tornheim series and
therefore can be expressed in terms of Witten's \(\mathrm{SU}(3)\) zeta
function.

\subsection{The parabolic arc}

Let \(\Gamma_{\mathrm{par}}\) be the arc joining \((1,0)\) and \((0,1)\), and
let
\[
\Omega_{\mathrm{par}}
:=
\{(x,y)\in \R_{\ge0}^2:x+y\le 1,\ \sqrt{x}+\sqrt{y}\ge 1\}
\]
be the convex domain bounded by \(\Gamma_{\mathrm{par}}\) and the two
coordinate axes.

For a primitive vector
\[
(a,b)\in \Z_{\ge0}^2\setminus\{(0,0)\},
\]
let
\[
ax+by=\gamma_{a,b}
\]
be the supporting line to \(\Omega_{\mathrm{par}}\) with inward normal
\((a,b)\).

\begin{lemma}\label{lem:parabola-support}
For every \(a,b\ge 0\), not both zero, one has
\[
\gamma_{a,b}=\frac{ab}{a+b},
\]
with the convention that \(\gamma_{1,0}=\gamma_{0,1}=0\).
\end{lemma}

\begin{proof}
The boundary arc admits the parametrization
\[
x=t^2,\qquad y=(1-t)^2,\qquad 0\le t\le 1.
\]
Hence
\[
ax+by=a t^2+b(1-t)^2.
\]
Differentiating with respect to \(t\), we obtain
\[
\frac{d}{dt}\bigl(a t^2+b(1-t)^2\bigr)=2(a+b)t-2b,
\]
so the minimum is attained at
\[
t=\frac{b}{a+b}.
\]
Substituting this value gives
\[
\gamma_{a,b}
=
a\Bigl(\frac{b}{a+b}\Bigr)^2
+
b\Bigl(\frac{a}{a+b}\Bigr)^2
=
\frac{ab}{a+b}.
\]
\end{proof}

Now let \((a,b)\) and \((c,d)\) be a unimodular pair, so that
\[
a,b,c,d\in \Z_{\ge0},
\qquad
ad-bc=1.
\]
Recall that
\[
f_{\Gamma_{\mathrm{par}}}(a,b,c,d)
=
\bigl|\gamma_{a,b}+\gamma_{c,d}-\gamma_{a+c,b+d}\bigr|.
\]

\begin{proposition}\label{prop:parabola-defect}
For every unimodular pair \((a,b),(c,d)\), one has
\[
f_{\Gamma_{\mathrm{par}}}(a,b,c,d)
=
\frac{1}{(a+b)(c+d)(a+b+c+d)}.
\]
\end{proposition}

\begin{proof}
By Lemma~\ref{lem:parabola-support},
\[
\gamma_{a,b}+\gamma_{c,d}-\gamma_{a+c,b+d}
=
\frac{ab}{a+b}+\frac{cd}{c+d}
-\frac{(a+c)(b+d)}{a+b+c+d}.
\]
A direct simplification gives
\[
\frac{ab}{a+b}+\frac{cd}{c+d}
-\frac{(a+c)(b+d)}{a+b+c+d}
=
-\frac{(ad-bc)^2}{(a+b)(c+d)(a+b+c+d)}.
\]
Since \(ad-bc=1\), taking the absolute value yields
\[
f_{\Gamma_{\mathrm{par}}}(a,b,c,d)
=
\frac{1}{(a+b)(c+d)(a+b+c+d)}.
\]
\end{proof}

\subsection{The boundary series of the parabola}

We now identify the boundary zeta series of \(\Gamma_{\mathrm{par}}\) with a
primitive Mordell--Tornheim series.

\begin{lemma}\label{lem:parabola-bijection}
The map
\[
(a,b,c,d)\longmapsto (p,q):=(a+b,c+d)
\]
is a bijection between the set of quadruples
\[
a,b,c,d\in \Z_{\ge0},
\qquad
ad-bc=1,
\]
and the set of coprime pairs \((p,q)\in\N^2\) with \(\gcd(p,q)=1\).
\end{lemma}

\begin{proof}
If \(p=a+b\) and \(q=c+d\), then
\[
pd-qb=(a+b)d-(c+d)b=ad-bc=1,
\]
hence \(\gcd(p,q)=1\).

Conversely, let \(p,q\in\N\) be coprime. Choose the unique integer \(b\) with
\[
0\le b<p,
\qquad
qb\equiv -1 \pmod p.
\]
Then
\[
d:=\frac{qb+1}{p}\in\N.
\]
Since \(0\le b<p\), we have \(0<d\le q\). Now define
\[
a:=p-b,
\qquad
c:=q-d.
\]
Then \(a,b,c,d\in \Z_{\ge0}\), and
\[
ad-bc=(p-b)d-b(q-d)=pd-bq=1.
\]
Thus every coprime pair \((p,q)\) arises. Uniqueness follows from the
uniqueness of \(b\) modulo \(p\) in the range \(0\le b<p\).
\end{proof}

\begin{proposition}\label{prop:parabola-boundary-series}
For \(\Re(s)\) sufficiently large,
\[
F_{\Gamma_{\mathrm{par}}}(s)
=
\sum_{\substack{a,b,c,d\in\Z_{\ge0}\\ ad-bc=1}}
\frac{1}{\bigl((a+b)(c+d)(a+b+c+d)\bigr)^s}
=
\sum_{\substack{p,q\ge1\\ \gcd(p,q)=1}}
\frac{1}{\bigl(pq(p+q)\bigr)^s}.
\]
Consequently,
\[
F_{\Gamma_{\mathrm{par}}}(s)
=
\frac{1}{\zeta(3s)}
\sum_{p,q\ge1}\frac{1}{\bigl(pq(p+q)\bigr)^s}.
\]
\end{proposition}

\begin{proof}
The first identity follows from Proposition~\ref{prop:parabola-defect} and
Lemma~\ref{lem:parabola-bijection}.

For the second, write
\[
\sum_{p,q\ge1}\frac{1}{\bigl(pq(p+q)\bigr)^s}
=
\sum_{m\ge1}\ 
\sum_{\substack{u,v\ge1\\ \gcd(u,v)=1}}
\frac{1}{\bigl(mu\cdot mv\cdot m(u+v)\bigr)^s}.
\]
This becomes
\[
\sum_{m\ge1}\frac{1}{m^{3s}}
\sum_{\substack{u,v\ge1\\ \gcd(u,v)=1}}
\frac{1}{\bigl(uv(u+v)\bigr)^s}
=
\zeta(3s)\,F_{\Gamma_{\mathrm{par}}}(s),
\]
which proves the claim.
\end{proof}

Therefore,
\[
F_{\Gamma_{\mathrm{par}}}(s)
=
2^{-s}\frac{\zeta_{\mathrm{SU}(3)}(s)}{\zeta(3s)},
\]
where
\[
\zeta_{\mathrm{SU}(3)}(s)
:=
2^s\sum_{p,q\ge1}\frac{1}{\bigl(pq(p+q)\bigr)^s}
\]
is Witten's \(\mathrm{SU}(3)\) zeta function
\cite{Witten1991QuantumGauge2D}, which we discuss in
Subsection~\ref{ss:wittensu3}.

\subsection{\texorpdfstring{The special domain \(L\)}{The special domain L}}\label{ss:specialdomain}

We now return to the symmetric domain
\[
L
=
\Bigl\{
(x,y)\in\R^2:\sqrt{1-|x|}+\sqrt{1-|y|}\ge 1
\Bigr\}.
\]
Its boundary consists of four congruent parabolic arcs. In the first quadrant
the boundary is given by
\[
\sqrt{1-x}+\sqrt{1-y}=1,
\qquad
0\le x\le 1,\quad 0\le y\le 1.
\]
Equivalently, after the affine change of variables
\[
u=1-x,
\qquad
v=1-y,
\]
this arc is identified with \(\Gamma_{\mathrm{par}}\).

\begin{figure}[htbp]
    \centering
    \includegraphics[width=\linewidth]{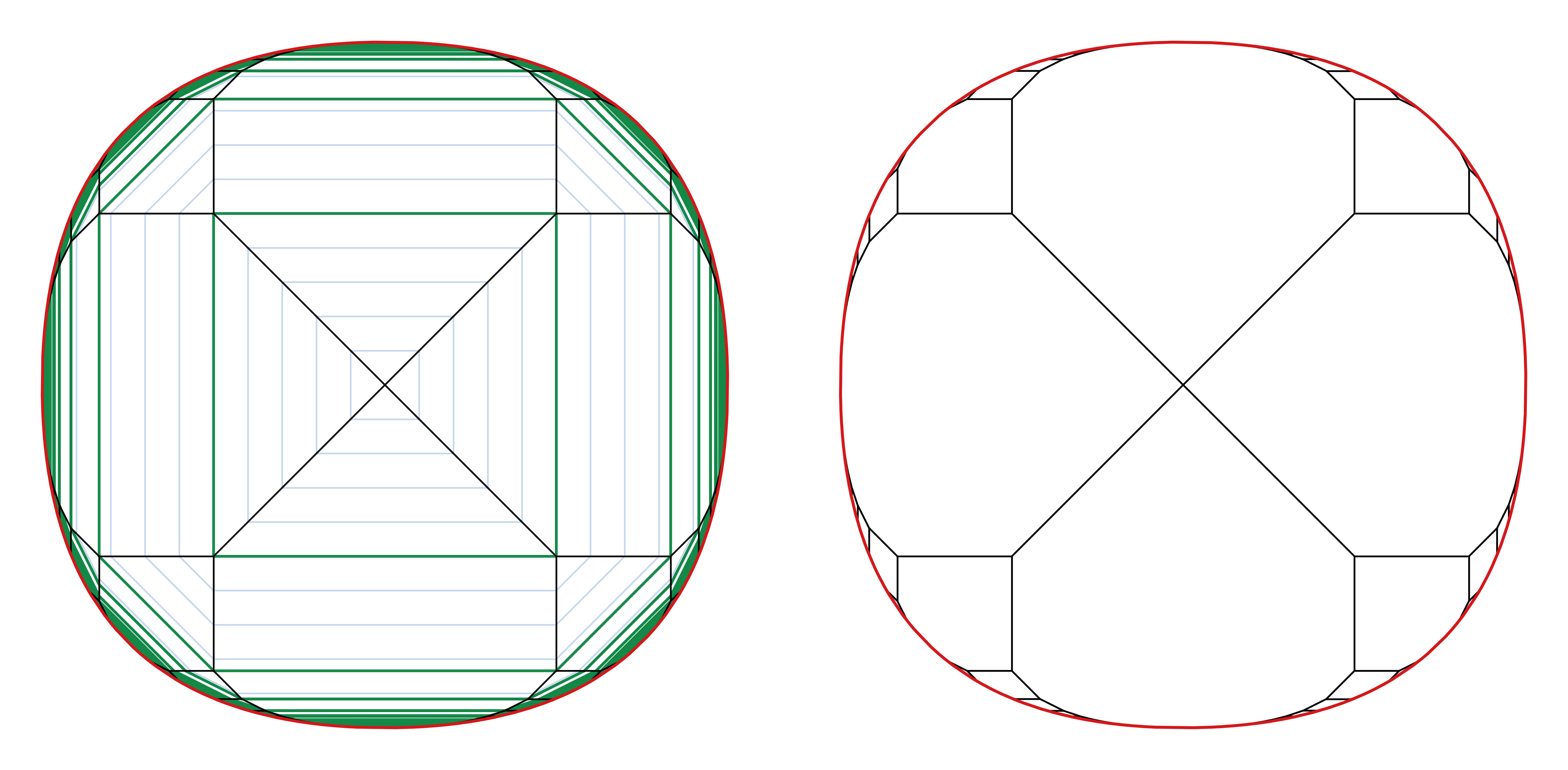}
    \caption{Tropical wave fronts and caustic in the parabolic model \(L\).
    The red boundary consists of four parabolic arcs. The blue and green
    polygons are tropical wave fronts, produced by successive unimodular
    corner cuts governed by the Stern--Brocot/Farey recursion. The
    right-hand picture isolates the corresponding caustic tree. The support
    defects attached to neighboring primitive directions are exactly the
    sizes of the triangles cut by this flow.}
    \label{fig:parabolicwavefront}
\end{figure}

\begin{proposition}\label{prop:L-affine-length}
The equiaffine arc length of one parabolic arc of \(\partial L\) is
\(4^{1/3}\). Hence
\[
\Length_{\mathrm{equiaffine}}(\partial L)=4^{4/3}.
\]
\end{proposition}

\begin{proof}
Consider the arc in the first quadrant. It is parametrized by
\[
\gamma(t)=(1-t^2,\ 2t-t^2),
\qquad
0\le t\le 1.
\]
Indeed,
\[
\sqrt{1-(1-t^2)}+\sqrt{1-(2t-t^2)}
=
t+(1-t)=1.
\]

Now
\[
\gamma'(t)=(-2t,\ 2-2t),
\qquad
\gamma''(t)=(-2,\ -2),
\]
hence
\[
\det(\gamma'(t),\gamma''(t))
=
\det
\begin{pmatrix}
-2t & -2\\
2-2t & -2
\end{pmatrix}
=4.
\]
Therefore the equiaffine speed is constant and equal to \(4^{1/3}\). Thus
one arc has equiaffine length
\[
\int_0^1 4^{1/3}\,dt=4^{1/3}.
\]
Since \(\partial L\) consists of four congruent arcs, the total equiaffine
length is
\[
4\cdot 4^{1/3}=4^{4/3}.
\]
\end{proof}

\begin{theorem}\label{thm:L-zeta}
For \(\Re(s)>2\),
\[
Z_L(s)
=
\left(
8-2^{2-s}\frac{\zeta_{\mathrm{SU}(3)}(s)}{\zeta(3s)}
\right)\frac{1}{s(s-1)}.
\]
In particular, \(Z_L(s)\) admits a meromorphic continuation to every point
where \(\zeta_{\mathrm{SU}(3)}(s)\) does.
\end{theorem}

\begin{proof}
By symmetry, the four boundary arcs of \(\partial L\) contribute the same
boundary series, namely
\[
F_{\partial L}(s)=4F_{\Gamma_{\mathrm{par}}}(s).
\]
Using Proposition~\ref{prop:parabola-boundary-series}, we obtain
\[
F_{\partial L}(s)
=
4\cdot 2^{-s}\frac{\zeta_{\mathrm{SU}(3)}(s)}{\zeta(3s)}
=
2^{2-s}\frac{\zeta_{\mathrm{SU}(3)}(s)}{\zeta(3s)}.
\]

On the other hand, the minimal model of \(L\) is the square
\[
\widehat L=[-1,1]^2.
\]
Applying the rectangle formula from the remark after
Theorem~\ref{thm:integral-boundary} with
\[
P=Q=2,
\]
we get
\[
H_{\widehat L}(s)=8.
\]
Therefore Theorem~\ref{thm:integral-boundary} yields
\[
s(s-1)Z_L(s)
=
-F_{\partial L}(s)+8,
\]
that is,
\[
Z_L(s)
=
\left(
8-2^{2-s}\frac{\zeta_{\mathrm{SU}(3)}(s)}{\zeta(3s)}
\right)\frac{1}{s(s-1)}.
\]
\end{proof}

\begin{corollary}\label{cor:L-residue}
The function \(Z_L(s)\) has a simple pole at \(s=\frac23\), and
\[
\Res_{s=2/3} Z_L(s)
=
\frac{18\sqrt3}{\pi^3}\Gamma\!\left(\frac13\right)^3.
\]
Equivalently,
\[
\Res_{s=2/3} Z_L(s)
=
\frac{9\sqrt3}{2\cdot 4^{1/3}\pi^3}
\Gamma\!\left(\frac13\right)^3
\Length_{\mathrm{equiaffine}}(\partial L).
\]
\end{corollary}

\begin{proof}
The second formula follows immediately from Proposition~\ref{prop:L-affine-length}
and the general residue theorem for smooth convex domains with everywhere
nonvanishing curvature.\footnote{The theorem assumes a \(C^3\)-smooth boundary, whereas \(\partial L\) is only \(C^1\) at four points. This causes no difficulty because these are precisely the junctions of the \(C^3\) arcs determined by the minimal-model subdivision.} Since
\[
\Length_{\mathrm{equiaffine}}(\partial L)=4^{4/3},
\]
we obtain
\[
\Res_{s=2/3} Z_L(s)
=
\frac{9\sqrt3}{2\cdot 4^{1/3}\pi^3}
\Gamma\!\left(\frac13\right)^3
\cdot 4^{4/3}
=
\frac{18\sqrt3}{\pi^3}\Gamma\!\left(\frac13\right)^3.
\]
\end{proof}

\begin{figure}[htbp]
    \centering
    \includegraphics[width=0.8\linewidth]{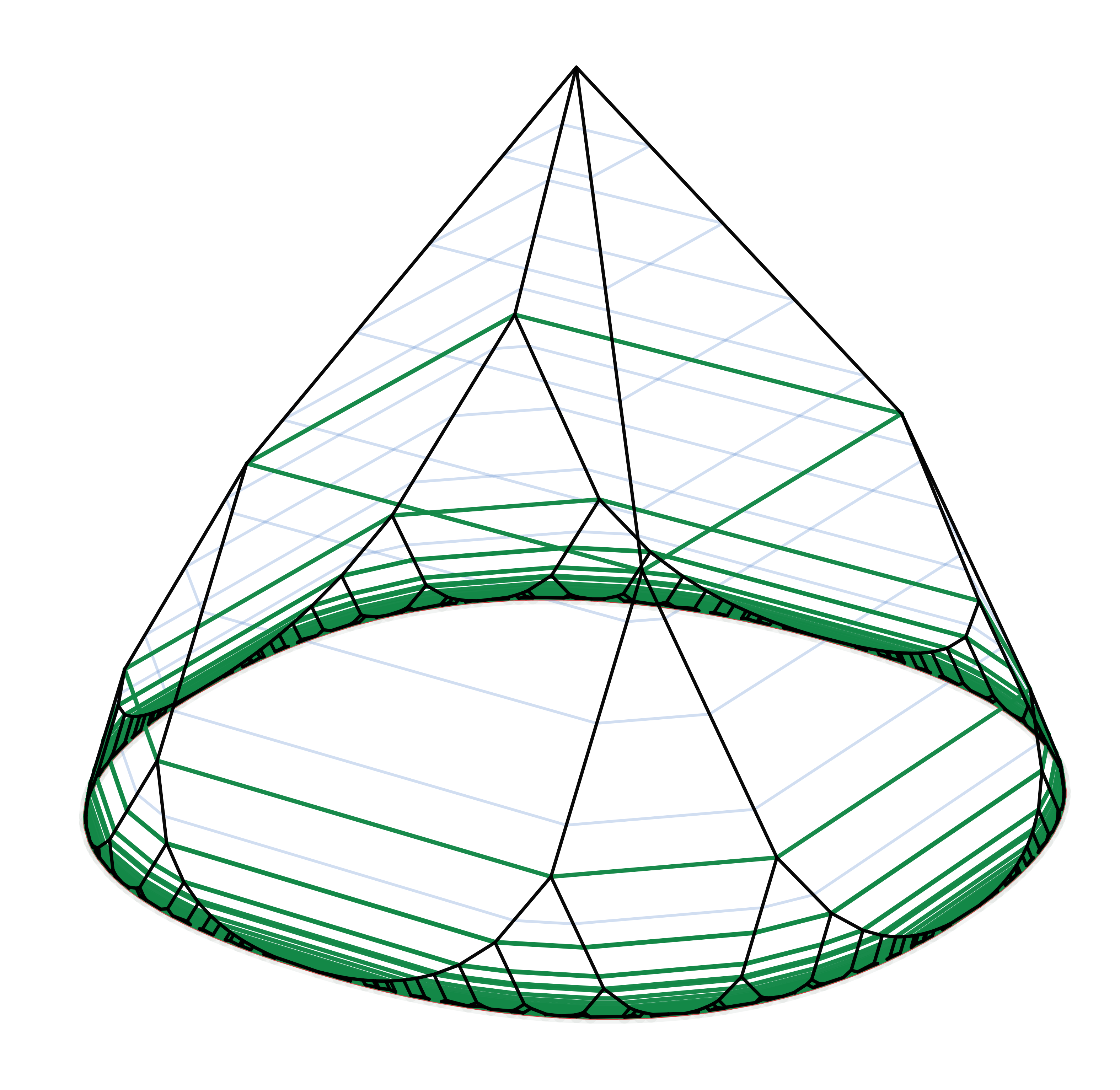}
    \caption{Three-dimensional lift of the parabolic model. Horizontal
    sections of the graph of \(\rho_L\) give the tropical wave fronts, while
    the ridge locus gives the caustic and its Farey branching. The visible
    polyhedral facets correspond to primitive supporting directions. In this
    model the support defects are read from elementary triangular pieces, and
    their zeta summation is the geometric source of the
    Mordell--Tornheim/\(\mathrm{SU}(3)\) structure.}
    \label{fig:parabolaseries}
\end{figure}

The explicit formula above is one reason the parabola should be regarded as
the basic model case. Quite remarkably, the same parabolic geometry appears
in several different contexts.

\subsection{\texorpdfstring{Witten's \(\mathrm{SU}(3)\) zeta function}{Witten's SU(3) zeta function}}\label{ss:wittensu3}

Following Witten and Zagier
\cite{Witten1991QuantumGauge2D,Zagier1994ValuesZeta}, for a compact
semisimple Lie group \(G\), one considers the representation zeta function
\[
\zeta_G(s):=\sum_{\rho\in \operatorname{Irrep}(G)} (\dim \rho)^{-s}.
\]
In the case \(G=\mathrm{SU}(2)\), the irreducible representations have
dimensions \(1,2,3,\dots\), so \(\zeta_{\mathrm{SU}(2)}(s)\) is simply the
Riemann zeta function. The next case,
\[
\zeta_{\mathrm{SU}(3)}(s),
\]
is exactly the one that appears in the parabolic model. Indeed, irreducible
representations of \(\mathrm{SU}(3)\) are indexed by pairs \((p,q)\in\N^2\),
and
\[
\dim V_{p,q}=\frac12\,pq(p+q).
\]
Therefore
\[
\zeta_{\mathrm{SU}(3)}(s)
=
2^s\sum_{p,q\ge1}\frac{1}{\bigl(pq(p+q)\bigr)^s}.
\]
Our computation shows that the boundary zeta series of the parabola is
precisely the primitive version of this double series, obtained by dividing by
\(2^s\zeta(3s)\).

The analytic structure of \(\zeta_{\mathrm{SU}(3)}\) was studied in
\cite{Romik2017}. In particular, \(\zeta_{\mathrm{SU}(3)}(s)\) admits a
meromorphic continuation to \(\C\), with simple poles at
\[
s=\frac23,
\qquad
s=\frac12-k
\quad
(k=0,1,2,\dots),
\]
and
\[
\Res_{s=2/3}\zeta_{\mathrm{SU}(3)}(s)
=
\frac{4^{1/3}}{2\pi\sqrt3}\,
\Gamma\!\left(\frac13\right)^3.
\]
This is the same transcendental factor that appears in our main residue
formula after the elementary normalization coming from the geometric setup.

\subsection{Moduli space volume}\label{sec_moduli}

The same triangle-cutting mechanism also appears in the algebro-geometric
story behind the universal elliptic curve over the moduli space \(A_1\). On
the one hand, for moduli spaces \(A_g\) the volume is interpreted as a
self-intersection number of the Hodge bundle; on the other hand, for the
universal elliptic curve \(B_1\to A_1\), Kramer and von Pippich explain that
one compares the self-intersection \(L_1\cdot L_1\) of the distinguished line
bundle \(L_1\) with the Chern--Weil integral
\[
\operatorname{Vol}(B_1)
=
\int_{B_1}
c_1(L_1,\|\cdot\|_\Theta)\wedge c_1(L_1,\|\cdot\|_\Theta),
\]
and the defect is measured by the special value
\(\zeta_{\mathrm{MT}}(2,2;2)\) \cite{KramerPippich2015Snapshot}. More
precisely, after compactifying \(B_1\) over the cusp and resolving the nodal
boundary fiber by an infinite chain of blow-ups, each exceptional divisor
contributes a correction term, so that the discrepancy is expressed by a
Mordell--Tornheim series
\cite{KramerPippich2015Snapshot,BurgosGilKramerKuehn2016}.

What is especially close to our picture is that, in the toric model extracted
in \cite{BurgosGilKramerKuehn2016}, self-intersection is literally computed
by areas of convex regions: if
\[
\Delta=\operatorname{conv}\{(0,0),(1,0),(0,1)\},
\]
then
\[
\operatorname{div}(x_0)^2=2\operatorname{Vol}(\Delta)=1,
\]
whereas for the singular metric one obtains a b-divisor with
\[
\mathrm{b}\text{-}\operatorname{div}
(x_0,\|\cdot\|_{\mathrm{sing}})^2
=
2\operatorname{Vol}(\Delta_{\mathrm{sing}}),
\]
where
\[
\Delta_{\mathrm{sing}}
=
\{(x,y)\in\R_{\ge 0}^2:x+y\le 1,\ \sqrt{x}+\sqrt{y}\ge 1\}.
\]
Hence
\begin{align*}
\operatorname{div}(x_0)^2
-
\mathrm{b}\text{-}\operatorname{div}
(x_0,\|\cdot\|_{\mathrm{sing}})^2
&=
2\operatorname{Vol}(\Delta)-2\operatorname{Vol}(\Delta_{\mathrm{sing}}) \\
&=
2\int_0^1 (1-\sqrt{x})^2\,dx \\
&=
\frac13.
\end{align*}
Equivalently, this is the statement
\[
\zeta(2,2;2)=\frac13\,\zeta(6),
\]
and geometrically the relevant boundary curve is exactly the parabola
\[
\sqrt{x}+\sqrt{y}=1.
\]
This is why, in our setting, the parabola should be regarded as the model
case: it is the same curve along which the infinite resolution creates the
triangular defect terms whose total contribution is governed by the
Mordell--Tornheim zeta value
\cite{KramerPippich2015Snapshot,BurgosGilKramerKuehn2016}.

\subsection{Concentration of measure}

There is also a probabilistic reason for singling out this domain. Bárány
proved \cite{Barany1995} that, for convex lattice polygons in the square,
asymptotically almost all of them approach a deterministic limit shape, and
this limit shape is exactly
\[
L
=
\Bigl\{
(x,y)\in\R^2:\sqrt{1-|x|}+\sqrt{1-|y|}\ge 1
\Bigr\}.
\]
More precisely, if
\[
\mathcal P_n
=
\Bigl\{
P\subset[-1,1]^2:
P \text{ is a convex } (1/n)\Z^2\text{-lattice polygon}
\Bigr\},
\]
then for every \(\varepsilon>0\),
\[
\frac{
\#\{P\in\mathcal P_n:\Area(P\triangle L)>\varepsilon\}
}{
\#\mathcal P_n
}
\longrightarrow 0
\qquad (n\to\infty),
\]
where
\[
P\triangle L=(P\setminus L)\cup(L\setminus P)
\]
is the symmetric difference.

Thus the same domain \(L\) that appears naturally in our explicit
tropical-zeta computation is also the limit shape in a classical lattice
polygon problem.

\section{\texorpdfstring{Lattice-point counting on \(L\)}{Lattice-point counting on L}}\label{sec_6}

We continue with the special domain from Subsection~\ref{ss:specialdomain}:
\begin{equation}\label{eq:Ldomain}
L=\{(x,y)\in \R^2:\sqrt{1-|x|}+\sqrt{1-|y|}\ge 1\}.
\end{equation}
Its area is
\begin{equation}\label{eq:areaL}
\Area(L)=\frac{10}{3}.
\end{equation}
For \(n\ge1\) put
\begin{equation}\label{eq:NLAL}
N_L(n)=|nL\cap\Z^2|,
\qquad
E_L(n)=N_L(n)-\frac{10}{3}n^2,
\qquad
A_L(N)=\frac1N\sum_{n\le N}E_L(n).
\end{equation}

The aim of this section is to evaluate the first nonvanishing term of
\(A_L(N)\). Conceptually, this gives an exactly solvable test case for a
tropical approach to the Gauss circle problem. The guiding geometric
observation is that the parabolic domain \(L\) and the Euclidean disk have
the same caustic combinatorics in the tropical/Farey cutting picture: both
have the minimal model \([-1,1]^2\), the same Farey tree of primitive
directions, the same mediant recursion, and the same caustic-edge directions. What changes from
one domain to the other is the size assigned to each combinatorial piece. For
\(L\), these sizes are the rational cubic weights
\[
m_{p,q}=pq(p+q),
\]
which lead to the \(\mathrm{SU}(3)\) zeta function. For the disk, the
corresponding sizes are expected to involve Euclidean norms, square roots,
and continued-fraction arithmetic. Thus Section~\ref{sec_6} analyzes the rational/parabolic member of this
caustic class, while the disk is its classical quadratic-irrational member.

This perspective places the theorem below between the exact computation of
\(Z_L(s)\) and the Gauss circle problem. The classical circle problem asks for
sharp control of the pointwise error
\[
E_D(R)=|RD\cap\Z^2|-\pi R^2
\]
for the disk \(D\). The statistic studied here is softer: it is an averaged
discrepancy for the special domain \(L\). Its value is nevertheless highly
structured. The tropical zeta function supplies the natural Farey scale
through the pole at \(s=\frac23\), the head--tail decomposition organizes the
raw cell counts, and the apparent bulk contribution is absorbed by the
threshold term. The first surviving coefficient comes from the second layer
of Farey-edge corrections.

The sign of the result is also part of the structure. Since
\[
\zeta\!\left(-\frac12\right)<0,
\]
the constant
\[
C_L=\frac{32}{3}\zeta\!\left(-\frac12\right)\frac{\zeta(2)}{\zeta(3)}
\]
is negative. Thus, after averaging, the dilates \(nL\) contain fewer lattice
points than the area term \(\frac{10}{3}n^2\) predicts at first order. The
computation therefore separates three features that are usually entangled in
lattice-point problems: the primitive-direction combinatorics, the
tropical-zeta scale, and the signed residual discrepancy. This makes \(L\) a
concrete model for the question of whether the disk admits an analogous
tropical bulk cancellation and whether its Gauss error is governed by a
second layer of caustic arithmetic.

The precise statement is the following.

\begin{theorem}[Averaged lattice count on \(L\)]\label{thm:mainL}
One has
\begin{equation}\label{eq:mainCL}
A_L(N)\sim C_L\sqrt N,
\qquad
C_L=\frac{32}{3}\zeta\!\left(-\frac12\right)\frac{\zeta(2)}{\zeta(3)}.
\end{equation}
Equivalently,
\begin{equation}\label{eq:mainCLraw}
N^{-3/2}\sum_{n\le N}
\left(|nL\cap\Z^2|-\frac{10}{3}n^2\right)
\longrightarrow
\frac{32}{3}\zeta\!\left(-\frac12\right)\frac{\zeta(2)}{\zeta(3)}.
\end{equation}
\end{theorem}

Since \(\zeta(-1/2)<0\), the first averaged correction is negative: after
averaging, the dilates \(nL\) contain fewer lattice points than the area term
predicts at order \(\sqrt N\). The same computation gives the doubled
square-root sawtooth constant appearing in the column count.

\begin{corollary}[Doubled square-root sawtooth constant]\label{cor:saw}
Let
\[
\psi(x)=\{x\}-\frac12.
\]
Then
\begin{equation}\label{eq:triSaw}
\lim_{N\to\infty}N^{-3/2}
\sum_{1\le i<j\le N}\psi(2\sqrt{ij})
=
\frac43\zeta\!\left(-\frac12\right)
\left(1-2\frac{\zeta(2)}{\zeta(3)}\right).
\end{equation}
Consequently,
\begin{equation}\label{eq:squareSaw}
\lim_{N\to\infty}N^{-3/2}
\sum_{1<i,j\le N}\psi(2\sqrt{ij})
=
\frac83\zeta\!\left(-\frac12\right)
\left(1-2\frac{\zeta(2)}{\zeta(3)}\right).
\end{equation}
The omitted line \(i=1\) and the diagonal contribute only \(O(N)\), and
therefore do not affect the \(N^{3/2}\)-normalized limit.
\end{corollary}

\subsection{The tropical-zeta input and the bulk scale}

For this particular domain the tropical zeta function was computed in
Subsection~\ref{ss:specialdomain}. The boundary series on one parabolic arc is
\begin{equation}\label{eq:FGamma}
F_{\Gamma_{\mathrm{par}}}(s)
=
\sum_{\substack{p,q\ge1\\ \gcd(p,q)=1}}
\frac{1}{(pq(p+q))^s},
\end{equation}
and the four arcs contribute \(4F_{\Gamma_{\mathrm{par}}}(s)\). Equivalently,
using Witten's \(\mathrm{SU}(3)\) zeta function,
\begin{equation}\label{eq:ZL}
Z_L(s)
=
\left(
8-2^{2-s}\frac{\zeta_{\mathrm{SU}(3)}(s)}{\zeta(3s)}
\right)
\frac{1}{s(s-1)}.
\end{equation}
Thus the tropical zeta function organizes the Farey triangles by the cubic
weight
\begin{equation}\label{eq:mpq}
m_{p,q}=pq(p+q),
\qquad
f_{p,q}=m_{p,q}^{-1}.
\end{equation}
The corresponding counting function
\begin{equation}\label{eq:Bcount}
B(X)
=
\#\{(p,q)\in\N^2:\gcd(p,q)=1,\ pq(p+q)\le X\}
\end{equation}
has a natural \(X^{2/3}\)-bulk layer and a further \(X^{1/2}\)-layer. These
are the layers predicted by the pole structure of \eqref{eq:ZL}; if one
treats the head and tail terms separately, they appear throughout the
lattice-counting decomposition. The point of this section is that these bulk
layers enter the exact local bookkeeping and are then absorbed by the
head--tail cancellation. The first surviving averaged term comes from the
Farey edges \(p\) fixed or \(q\) fixed.

\subsection{Farey phases and the head--tail decomposition}

Let \(\lambda=(a,b)\) be a primitive direction in the first quadrant, and write
\[
|\lambda|=a+b.
\]
We use the phase convention
\begin{equation}\label{eq:alpha}
\alpha_\lambda\equiv -\frac{ab}{a+b}\pmod 1.
\end{equation}
Let \(\lambda_1,\lambda_2\) be Farey neighbors with
\[
|\lambda_1|=p,
\qquad
|\lambda_2|=q,
\qquad
|\lambda_1+\lambda_2|=p+q.
\]
The local head phase is
\begin{equation}\label{eq:phi}
\Phi_{p,q}(n)
=
\{n\alpha_{\lambda_1}\}
+
\{n\alpha_{\lambda_2}\}
-
\{n\alpha_{\lambda_1+\lambda_2}\}.
\end{equation}
The phase identity
\begin{equation}\label{eq:phaseID}
\alpha_{\lambda_1+\lambda_2}
\equiv
\alpha_{\lambda_1}+\alpha_{\lambda_2}
-\frac{1}{m_{p,q}}
\pmod 1
\end{equation}
will be used repeatedly. It is proved in Appendix~E.4.

A Farey triangle is in the head at time \(n\) if \(m_{p,q}\le n\), and in the
tail if \(m_{p,q}>n\). The local cell is a shifted unimodular triangle. Its
effective integer size is the continuous size \(n/m_{p,q}\) reduced by the
phase defect \(\Phi_{p,q}(n)\). Equivalently, if
\[
\delta_{p,q}(n)
=
\frac{n}{m_{p,q}}-\Phi_{p,q}(n)+2,
\]
then the number of lattice points assigned to the cell is the generalized Pick
count
\[
\frac{(\delta_{p,q}(n)-1)(\delta_{p,q}(n)-2)}2,
\]
with the half-open convention fixed globally. Expanding this single formula
in the head and tail regimes gives the following exact decomposition.

\begin{proposition}[Head--tail decomposition]\label{prop:ledger}
With the coherent half-open convention of Appendix~E.3, one has the exact
identity
\begin{equation}\label{eq:ledger}
E_L(n)=T_1(n)+T_2(n)-T_0(n)+H_0(n)+H_{12}(n)+1,
\end{equation}
where
\begin{align}
H_0(n)
&=
4\sum_{\substack{p,q\ge1\\ \gcd(p,q)=1\\ m_{p,q}\le n}}
n f_{p,q}\Phi_{p,q}(n),
\label{eq:H0}\displaybreak[0]\\
H_{12}(n)
&=
4\sum_{\substack{p,q\ge1\\ \gcd(p,q)=1\\ m_{p,q}\le n}}
\left(\frac12\Phi_{p,q}(n)-\frac12\Phi_{p,q}(n)^2\right),
\label{eq:H12}\\
T_1(n)
&=
4\sum_{\substack{p,q\ge1\\ \gcd(p,q)=1\\ m_{p,q}>n}}
\frac12 n f_{p,q},
\label{eq:T1}\\
T_2(n)
&=
4\sum_{\substack{p,q\ge1\\ \gcd(p,q)=1\\ m_{p,q}>n}}
\frac12 n^2f_{p,q}^2,
\label{eq:T2}\\
T_0(n)
&=
4\sum_{\substack{p,q\ge1\\ \gcd(p,q)=1\\ m_{p,q}>n}}
\ind_{\{n\alpha_{\lambda_1}\}+\{n\alpha_{\lambda_2}\}<n/m_{p,q}}.
\label{eq:T0}
\end{align}
The final \(+1\) is global: it is the Euler term already visible in the
ambient square count. It is not an \(O(1)\)-term attached to every Farey
triangle.
\end{proposition}

The density approximation to the threshold term is
\begin{equation}\label{eq:T0dens}
T^{\mathrm{dens}}_0(n)
=
4\sum_{\substack{p,q\ge1\\ \gcd(p,q)=1\\ m_{p,q}>n}}
\frac12\left(\frac{n}{m_{p,q}}\right)^2.
\end{equation}
In the present normalization \(T^{\mathrm{dens}}_0(n)=T_2(n)\). We therefore
define
\begin{equation}\label{eq:RT0}
R_{T_0}(n)=-T_0(n)+T^{\mathrm{dens}}_0(n).
\end{equation}
The sign-safe form of the decomposition is
\begin{equation}\label{eq:safeLedger}
E_L(n)=(H_0(n)+T_1(n))+H_{12}(n)+R_{T_0}(n)+1.
\end{equation}

\begin{figure}[p]
\centering
\thispagestyle{empty}

\begin{tikzpicture}[
  >=Latex,
  font=\scriptsize,
  box/.style={
    draw,
    rounded corners=2.5pt,
    align=center,
    inner sep=4.5pt,
    text width=0.82\textwidth
  },
  midbox/.style={
    draw,
    rounded corners=2.5pt,
    align=center,
    inner sep=3.8pt,
    text width=0.255\textwidth,
    font=\tiny
  },
  final/.style={
    draw,
    very thick,
    rounded corners=3pt,
    align=center,
    inner sep=5pt,
    text width=0.82\textwidth
  },
  arrow/.style={->, thick},
  dashedarrow/.style={->, thick, dashed},
  node distance=5mm
]

\node[box] (setup) {
\textbf{Farey scale and phase identity}\\[-1pt]
The parabolic domain \(L\) is organized by Farey cells with
\[
m_{p,q}=pq(p+q).
\]
For neighboring primitive directions,
\[
\alpha_{\lambda_1}+\alpha_{\lambda_2}-\alpha_{\lambda_{12}}
=
\frac{1}{m_{p,q}},
\qquad
\alpha_{(a,b)}=-\frac{ab}{a+b}.
\]
Thus the local lattice count is governed by fractional parts with period \(m_{p,q}\).
};

\node[box, below=of setup] (decomp) {
\textbf{Exact head--tail decomposition}\\[-1pt]
For
\[
E_L(n):=|nL\cap\mathbb Z^2|-\frac{10}{3}n^2,
\]
combining and rearranging the terms arising from the triangle cuts gives
\[
E_L(n)=
\bigl(H_0(n)+T_1(n)\bigr)+H_{12}(n)+R_{T_0}(n)+1.
\]
After the density cancellation \(T_2=T_0^{\mathrm{dens}}\), the problem reduces to
three second-layer packages.
};

\path (decomp.south west) -- (decomp.south east)
  coordinate[pos=0.10] (pkgL)
  coordinate[pos=0.50] (pkgC)
  coordinate[pos=0.90] (pkgR);

\node[midbox, anchor=north] (h0) at ($(pkgL)+(0,-8mm)$) {
\textbf{\(H_0+T_1\)}\\[-1pt]
Weighted-tail identity:
\[
H_0^{\mathrm{mean}}(n)+T_1(n)
=
2n\sum_{\substack{\gcd(p,q)=1\\m_{p,q}>n}}W_{p,q}.
\]
Head and tail are combined \emph{before} asymptotics.
};

\node[midbox, anchor=north] (h12) at ($(pkgC)+(0,-8mm)$) {
\textbf{\(H_{12}\)}\\[-1pt]
Decompose into:\\\ \\[-1pt]complete periods\\[-1pt]
+\\[-1pt]
moving prefixes\\[-1pt]
+\\[-1pt]
balanced complement.\\\ \\[-1pt]
The balanced part is controlled by the product-discrepancy lemma.
};

\node[midbox, anchor=north] (rt0) at ($(pkgR)+(0,-8mm)$) {
\textbf{\(R_{T_0}\)}\\[-1pt]
Threshold residual:
\[
R_{T_0}
=
-T_0+T_0^{\mathrm{dens}}.
\]
Finite strips give the edge term; the balanced complement is negligible.
};

\node[box, below=12mm of h12] (cancel) {
\textbf{Cancellation of the tropical-zeta bulk}\\[-1pt]
The cubic weights \(m_{p,q}=pq(p+q)\) produce an apparent \(N^{2/3}\)-scale bulk.
In the exact decomposition this bulk cancels between the reorganized packages.
After normalization by \(N^{3/2}\), only the second-layer Farey-edge contribution remains.
};

\node[box, below=of cancel] (constant) {
\textbf{Final algebra}\\[-1pt]
The surviving contribution reduces to
\[
\frac{32}{3}\zeta\!\left(-\frac12\right)
\sum_{p\ge1}\frac{\varphi(p)}{p^3}.
\]
Using
\[
\sum_{p\ge1}\frac{\varphi(p)}{p^3}
=
\frac{\zeta(2)}{\zeta(3)},
\]
one obtains the final constant.
};

\node[final, below=of constant] (theorem) {
\textbf{Averaged lattice count on \(L\)}\\[-1pt]
\[
N^{-3/2}\sum_{n\le N}
\left(
|nL\cap\mathbb Z^2|-\frac{10}{3}n^2
\right)
\longrightarrow
\frac{32}{3}\,
\zeta\!\left(-\frac12\right)
\frac{\zeta(2)}{\zeta(3)}.
\]
};

\draw[arrow] (setup) -- (decomp);

\draw[dashedarrow] (decomp.south) -- (h0.north);
\draw[dashedarrow] (decomp.south) -- (h12.north);
\draw[dashedarrow] (decomp.south) -- (rt0.north);

\draw[dashedarrow] (h0.south) -- ($(cancel.north west)+(1.6cm,0)$);
\draw[dashedarrow] (h12.south) -- (cancel.north);
\draw[dashedarrow] (rt0.south) -- ($(cancel.north east)+(-1.6cm,0)$);

\draw[arrow] (cancel) -- (constant);
\draw[arrow] (constant) -- (theorem);

\end{tikzpicture}

\caption{Summary of lattice-point counting on \(L\).}\label{fig_Lcounting}

\end{figure}

This is the first cancellation in the head--tail decomposition: the \(T_2\)-bulk is paired with
the triangular-density part of \(T_0\) before the averaging process begins.

This decomposition gives the structure of the proof. The tropical zeta
function first supplies the natural Farey scale \(m_{p,q}=pq(p+q)\), and
this scale is visible in the head and tail contributions of the Farey cells.
The exact head--tail identity then organizes the lattice count so that the
main bulk term is paired with its threshold-density counterpart before any
averaging is performed. After this pairing, the averaged discrepancy is
carried by three second-layer packages: the weighted tail \(H_0+T_1\), the
corrected complete-period contribution \(H_{12}\), and the finite-strip
threshold residual \(R_{T_0}\). Each package has its own limiting
coefficient. The final step is purely arithmetic: the strip constants
\(I_p\) and \(K_p\) fit together through the identity
\[
I_p=\frac{K_p}{p}-\frac{1}{3p}+\frac{1}{18p^2}
+\frac{4\zeta(-1/2)}{3p^{3/2}},
\]
and the Euler series
\[
\sum_{p\ge1}\frac{\varphi(p)}{p^s}
=
\frac{\zeta(s-1)}{\zeta(s)}
\]
collapses the combined coefficient to
\[
C_L=
\frac{32}{3}\zeta\!\left(-\frac12\right)
\frac{\zeta(2)}{\zeta(3)}.
\]
Thus the proof proceeds from the tropical Farey scale, to the exact
head--tail decomposition, to three second-layer limiting contributions, and
finally to the closed zeta-value coefficient.

\subsection{The three second-layer packages}

The detailed proofs are deferred to Appendix~E. The main text records the
three limiting contributions and the final cancellation among them.

\begin{proposition}[The weighted \(H_0+T_1\) contribution]\label{prop:H0T1}
One has
\begin{equation}\label{eq:H0T1limit}
N^{-3/2}\sum_{n\le N}(H_0(n)+T_1(n))
\longrightarrow
\frac83\frac{\zeta(3/2)}{\zeta(5/2)}.
\end{equation}
\end{proposition}

\begin{proposition}[The corrected \(H_{12}\) contribution]\label{prop:H12corr}
One has
\begin{equation}\label{eq:H12limit}
N^{-3/2}\sum_{n\le N}H_{12}(n)
\longrightarrow
-\frac49\frac{\zeta(5/2)}{\zeta(7/2)}
+8\sum_{p\ge1}\frac{\varphi(p)}{p^{3/2}}I_p,
\end{equation}
where
\begin{equation}\label{eq:Ipdef}
I_p
=
\frac89-\frac{2}{3p}+\frac{1}{18p^2}
+\frac{4}{3p^{3/2}}
\left(
\zeta\!\left(-\frac12\right)-\sum_{i=1}^{p-1}\sqrt i
\right).
\end{equation}
The sum is empty for \(p=1\).
\end{proposition}

\begin{proposition}[The tail-threshold contribution]\label{prop:RT0}
One has
\begin{equation}\label{eq:RT0limit}
N^{-3/2}\sum_{n\le N}R_{T_0}(n)
\longrightarrow
-8\sum_{p\ge1}\frac{\varphi(p)}{p^{5/2}}K_p,
\end{equation}
where
\begin{equation}\label{eq:Kpdef}
K_p
=
\frac13
\left(
\frac{8p}{3}-1-\frac4{\sqrt p}\sum_{i=0}^{p-1}\sqrt i
\right).
\end{equation}
\end{proposition}

The two strip constants are related by
\begin{equation}\label{eq:IpKp}
I_p
=
\frac{K_p}{p}
-\frac{1}{3p}
+\frac{1}{18p^2}
+\frac{4\zeta(-1/2)}{3p^{3/2}}.
\end{equation}
Therefore, using
\begin{equation}\label{eq:phis}
\sum_{p\ge1}\frac{\varphi(p)}{p^s}
=
\frac{\zeta(s-1)}{\zeta(s)},
\qquad
\Re(s)>2,
\end{equation}
we obtain
\begin{align*}
&\frac83\frac{\zeta(3/2)}{\zeta(5/2)}
-\frac49\frac{\zeta(5/2)}{\zeta(7/2)}
+8\sum_{p\ge1}\frac{\varphi(p)}{p^{3/2}}I_p
-8\sum_{p\ge1}\frac{\varphi(p)}{p^{5/2}}K_p \\
&\hspace{2cm}
=
\frac{32}{3}\zeta\!\left(-\frac12\right)
\sum_{p\ge1}\frac{\varphi(p)}{p^3}
=
\frac{32}{3}\zeta\!\left(-\frac12\right)
\frac{\zeta(2)}{\zeta(3)}.
\end{align*}
Together with \eqref{eq:safeLedger}, this proves Theorem~\ref{thm:mainL}.

\subsection{The column comparison}

The same lattice error can be computed by columns. With
\begin{equation}\label{eq:Scoldef}
S_{\mathrm{col}}(n)
=
\sum_{1\le i<n}\psi(2\sqrt{ni}),
\end{equation}
the column formula is
\begin{equation}\label{eq:columnFormula}
E_L(n)
=
-4S_{\mathrm{col}}(n)
+
8\zeta\!\left(-\frac12\right)\sqrt n
+
\frac{10}{3}
+
O(n^{-1/2}).
\end{equation}
Averaging \eqref{eq:columnFormula} and comparing with
Theorem~\ref{thm:mainL} gives \eqref{eq:triSaw}; the passage from triangular
to square sums gives \eqref{eq:squareSaw}. The proof of
\eqref{eq:columnFormula}, including the boundary strips \(i=0\) and \(j=0\),
is included in Appendix~E.2.

This completes the lattice-counting part of the parabolic model. The result
shows that \(L\) provides an exactly computable tropical analogue of the
Gauss circle problem for averaged integer dilates. The same
primitive-direction architecture appears: Farey neighbors, mediants,
caustic edges, and continued-fraction arithmetic. For \(L\), the edge sizes
are rational cubic weights \(pq(p+q)\), and this makes the complete
head--tail calculation possible. For the Euclidean disk, the same caustic
skeleton should carry quadratic-irrational size data, governed by square roots
and Diophantine approximation. The open problem is to carry out the analogous
program for the disk: start from the same structural decomposition, prove the
corresponding bulk cancellation, and compute the second-layer caustic terms
that control the circle discrepancy. The indexing of the sums should be the
same; only the tropical coefficients change.

In this sense, Theorem~\ref{thm:mainL} is more than an exact formula for a
special convex domain. It gives a model for a program: use tropical caustics
to decompose lattice-point error into Farey packages, locate the canonical
head--tail cancellations, and then read the signed discrepancy from the
remaining edge terms. The negative constant
\[
\frac{32}{3}\zeta\!\left(-\frac12\right)\frac{\zeta(2)}{\zeta(3)}
\]
is the arithmetic trace of this mechanism in the parabolic case. The Gauss
circle problem asks for the corresponding trace in the circular case.

\subsection{Concluding remarks}

The domain \(L\) is an example of what we call a
\textit{tropically rational} convex domain, meaning one whose tropical series
has rational coefficients. Such domains are preserved by
\(\mathrm{GL}(2,\Q)\) and by translations in \(\Q^2\). If the space of convex
domains is viewed as a polyhedral complex, with strata indexed by caustic
singularity types and coordinates given by tropical caustic-edge lengths, then
the tropically rational domains form a set of rational points. A broad source
of examples comes from univariate rational functions with rational
coefficients. Start with such a function \(h\), monotone on an interval
\([a,b]\) with \(a>0\). Choose an antiderivative \(f=\int h^{-1}\), take the
arc in \(\R^2\) given by its graph, and apply a
\(\mathrm{GL}(2,\Q)\)-transformation followed by a \(\Q^2\)-translation.
Convex domains whose boundaries decompose into such arcs are tropically
rational, and \(L\) belongs to this subfamily.

In Corollary~\ref{cor:saw} one sees the factor \(2\) under the fractional
part. It appears that this factor can be moved outside the limit. Although we
do not have a complete proof, this has been confirmed by extensive computer
experiments of Stanislav Shkolnikov, which also led to the following
empirical statements, recorded here as conjectures. First, for every real
\(\mu\), the limit\footnote{The convergence of such sequences is very slow
and non-monotone, with an oscillatory profile of Voronoi type.}
\[
\operatorname{saw}(\mu)
:=
\lim_{N\to\infty}
N^{-3/2}\sum_{1<k,l<N}\left(\{\mu\sqrt{kl}\}-\frac12\right)
\]
exists. Second, there is a universal constant \(S\) such that
\[
|\operatorname{saw}(\mu)|\le S
\qquad
(\mu\in\R).
\]
Third, for every odd integer \(d\),
\[
\operatorname{saw}(2d)=2\operatorname{saw}(d).
\]

The decomposition and regrouping scheme presented in this section applies,
in principle, to any domain whose minimal model is a lattice polygon. The
special feature of \(L\) is that all its tropical monomials are rational,
although non-integral except for four monomials. For the Gauss circle problem
this latter property fails: there are infinitely many monomials
\(\alpha_{(p,q)}=\sqrt{p^2+q^2}\) which are integral, and the corresponding
oscillations disappear. This may produce a further correction of a different
kind. Numerically, one observes that the errors for dilates of the disk tend
to stay negative, while positive errors have much smaller magnitude. A precise
form of this observation is the conjecture that the average of the positive
errors alone, with negative errors counted as zero, is \(o(\sqrt N)\). If this
holds, and if the scheme illustrated here for \(L\) can be carried out for the
disk with total averaged error of order \(\sqrt N\), it would give a very
strong averaged form of the Gauss circle estimate.\footnote{A nonaveraged version would also require bounds on the negative spikes; this would require horizontal cancellation of the oscillations across Farey cells.}

We have already mentioned that \(L\) and the disk are ``twins'' from the
combinatorial perspective. The previous paragraph describes one of their
number-theoretic deviations; another is the non-rationality of
\(\sqrt{p^2+q^2}\) when \(p^2+q^2\) is not a square. Geometrically, however,
the two domains are very similar. Both have \(C^1\)-smooth boundaries and are strictly convex; the boundary of
\(L\) fails to be \(C^2\) at four points and is \(C^\infty\) elsewhere. From the metric-asymptotic perspective used at the
decisive Abel summation step for the centered part of the head term \(H_0\),
the side length in the corner-cut process, before dilation, decays quadratically
in the cutting threshold for both domains. After dilation, this produces slow
side growth, which is asymptotically negligible on average after multiplication
by a zero-mean oscillatory term. This quadratic decay is a consequence of the
quadratic decay of the triangle size in terms of its Farey generation when,
starting from any pair, one always goes left or right. The phenomenon is quite
general for domains of bounded curvature: zero curvature at a boundary point
with rational tangent gives faster decay of the side parallel to the tangent,
whereas a corner gives slower decay and can therefore produce a higher-order
term in lattice-point counting.

Finally, we speculate that the scheme above is similar in spirit to
renormalization in quantum field theory.  The analogy suggests
looking for a geometric counterpart of quantum field theory in the passage
from polyhedral moment spaces to arbitrary convex domains.  Rational-slope
polygons are moment domains of symplectic toric surfaces, which, in physical
language, correspond to classical integrable systems.  Quantum integrable
systems lead naturally to quantum toric varieties: noncommutative-geometric
objects which, more concretely, may be viewed as fans and dual polyhedra
after the rationality condition on slopes has been dropped
\cite{KatzarkovLupercioMeerssemanVerjovsky2014NC,katzarkov2021quantum,KatzarkovLeeLupercioMeersseman2025Hodge}.
The next step is then to drop not only rationality but also polyhedrality,
and to pass from polyhedra to arbitrary convex domains.

What is missing is the geometric counterpart of such a generalized toric
symplectic space, with arbitrary convex bodies as moment domains.  The
tropical caustic should be part of this structure: it plays the role of a
generalized fan, records the elementary toric transitions, and organizes the
scale-by-scale cancellations in the exact head--tail decomposition.  Constructing this
framework would be desirable because lattice-point counting for convex
domains would then become a literal extension of Ehrhart theory, and the
structural results proved here might become consequences of a suitable
Hirzebruch--Riemann--Roch theorem. 

\appendix

\section{Basic geometry in dimension two}\label{app1}

This appendix records the elementary two-dimensional computations used in the
main text. We first compute the tropical zeta function of the simplest
minimal models, then prove the one-cut identity underlying the global
integral--boundary formula, and finally record the exact mean-value identity
used in the Farey/Hata analysis.

\subsection{A minimal model}

We begin with the rectangular support box determined by the horizontal and
vertical supporting lines. The global integral identity will be obtained by
starting from this box and inserting the remaining supporting lines one by
one. Each new supporting line cuts off a single triangular region, and the
next lemma computes the contribution of the initial rectangle.

Let
\[
R=[0,P]\times[0,Q],
\qquad P\ge Q>0,
\]
and define
\[
\rho_R(x,y):=\min(x,P-x,y,Q-y).
\]

\begin{lemma}[Rectangle identity]\label{lem:rectangle}
For \(\Re(s)>2\),
\[
s(s-1)\int_R \rho_R(x,y)^{\,s-2}\,dx\,dy
=
8\Bigl(\frac Q2\Bigr)^s
+
2s(P-Q)\Bigl(\frac Q2\Bigr)^{s-1}.
\]
\end{lemma}

\begin{proof}
For \(0\le t\le Q/2\), the superlevel set of \(\rho_R\) is
\[
\{\rho_R\ge t\}=[t,P-t]\times[t,Q-t],
\]
and therefore
\[
\Area\{\rho_R\ge t\}=(P-2t)(Q-2t).
\]
By the layer-cake formula,
\[
\int_R \rho_R(x,y)^{\,s-2}\,dx\,dy
=
(s-2)\int_0^{Q/2}t^{s-3}(P-2t)(Q-2t)\,dt.
\]
Expanding the integrand gives
\[
(s-2)\int_0^{Q/2}
t^{s-3}\bigl(PQ-2(P+Q)t+4t^2\bigr)\,dt.
\]
Integrating term by term and simplifying yields
\[
s(s-1)\int_R \rho_R(x,y)^{\,s-2}\,dx\,dy
=
8\Bigl(\frac Q2\Bigr)^s
+
2s(P-Q)\Bigl(\frac Q2\Bigr)^{s-1}.
\]
\end{proof}

We next isolate the local effect of inserting a single additional supporting
line. After a unimodular affine normalization, every such local move is
reduced to cutting the first quadrant by a line of the form \(x+y=\lambda\).
The following identity computes exactly the change of the integral under this
elementary operation; see Figure~\ref{fig:rectangularonecut3d}. Since the
statement is local, the same identity applies inside arbitrary minimal models.

\begin{figure}
  \hspace*{-15pt}
\begin{tikzpicture}[line join=bevel,x=-1,y=1.3,z=2]

\begin{scope}[scale=28]
    \begin{scope}

\coordinate (R1) at (-2,-3,0);
\coordinate (R2) at (2,-3,0);
\coordinate (R3) at (2,3,0);
\coordinate (R4) at (-2,3,0);

\coordinate (V1) at (0,1,2);
\coordinate (V2) at (0,-1,2);

\coordinate (P1) at (1,-3,0);
\coordinate (P2) at (2,-2,0);

\draw (V1)--(V2);

\draw (R1)--(V2)--(R2);
\draw (R3)--(V1)--(R4);

\draw[fill opacity=0.9,fill=black](P1) -- (P2)--(R2)--cycle;

\draw [fill opacity=0.7,fill=red!60](R3)--(V1)--(R4) -- cycle;
\draw [fill opacity=0.7,fill=red!60](R1)--(V2)--(V1)--(R4) -- cycle;
\draw [fill opacity=0.7,fill=red!60](R2)--(V2)--(V1)--(R3) -- cycle;
\draw [fill opacity=0.7,fill=red!60](R1)--(V2)--(R2) -- cycle;

\draw[dashed,red!60] (R1)--(R4);
\draw[dashed,red!60] (R3)--(R4);
\draw[dashed,red!60] (V1)--(R4);

\draw (R2)--(V2);

\end{scope}

\begin{scope}[xshift=9]

\coordinate (NR1) at (-2,-3,0);
\coordinate (NR2) at (2,-3,0);
\coordinate (NR3) at (2,3,0);
\coordinate (NR4) at (-2,3,0);

\coordinate (NV1) at (0,1,2);
\coordinate (NV2) at (0,-1,2);

\coordinate (NV3) at (1,-2,1);

\coordinate (NP1) at (1,-3,0);
\coordinate (NP2) at (2,-2,0);

\draw[thick] (NP1) -- (NP2);

\draw (NR1)--(NP1)--(NP2)--(NR3)--(NR4)--cycle;

\draw (NV1)--(NV2);

\draw (NR1)--(NV2)--(NV3);
\draw (NR3)--(NV1)--(NR4);

\draw [fill opacity=0.7,fill=red!60](NP1)--(NP2)--(NV3) -- cycle;
\draw [fill opacity=0.7,fill=red!60](NR3)--(NV1)--(NR4) -- cycle;
\draw [fill opacity=0.7,fill=red!60](NR1)--(NV2)--(NV1)--(NR4) -- cycle;
\draw [fill opacity=0.7,fill=red!60] (NR1)--(NV2)--(NV3)--(NP1)--cycle;
\draw [fill opacity=0.7,fill=red!60] (NR3)--(NV1)--(NV2)--(NV3)--(NP2)--cycle;

\draw[dashed,red!60] (NR1)--(NR4);
\draw[dashed,red!60] (NR3)--(NR4);
\draw[dashed,red!60] (NV1)--(NR4);

\draw(NV3)--(NV2);

\end{scope}

\begin{scope}[xshift=5]

    \coordinate (SP1) at (1,-3,0.3);
    \coordinate (SP2) at (2,-2,0.3);

    \coordinate (SV3) at (1,-2,1.3);

    \coordinate (SR2) at (2,-3,0.3);

    \draw (SP1)--(SP2);
    \draw (SP1)--(SV3);
    \draw (SP1)--(SR2);

    \draw (SP2)--(SV3);
    \draw (SP2)--(SR2);

    \draw (SV3)--(SR2);

\end{scope}

\draw[thick, dotted] (SV3)--(NV3);

\draw[thick,dotted] (SP1)--(NP1);
\draw[thick, dotted] (SP2)--(NP2);

\draw[thick, dotted] (SP2)--(P2);
\draw[thick, dotted] (SP1)--(P1);

\draw[thick, dotted] (SR2)--(R2);

\end{scope}

\end{tikzpicture}

\caption{The one-cut identity. Left: the graph of the tropical distance
function before the cut, with the triangle to be removed marked in dark.
Right: the graph after inserting the new supporting line. Center: the
unimodular tetrahedron whose base is the removed triangle and whose volume
computes the difference between the two undergraphs.}
\label{fig:rectangularonecut3d}
\end{figure}

\begin{lemma}[One-cut identity]\label{lem:onecut}
Let \(\lambda>0\), and define
\[
\phi_-(x,y):=\min(x,y)
\qquad\text{on } \{x\ge0,\ y\ge0\},
\]
and
\[
\phi_+(x,y):=\min(x,y,x+y-\lambda)
\qquad\text{on } \{x\ge0,\ y\ge0,\ x+y\ge\lambda\}.
\]
Then, for \(\Re(s)>2\),
\begin{multline*}
\int_{\{x+y<\lambda\}}\phi_-(x,y)^{\,s-2}\,dx\,dy \\
+
\int_{\{x\ge0,\ y\ge0,\ x+y\ge\lambda\}}
\Bigl(\phi_-(x,y)^{\,s-2}-\phi_+(x,y)^{\,s-2}\Bigr)\,dx\,dy
=
\frac{\lambda^s}{s(s-1)}.
\end{multline*}
\end{lemma}

\begin{proof}
For \(t\ge0\), the difference of the superlevel sets appearing on the
left-hand side is
\[
E_t:=\{x\ge t,\ y\ge t,\ x+y\le \lambda+t\}.
\]
After the change of variables \(x'=x-t\), \(y'=y-t\), this becomes
\[
\{x'\ge0,\ y'\ge0,\ x'+y'\le \lambda-t\}.
\]
Hence, for \(0\le t\le\lambda\),
\[
\Area(E_t)=\frac{(\lambda-t)^2}{2},
\]
while \(E_t=\varnothing\) for \(t>\lambda\). By the layer-cake formula, the
left-hand side equals
\[
(s-2)\int_0^\lambda t^{s-3}\frac{(\lambda-t)^2}{2}\,dt.
\]
Substituting \(t=\lambda u\) gives
\[
\frac{s-2}{2}\lambda^s\int_0^1u^{s-3}(1-u)^2\,du.
\]
Since
\[
\int_0^1u^{s-3}(1-u)^2\,du
=
B(s-2,3)
=
\frac{2}{s(s-1)(s-2)},
\]
we obtain \(\lambda^s/(s(s-1))\).
\end{proof}

\subsection{Tropical zeta functions of minimal models}\label{ss:tropmmzeta}

For a compact convex planar domain, the tropical wave front can terminate
either at a point or along a segment. The point case is the usual minimal
model from Definition~\ref{def:minimalmodel}: after translation and
rescaling, it is one of the sixteen reflexive polygons. For the computation
below, only the terminal time and the normalized lattice perimeter are needed.

\begin{lemma}[Point minimal model]\label{lem:point-minimal-model-zeta}
Assume that the maximal locus \(M_{\widehat\Omega}\) of the tropical distance
function consists of a single point. Let \(m_{\widehat\Omega}\) be the maximal
value of the tropical distance function. Then
\[
s(s-1)Z_{\widehat\Omega}(s)
=
\Length_{\Z}(\partial\widehat\Omega)\,
m_{\widehat\Omega}^{s-1}.
\]
\end{lemma}

\begin{proof}
Starting from the terminal point and reversing the tropical wave-front
evolution, the minimal model is obtained by a finite collection of unimodular
one-cuts. Lemma~\ref{lem:onecut} shows that each boundary edge contributes
its normalized lattice length times \(m_{\widehat\Omega}^{s-1}\) to
\(s(s-1)Z_{\widehat\Omega}(s)\). Summing over the boundary gives the formula.
\end{proof}

The final-segment case has one additional continuous parameter: the lattice
length of the terminal segment. Up to unimodular transformations,
translations, and rescaling, the final-segment minimal models fall into three
types. Their geometry and caustics are displayed in
Figure~\ref{fig:segmentwf}; the zeta functions are summarized immediately
afterwards.

\begin{figure}[htbp]
   \noindent\hspace{-37pt}
    \begin{tikzpicture}

    \draw (0,2.4) node {\bf Degenerate minimal-model types};
    \draw (0,1.9) node {
    For \(\ell,m>0\), \(k_1,k_2\in\Z\) such that
    \(\ell\geq m(|k_1-k_2|-1)\):
    };

    \begin{scope}[xshift=-20]

        \draw (-1.2,1.3) node {\((-\ell/2+mk_2,m)\)};
        \draw (-2,-1.3) node {\((-\ell/2-m(k_2+1),-m)\)};
        \draw (3.2,-1.3) node {\((\ell/2-mk_1,-m)\)};
        \draw (2.2,1.3) node {\((\ell/2+m(k_1+1),m)\)};

        \draw(-1,1)--(-1,0);
        \draw(-2,-1)--(-1,0);
        \draw(-1,0)--(2,0);
        \draw(2,0)--(2,1);
        \draw(2,0)--(3,-1);

        \draw[thick] (-1,1)--(2,1)--(3,-1)--(-2,-1)--(-1,1);

    \end{scope}

    \draw (0,-1.8) node {
    All caustic edges have weight \(2\), so no branching may occur.
    };

    \begin{scope}[yshift=-150]
    \draw(-6.2,3.2)--(6.2,3.2);

        \draw(-2,0)--(2,0);

        \draw(-2,0)--(-3,0);
        \draw(-2,0)--(-1,1);
        \draw(-2,0)--(-1,-1);

        \draw(2,0)--(3,0);
        \draw(2,0)--(1,1);
        \draw(2,0)--(1,-1);

        \draw[thick](-3,0)--(-1,-1)--(1,-1)--(3,0)--(1,1)--(-1,1)--(-3,0);

        \draw (-4.4,0) node {\((-\ell/2-m,0)\)};
        \draw (4.2,0) node {\((\ell/2+m,0)\)};
        \draw (-1.7,1.3) node {\((-\ell/2+mn_4,m)\)};
        \draw (-2,-1.3) node {\((-\ell/2+mn_3,-m)\)};
        \draw (1.7,1.3) node {\((\ell/2+mn_1,m)\)};
        \draw (2,-1.3) node {\((\ell/2+mn_2,-m)\)};

        \draw (0,2.9) node {\bf Branching minimal-model types};

        \draw (0,2.4) node {
        For \(\ell,m>0\), \(n_1,n_2,n_3,n_4\in\Z\) such that
        };

        \draw (0,1.9) node {
        \(n_3+n_4,-n_1-n_2\geq -2\) and
        \(\ell\geq m(n_1-n_4),\,m(n_2-n_3)\):
        };

        \draw(0,-1.8) node {
        The weight of the left horizontal edge is \(2+n_3+n_4\),
        };

        \draw(0,-2.2) node {
        the weight of the right horizontal edge is \(2-n_1-n_2\), and the central one has weight \(2\);
        };

        \draw(0,-2.6) node {
        the remaining four edges have weight \(1\), and so they can branch.
        };

    \end{scope}

    \begin{scope}[yshift=-325]
    \draw(-6.2,3.2)--(6.2,3.2);

    \draw(0,2.9) node {\bf Mixed minimal-model types};
    \draw(0,2.4) node {
    For \(\ell,m>0\) and \(k,n_1,n_2\in\Z\) such that:
    };
    \draw(0,2) node {
    \(n_1+n_2\leq 2\) and
    \(\ell\geq m(n_1-k),\,m(n_2+k+1)\):
    };

    \begin{scope}[xshift=-20]

        \draw(-1,0)--(-2,-1);
        \draw(-1,0)--(-1,1);
        \draw(-1,0)--(2,0);

        \draw(2,0)--(3,0);
        \draw(2,0)--(1,1);
        \draw(2,0)--(1,-1);

        \draw[thick](-2,-1)--(-1,-1)--(1,-1)--(3,0)--(1,1)--(-1,1)--(-2,-1);

        \draw (-1.2,1.3) node {\((-\ell/2+mk,m)\)};
        \draw (-2,-1.3) node {\((-\ell/2-m(k+1),-m)\)};

        \draw (1.7,1.3) node {\((\ell/2+mn_1,m)\)};
        \draw (2,-1.3) node {\((\ell/2+mn_2,-m)\)};

        \draw (4.2,0) node {\((\ell/2+m,0)\)};

    \end{scope}

    \draw(0,-1.8) node {
    The two edges on the left, and the central edge have weight \(2\),
    };

    \draw(0,-2.2) node {
    \(2-n_2-n_1\) is the weight of the right horizontal edge,
    };

    \draw(0,-2.6) node {
    the remaining edges have weight \(1\), and so they can branch.
    };

    \draw(-6.2,-2.85)--(6.2,-2.85);
    \end{scope}

    \end{tikzpicture}

    \caption{\small Schematic representation of final-segment minimal models
    and their caustics. Degenerate minimal models are equivalent if and only
    if they have the same value of \(|k_1-k_2|\). Branching minimal models
    are equivalent if and only if their triples
    \((n_3+n_4,n_1+n_2,n_1-n_4)\) agree; if one of the caustic weights
    \(2+n_3+n_4\) or \(2-n_1-n_2\) vanishes, the corresponding edge and
    vertex disappear from the picture. In the mixed type, the complete
    invariant is \((n_1-k,n_1+n_2)\), and if \(2-n_1-n_2\) vanishes, the
    right horizontal edge disappears.}
    \label{fig:segmentwf}
\end{figure}

\begin{proposition}[Final-segment minimal models]
\label{prop:segment-minimal-models}
Let \(\widehat\Omega\) be a minimal model whose tropical wave front collapses
to a segment of lattice length \(\ell>0\) at time \(m>0\). Then
\[
Z_{\widehat\Omega}(s)
=
\frac{m^{s-1}}{s(s-1)}
\bigl(2\ell s+k m\bigr),
\]
where
\[
k=
\begin{cases}
4, & \text{degenerate type},\\
4+n_1+n_2-n_3-n_4, & \text{branching type},\\
4+n_1+n_2, & \text{mixed type}.
\end{cases}
\]
\end{proposition}

\begin{proof}
In the final-segment case, the rectangular contribution gives
\[
\frac{m^{s-1}}{s(s-1)}\,2\ell s,
\]
as in Lemma~\ref{lem:rectangle}. The transverse part is a finite sum of
one-cut contributions, computed by Lemma~\ref{lem:onecut}. The normalized
total of these transverse contributions is precisely the integer \(k\) listed
above.
\end{proof}

The point and final-segment cases combine into the following uniform formula.

\begin{proposition}\label{prop:minmodzeta}
For a planar minimal model \(\widehat\Omega\),
\[
s(s-1)Z_{\widehat\Omega}(s)
=
m_{\widehat\Omega}^{s-1}
\bigl(
2\ell_{\widehat\Omega}s
+
k_{\widehat\Omega}m_{\widehat\Omega}
\bigr),
\]
where \(m_{\widehat\Omega}\) is the maximal value of the tropical distance
function, \(\ell_{\widehat\Omega}\) is the lattice length of the maximal locus
\(M_{\widehat\Omega}\), with \(\ell_{\widehat\Omega}=0\) in the point case,
and
\[
k_{\widehat\Omega}
=
m_{\widehat\Omega}^{-1}
\bigl(
\Length_{\Z}(\partial\widehat\Omega)
-
2\ell_{\widehat\Omega}
\bigr).
\]
\end{proposition}

When the relevant edge-length inequalities are sharp, the toric surface
defined by the dual fan of \(\widehat\Omega\) has at worst \(A\)-type
singularities. In that case one can speak about the canonical class, and
\(k_{\widehat\Omega}\) corresponds to its self-intersection contribution.

As a corollary of Proposition~\ref{prop:minmodzeta}, together with the
triangle-cutting procedure of Subsection~\ref{subsec:integral-boundary}, we
obtain the residues at \(s=0\) and \(s=1\) for rational-slope polygons.

\begin{theorem}[Residues at \(0\) and \(1\)]\label{thm:res0and1}
For a rational-slope polygon \(\Omega\),
\[
\Res_{s=1} Z_\Omega(s)=\Length_{\Z}(\partial\Omega).
\]
In addition, if the dual fan of \(\Omega\) defines a compact toric surface
\(X\) which is either smooth or has only \(A_n\)-singularities, then
\[
\Res_{s=0} Z_\Omega(s)=-K^2,
\]
where \(K\) denotes the canonical class of \(X\).
\end{theorem}

\begin{proof}
The residue at \(s=1\) follows from the boundary term in the
integral--boundary identity and the explicit minimal-model contribution in
Proposition~\ref{prop:minmodzeta}. In the smooth or \(A_n\)-singular toric
case, the normalized constant term at \(s=0\) agrees with the toric
self-intersection of the canonical class. The sign convention is the one in
which the residue gives \(-K^2\).
\end{proof}

Thus, in the rational-slope polygonal case, the residue at \(s=1\) is the
symplectic area of the anticanonical class, while the residue at \(s=0\) is
the negative self-intersection of the canonical class. This motivates, but
does not prove, a possible extension of these interpretations beyond the
toric setting.

\begin{remark}
The type of the minimal model gives a \(\mathrm{GL}(2,\Z)\)-invariant
stratification of the space of compact convex domains, up to translation and
rescaling. Degenerate final-segment models form finite-dimensional strata,
whereas models with weight-one caustic edges admit further unimodular corner
cuts and hence carry infinitely many cutting parameters. The finite-codimension
strata are precisely those in which only finitely many such cuts are allowed.
This viewpoint suggests a natural wall-crossing picture for families of
convex domains, but we do not use it in the sequel.
\end{remark}

\begin{remark}
Theorem~\ref{thm:res0and1} suggests a generalization of the canonical
self-intersection: for a convex domain for which the tropical zeta function
admits a meromorphic continuation to \(s=0\), one may regard
\[
-\Res_{s=0}Z_\Omega(s)
\]
as an analogue of \(K^2\), even when no classical toric surface is available.
For the special domain \(L\) of Subsection~\ref{ss:specialdomain}, the formula
for \(Z_L(s)\) gives
\[
\Res_{s=0}Z_L(s)
=
-\left(
8-4\frac{\zeta_{\mathrm{SU}(3)}(0)}{\zeta(0)}
\right)
=
-\left(8+\frac83\right)
=
-\frac{32}{3}.
\]
This raises the natural question of whether the value \(-32/3\) is universal
for domains with smooth boundary. The same fraction also appears in \(C_L\)
in Theorem~\ref{thm:mainL}.
\end{remark}

\subsection{The integral--boundary identity}\label{subsec:integral-boundary}

We now assemble the global identity. Starting from the minimal model
\(\widehat\Omega\), we insert the remaining supporting lines in
Stern--Brocot order. The minimal-model formula gives the initial
contribution, while the one-cut identity shows that each inserted support
line subtracts exactly one term of the boundary Dirichlet series.

\begin{figure}
    \hspace*{-1.8cm}
    \begin{tikzpicture}

        \begin{scope}[xshift=-70,yshift=-40,scale=0.9]

        \fill[gray!65](1,1)--(-1,1)--(-3,-1)--(-1,-3)--(1,-1)--(1,1);
        \draw [red,line width=0.4mm,domain=0:90] plot ({2*cos(\x)-1}, {2*sin(\x)-1});
        \draw[line width=0.4mm, red](-1,1)--(-3,-1)--(-1,-3)--(1,-1);

        \end{scope}

        \begin{scope}[xshift=-260,yshift=-80, scale=1.8]

        \fill[gray!65](1,-0.1715)--(-0.1715,1)--(-1,1)--(-3,-1)--(-1,-3)--(1,-1)--(1,-0.1715);

        \begin{scope}[xshift=2,yshift=2]
        \fill[gray!15](1,-0.1715)--(-0.1715,1)--(1,1)--(1,-0.1715);
        \end{scope}

        \draw [red,line width=0.4mm,domain=0:90] plot ({2*cos(\x)-1}, {2*sin(\x)-1});
        \draw[line width=0.4mm, red](-1,1)--(-3,-1)--(-1,-3)--(1,-1);

        \end{scope}

        \begin{scope}[yshift=-240,xshift=-60, scale=3.6]

          \fill[gray!65](1,-0.5279)--(0.6437,0.1847)--(0.1847,0.6437)--(-0.5279,1)--(-1,1)--(-3,-1)--(-1,-3)--(1,-1)--(1,-0.5279);

          \begin{scope}[xshift=0.8,yshift=1.6]
          \fill[gray!15](1,-0.5279)--(0.6437,0.1847)--(1,-0.1715)--(1,-0.5279);
          \end{scope}

          \begin{scope}[xshift=1.6,yshift=0.8]
          \fill[gray!15](-0.5279,1)--(0.1847,0.6437)--(-0.1715,1)--(-0.5279,1);
          \end{scope}

          \draw [red,line width=0.4mm,domain=0:90] plot ({2*cos(\x)-1}, {2*sin(\x)-1});
          \draw[line width=0.4mm, red](-1,1)--(-3,-1)--(-1,-3)--(1,-1);

        \end{scope}

        \draw (-2.65,-2) node {\(\Omega^{(0)}\)};

        \draw (-10.8,-3.3) node {\(\Omega^{(1)}\)};

        \draw (-5.5,-11.9) node {\(\Omega^{(2)}\)};

        \draw [->,thick] (-3.2,-2.1) -- (-10,-3.3);

        \draw [->,thick] (-10.8,-3.9)--(-6,-11.5);

        \begin{scope}[very thick,decoration={
    markings,
    mark=at position 0.5 with {\arrow{>}}}
    ]

        \draw [thick,path fading=south,postaction={decorate}] (-6.3,-12)--(-15,-18);

    \end{scope}

    \end{tikzpicture}
    \caption{Successive unimodular corner cuts starting from a pentagonal
    minimal model \(\widehat\Omega=\Omega^{(0)}\). The red curve is
    \(\partial\Omega\), and the gray polygons \(\Omega^{(N)}\) decrease to
    \(\Omega\). Each cut inserts a mediant supporting line and contributes one
    term to the boundary Dirichlet series.}
    \label{fig:cuttingtrianglespentagon}
\end{figure}

\begin{proof}[Proof of Theorem~\ref{thm:integral-boundary}]
\label{proof:thm:integral-boundary}
Let \(\widehat\Omega\) be the minimal model of \(\Omega\). By construction,
\(\Omega\) is obtained from \(\widehat\Omega\) by a sequence of unimodular
corner cuts in Stern--Brocot order. Let \(\Omega^{(N)}\) be the polygon
obtained after the first \(N\) generations of cuts, with
\[
\widehat\Omega=\Omega^{(0)}\supset \Omega^{(1)}\supset \cdots \supset \Omega,
\qquad
\Omega^{(N)}\downarrow \Omega.
\]

At each step, one inserts the mediant supporting line between two neighboring
supporting lines with inward primitive normals \((a,b)\) and \((c,d)\), where
\(ad-bc=1\). After an affine change of coordinates in \(\mathrm{SL}(2,\Z)\),
the two old supporting lines become
\[
x=0,\qquad y=0,
\]
and the new supporting line becomes
\[
x+y=\lambda,
\qquad
\lambda=
\bigl|
\gamma_{a,b}+\gamma_{c,d}-\gamma_{a+c,b+d}
\bigr|.
\]
Because the change of coordinates lies in \(\mathrm{SL}(2,\Z)\), it preserves
area. The defining property of the minimal model guarantees that this planar
corner cut corresponds, on the graph of the tropical distance function, to the
local model of Lemma~\ref{lem:onecut}. Hence each cut decreases the integral
by
\[
\frac{1}{s(s-1)}
\bigl|
\gamma_{a,b}+\gamma_{c,d}-\gamma_{a+c,b+d}
\bigr|^s.
\]

Therefore, after \(N\) steps,
\[
s(s-1)Z_{\Omega^{(N)}}(s)
=
s(s-1)Z_{\widehat\Omega}(s)
-
\sum_{\Delta\in\mathcal T_N}
\bigl(\sqrt{2\,\Area(\Delta)}\bigr)^s,
\]
where \(\mathcal T_N\) is the finite set of support triangles removed up to
stage \(N\).

As \(N\to\infty\), the domains \(\Omega^{(N)}\) decrease to \(\Omega\), and
the corresponding tropical distance functions decrease pointwise to
\(\rho_\Omega\). Extending all integrands by zero outside \(\Omega^{(N)}\),
the dominated convergence theorem applies for \(\Re(s)>2\), since
\[
0\le
\rho_{\Omega^{(N)}}^{\,\Re(s)-2}
\le
\rho_{\widehat\Omega}^{\,\Re(s)-2}
\in L^1(\widehat\Omega).
\]
Thus
\[
Z_{\Omega^{(N)}}(s)\longrightarrow Z_\Omega(s).
\]
On the other hand, for \(\Re(s)>2\), the finite sums over \(\mathcal T_N\)
converge absolutely to \(F_{\partial\Omega}(s)\). Passing to the limit gives
\[
s(s-1)Z_\Omega(s)
=
s(s-1)Z_{\widehat\Omega}(s)-F_{\partial\Omega}(s).
\]
Equivalently,
\[
s(s-1)Z_\Omega(s)
=
-
F_{\partial\Omega}(s)+H_{\widehat\Omega}(s),
\]
where
\[
H_{\widehat\Omega}(s):=s(s-1)Z_{\widehat\Omega}(s).
\]
\end{proof}

\begin{figure}
    \centering
    \includegraphics[width=\linewidth]{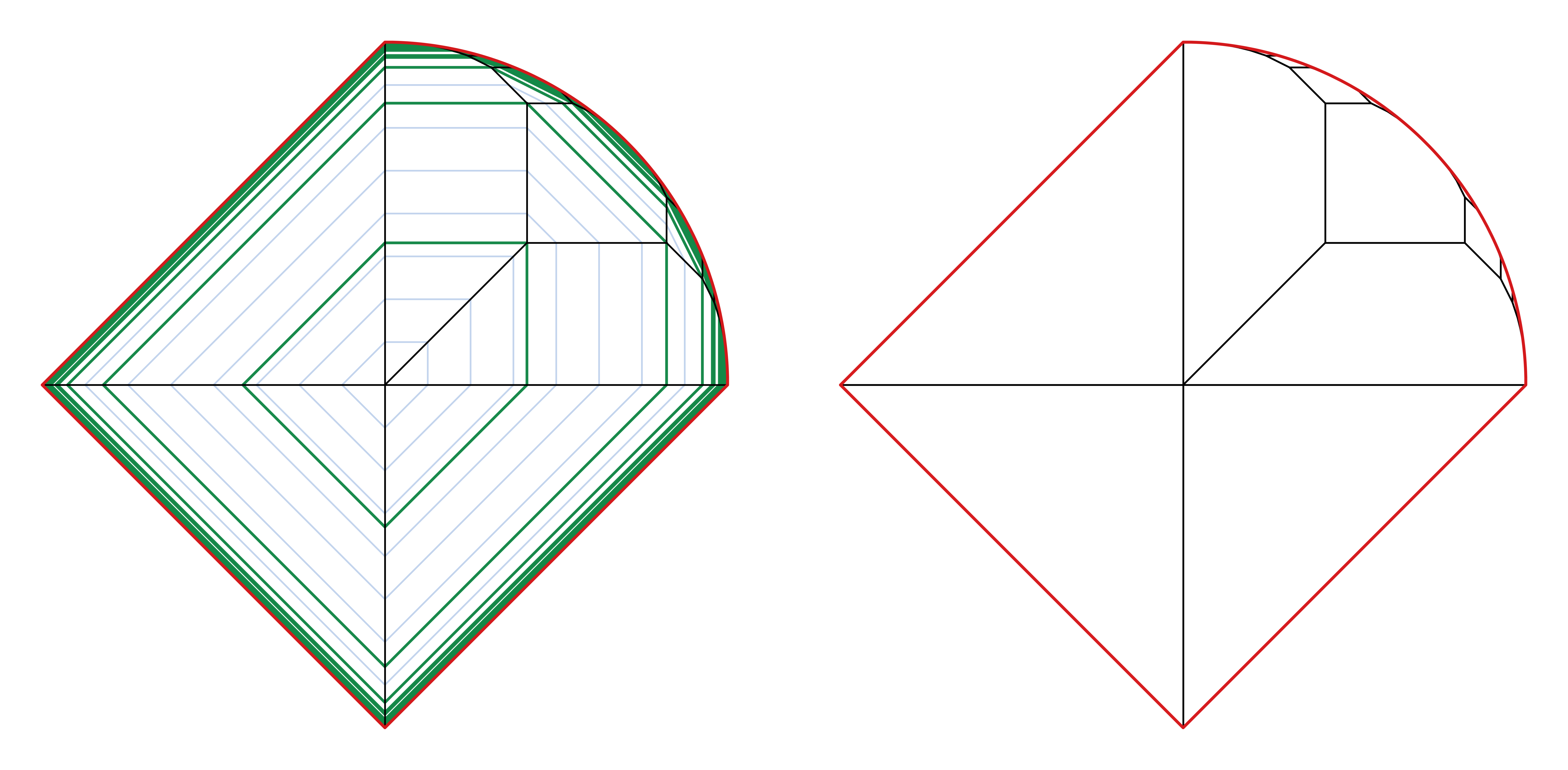}
    \caption{Tropical wave fronts and caustic for a convex planar domain with
    pentagonal minimal model of Figure~\ref{fig:cuttingtrianglespentagon}. The
    red curve is \(\partial\Omega\). The nested blue and green polygons are the
    tropical wave fronts \(\Omega_t=\{\rho_\Omega\ge t\}\). On the right, the
    black graph is the tropical caustic \(\mathcal K_\Omega\), the corner locus
    of \(\rho_\Omega\), whose vertices mark the critical times of the flow.}
    \label{fig:pentagonalwavefront}
\end{figure}

In particular, if the minimal model is the rectangle
\(\widehat\Omega=[0,P]\times[0,Q]\), with \(P\ge Q>0\), then
\[
H_{\widehat\Omega}(s)
=
8\Bigl(\frac Q2\Bigr)^s
+
2s(P-Q)\Bigl(\frac Q2\Bigr)^{s-1},
\]
and therefore
\[
s(s-1)Z_\Omega(s)
=
-F_{\partial\Omega}(s)+H_{\widehat\Omega}(s).
\]
Thus, in dimension two, the interior zeta function is determined by the
boundary Dirichlet series up to an explicit holomorphic correction term
coming from the minimal model. All nontrivial singular behavior of
\(Z_\Omega(s)\) is already encoded by the boundary series.

\begin{figure}
    \centering
    \includegraphics[width=0.8\linewidth]{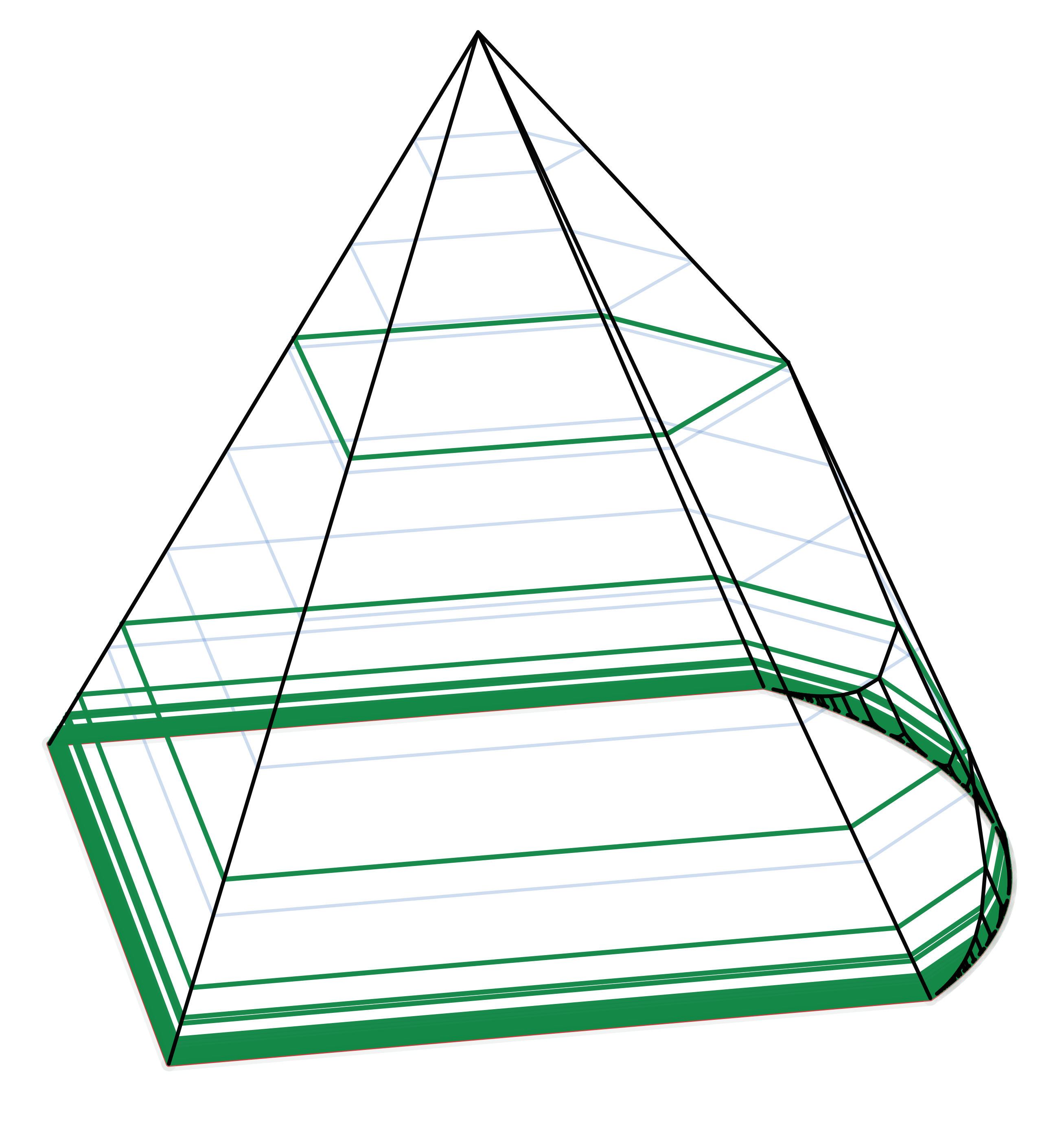}
    \caption{Three-dimensional lift of the tropical flow for the same domain
    as in Figures~\ref{fig:cuttingtrianglespentagon} and
    \ref{fig:pentagonalwavefront}. This surface is the graph
    \(z=\rho_\Omega(x,y)\). Its horizontal sections are the tropical wave
    fronts \(\Omega_t\), and its nonsmooth ridge set projects to the tropical
    caustic.}
    \label{fig:pentagonalseries}
\end{figure}

\begin{remark}
The same triangle-cutting procedure also underlies the lattice-counting
calculation of Section~\ref{sec_6}. There the individual Farey triangles are
counted with a coherent half-open convention and then reorganized into
head--tail packages. The outcome is that the tropical residue supplies the
natural Farey scale, while the averaged discrepancy is determined by the
second-layer terms left after the head--tail cancellation.
\end{remark}

\begin{remark}[Domains with prescribed rightmost pole]
Dropping the smoothness assumption on the boundary, one can construct a
domain whose boundary series has rightmost pole at any prescribed
\(\alpha\in(0,1)\). Start with the square
\[
[-2\zeta(\alpha^{-1}),\,2\zeta(\alpha^{-1})]^2
\]
and cut off successive unimodular triangles of sizes
\[
1,\quad 2^{-\alpha^{-1}},\quad 3^{-\alpha^{-1}},\quad\ldots
\]
using primitive normals \((1,1),(1,2),(1,3),\ldots\). The \(n\)-th cut is
possible because
\[
(n+1)^{-\alpha^{-1}}<n^{-\alpha^{-1}}
\qquad\text{and}\qquad
\zeta(\alpha^{-1})>\sum_{k=1}^{n}k^{-\alpha^{-1}}.
\]
The resulting convex domain \(D_\alpha\) satisfies
\[
F_{\partial D_\alpha}(s)=\zeta(\alpha^{-1}s),
\]
and hence its boundary series is meromorphic on \(\C\) with a unique pole at
\(s=\alpha\).
\end{remark}

\begin{figure}[t]
    \centering
    \includegraphics[width=\textwidth]{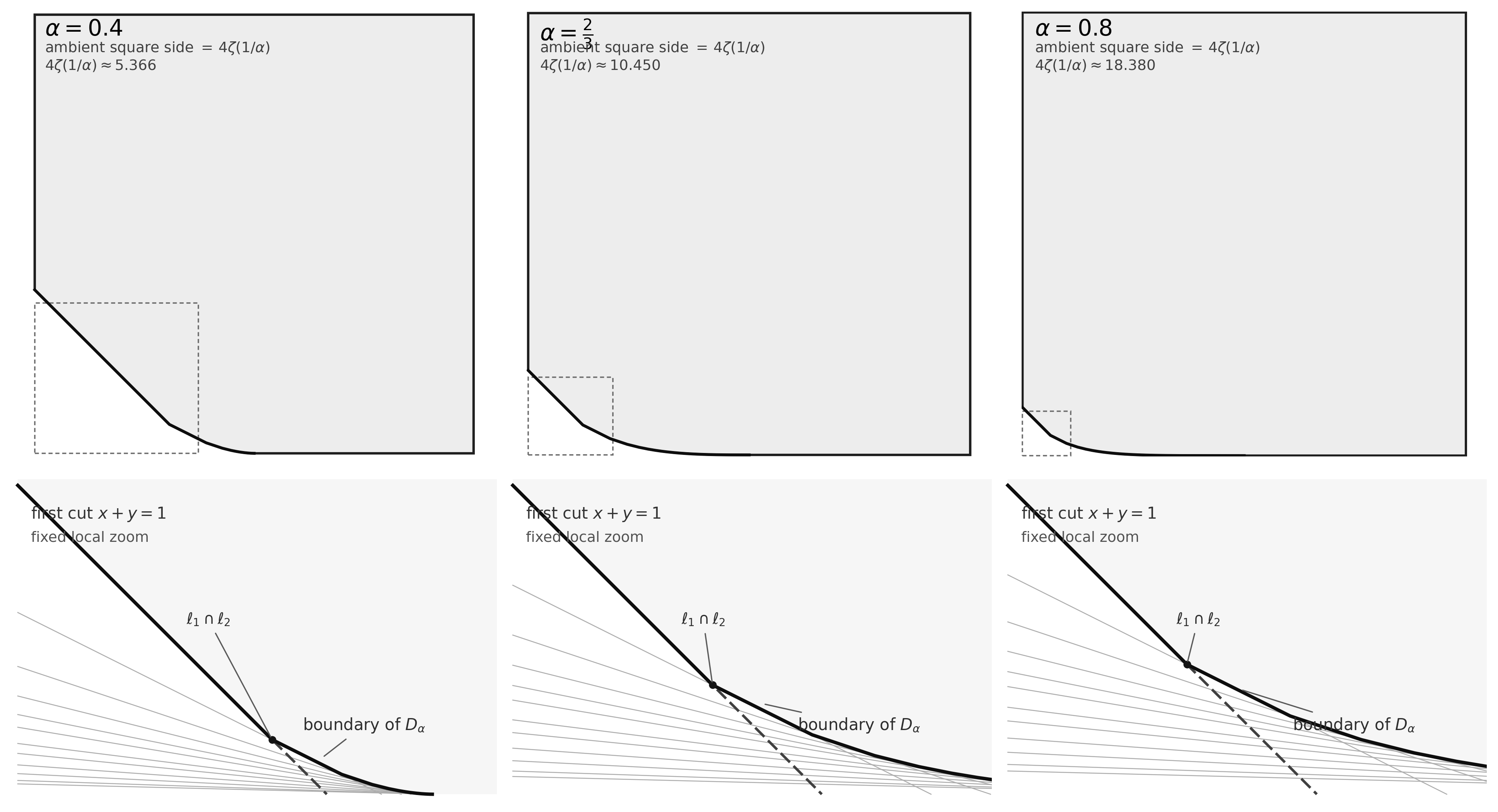}
    \caption{\textbf{Convex domains with prescribed rightmost pole.}
    For \(\alpha\in(0,1)\), the domain \(D_\alpha\) is obtained by successive
    corner cuts with primitive outward normals \((1,n)\), where the \(n\)-th
    cut has size \(n^{-1/\alpha}\). Equivalently, near the distinguished
    corner the boundary is the lower envelope of the support lines
    \(x+ny=\sum_{k=1}^{n}k^{-1/\alpha}\).
    The associated boundary series is
    \(F_{\partial D_\alpha}(s)=\zeta(\alpha^{-1}s)\), whose unique pole occurs
    at \(s=\alpha\).}
    \label{fig:Dalpha-prescribed-pole}
\end{figure}

\begin{remark}[Continuity and residues in families]
It is useful to distinguish the tropical zeta function from the local
coordinates used to compute it. The function
\[
Z_\Omega(s)=\int_\Omega \rho_\Omega(x)^{s-2}\,dx\,dy
\]
is defined intrinsically in terms of the tropical distance function. In its
half-plane of absolute convergence, this integral varies continuously under
standard convergence of convex domains, for instance under Hausdorff
convergence inside bounded nondegenerate families.

This continuity concerns only the original half-plane of convergence. It does
not imply continuity, in families, of the residues of meromorphic
continuations. For example,
\[
f_\varepsilon(s)=\frac{1}{s-1+\varepsilon}
\]
is holomorphic at \(s=1\) for every \(\varepsilon>0\), so
\[
\Res_{s=1}f_\varepsilon=0.
\]
However, as \(\varepsilon\to0^+\), the functions \(f_\varepsilon\) converge
locally uniformly on compact subsets of \(\Re(s)>1\) to \(1/(s-1)\), which
has residue \(1\) at \(s=1\). Thus convergence in the original half-plane
does not by itself imply continuity of polar data after meromorphic
continuation.
\end{remark}

\begin{remark}[The inverse problem]
It is natural to ask to what extent the tropical zeta function recovers the
domain. The map
\[
\Omega\longmapsto Z_\Omega
\]
is not expected to be globally injective on the full space of convex domains:
the cutting description suggests possible rearrangements of boundary pieces
which preserve the scalar Dirichlet series while changing the domain.

The analytic case may be more rigid. Suppose that \(\Omega_1\) and
\(\Omega_2\) have real-analytic strictly convex boundaries with nonvanishing
curvature, and that
\[
Z_{\Omega_1}(s)=Z_{\Omega_2}(s).
\]
Must \(\Omega_1\) and \(\Omega_2\) agree up to translation and the natural
\(\mathrm{GL}(2,\Z)\)-action?

We leave this question open. The residue at \(s=\frac23\) alone cannot
determine the domain, since it is only proportional to the equiaffine length:
\[
\Res_{s=2/3}Z_\Omega(s)
=
C_{\mathrm{aff}}\,
\Length_{\mathrm{equiaffine}}(\partial\Omega).
\]
Any possible rigidity must come from the full zeta function, not only from
its leading smooth residue.
\end{remark}

\subsection{Exact mean-value identity}

The following lemma is a standard two-point second-order mean-value formula.
We include the proof to keep the normalization of the Hata coefficient
explicit.

\begin{lemma}\label{lem:mv}
Let \(f\in C^2([0,1])\), let \(x\neq y\) in \([0,1]\), and let
\(\lambda\in(0,1)\). Put
\[
m=\lambda x+(1-\lambda)y.
\]
Then there exists \(\xi\) between \(x\) and \(y\) such that
\[
f(m)-\lambda f(x)-(1-\lambda)f(y)
=
-\frac12\lambda(1-\lambda)f''(\xi)(x-y)^2.
\]
\end{lemma}

\begin{proof}
Define
\[
\varphi(t):=f\bigl(y+t(x-y)\bigr),
\qquad t\in[0,1].
\]
Then \(\varphi\in C^2([0,1])\), and
\[
\varphi(\lambda)=f(m),
\qquad
\varphi(1)=f(x),
\qquad
\varphi(0)=f(y).
\]
Thus it is enough to prove that there exists \(\eta\in(0,1)\) such that
\[
\varphi(\lambda)-\lambda\varphi(1)-(1-\lambda)\varphi(0)
=
-\frac12\lambda(1-\lambda)\varphi''(\eta).
\]

Let
\[
L(t):=\varphi(0)+t\bigl(\varphi(1)-\varphi(0)\bigr),
\]
and set
\[
A:=\varphi(\lambda)-L(\lambda).
\]
Consider
\[
\psi(t)
:=
\varphi(t)-L(t)
-
A\,\frac{t(t-1)}{\lambda(\lambda-1)}.
\]
Then
\[
\psi(0)=\psi(\lambda)=\psi(1)=0.
\]
By Rolle's theorem, there exist \(u_1\in(0,\lambda)\) and
\(u_2\in(\lambda,1)\) such that
\[
\psi'(u_1)=\psi'(u_2)=0.
\]
Applying Rolle's theorem once more to \(\psi'\), we obtain
\(\eta\in(u_1,u_2)\subset(0,1)\) such that
\[
\psi''(\eta)=0.
\]
Since \(L''=0\) and \(\bigl(t(t-1)\bigr)''=2\), this gives
\[
0
=
\psi''(\eta)
=
\varphi''(\eta)-A\,\frac{2}{\lambda(\lambda-1)}.
\]
Hence
\[
A
=
\frac12\lambda(\lambda-1)\varphi''(\eta)
=
-\frac12\lambda(1-\lambda)\varphi''(\eta).
\]
Finally,
\[
\varphi''(t)
=
f''\bigl(y+t(x-y)\bigr)(x-y)^2.
\]
With
\[
\xi:=y+\eta(x-y),
\]
we obtain the desired identity.
\end{proof}

\begin{proof}[Proof of Lemma~\ref{lemma:exactT}]\label{proof:lemma:exactT}
Apply Lemma~\ref{lem:mv} with
\[
x=\frac ab,
\qquad
y=\frac cd,
\qquad
\lambda=\frac{b}{b+d}.
\]
Since \(ad-bc=1\), we have
\[
x-y=\frac{ad-bc}{bd}=\frac1{bd}\neq0.
\]
Moreover,
\[
m
=
\lambda x+(1-\lambda)y
=
\frac{b}{b+d}\frac ab
+
\frac{d}{b+d}\frac cd
=
\frac{a+c}{b+d}
=
\mu_I.
\]
Hence
\[
c_I(f)
=
f(\mu_I)
-
\frac{b}{b+d}f\!\left(\frac ab\right)
-
\frac{d}{b+d}f\!\left(\frac cd\right)
=
-\frac12\lambda(1-\lambda)f''(\xi_I)(x-y)^2
\]
for some \(\xi_I\) between \(a/b\) and \(c/d\). Multiplying by \(b+d\), we
obtain
\[
(b+d)c_I(f)
=
-\frac12(b+d)\lambda(1-\lambda)f''(\xi_I)(x-y)^2.
\]
Now
\[
\lambda(1-\lambda)=\frac{bd}{(b+d)^2},
\qquad
(x-y)^2=\frac{(ad-bc)^2}{b^2d^2}=\frac1{b^2d^2}.
\]
Substituting gives
\[
(b+d)c_I(f)
=
-\frac{f''(\xi_I)}{2\,bd(b+d)},
\]
as claimed.
\end{proof}

\section{Complex, symplectic, and tropical geometry}\label{app2}

This appendix recalls the geometric background used in the main text.  We
review the elementary topology of compact smooth toric surfaces, the
symplectic interpretation of moment polygons and corner cuts, and the
tropical-optical description of wave fronts and caustics.  The final two
subsections apply this background to the short-time behavior of the lattice
perimeter of tropical wave fronts.

\subsection{Topology of complex toric surfaces}\label{ssec:topologycomplex}

Compact smooth toric surfaces provide a tame and well-understood class of
closed complex manifolds.\footnote{A more general term used in algebraic
geometry is ``toric variety'', which can have arbitrary dimension and
singularities.  A detailed exposition, including topological aspects, can be
found in \cite{cox2024toric}.}
The term ``toric'' comes from their defining property of being equivariant
compactifications of the complex algebraic torus \((\C^*)^2\), where
\(\C^*\) denotes the multiplicative group of nonzero complex numbers.  A
toric surface is characterized by its fan, represented by a collection of
rational-slope rays.  In the compact case the fan is complete, meaning that
the rays cover the plane by their cones.  In the nonsingular case it is
unimodular, meaning that the primitive vectors spanning two consecutive rays
form a basis of the square lattice.

\begin{example}\label{ex_toric}
The standard examples are:
\begin{enumerate}
\item The fan generated by
\[
(-1,0),\qquad (0,-1),\qquad (1,1),
\]
which gives the projective plane \(\C P^2\).
\item The fans generated by
\[
(-1,0),\qquad (0,-1),\qquad (0,1),\qquad (1,d),
\]
which give the Hirzebruch surfaces, that is, \(\C P^1\)-bundles over
\(\C P^1\).  They are nontrivial when \(d\neq0\).
\end{enumerate}
\end{example}

Let \(X\) denote a compact smooth toric surface.  Then \((\C^*)^2\) is a
dense open subset of \(X\), and its complement is a union of boundary
divisors
\[
\{D_j\}.
\]
Each \(D_j\) is a copy of the Riemann sphere \(\C P^1\), and the boundary
divisors are in one-to-one correspondence with the rays of the fan of \(X\).

The fan can be reconstructed dynamically.  Take a nonzero real vector
\((a,b)\), and consider the path in \((\C^*)^2\) parametrized by
\[
t\in(1,+\infty)\longmapsto (t^a,t^b).
\]
As \(t\to+\infty\), this path has a unique limit point in \(X\).  This point
belongs to a unique boundary divisor precisely when \((a,b)\) lies on the
ray of the fan corresponding to that divisor.  More generally, if \((a,b)\)
lies in the cone between two adjacent rays, the limit point is the unique
intersection point of the two boundary divisors corresponding to those rays.

Two consecutive boundary divisors intersect transversely at the corresponding
torus-fixed point, with intersection multiplicity \(+1\).  These torus-fixed
points are exactly the zero-dimensional orbits of the torus action.  All
other orbits, apart from the dense torus itself, are copies of \(\C^*\); each
is obtained from a boundary divisor by removing its two fixed points.

The second cohomology of \(X\) is generated by the classes \([D_j]\), with
the two linear relations encoded by the fan:
\[
\sum_j v_j\otimes [D_j]=0,
\]
where \(v_j\) is the primitive vector spanning the ray corresponding to
\(D_j\).  The odd cohomology groups vanish.  Thus the intersection form is
determined by the self-intersections \([D_j]^2\).  These are read directly
from the fan: if \(v_{j-1}\) and \(v_{j+1}\) are the primitive generators of
the two rays adjacent to \(v_j\), with the rays indexed cyclically, then
\[
[D_j]^2=\det(v_{j-1},v_{j+1}).
\]

Blowing up a point is the basic birational transformation of complex
surfaces.  Abstractly, one replaces a point by the projectivization of the
corresponding tangent space.  Locally, this is the blow-up of \(\C^2\) at the
origin, realized as the total space of the tautological line bundle over
\(\C P^1\).  In other words, the point is replaced by a Riemann sphere whose
self-intersection is \(-1\).

For toric surfaces, blowing up a torus-fixed point has a simple fan
description.  Suppose the fixed point is the intersection of two boundary
divisors \(D_j\) and \(D_{j+1}\), corresponding to adjacent rays spanned by
primitive vectors \(v_j\) and \(v_{j+1}\).  Then the blow-up is obtained by
adding the new ray spanned by
\[
v_j+v_{j+1}.
\]
The corresponding boundary divisor has self-intersection \(-1\).  Conversely,
contracting this divisor is the toric blow-down.  For example,
Example~\ref{ex_toric}(1) is obtained from Example~\ref{ex_toric}(2) with
\(d=1\) by deleting the ray spanned by \((0,1)\).

The second cohomology of a compact smooth complex surface contains the
canonical class.\footnote{It is often useful to work with singular surfaces as
well.  For isolated surface singularities, the canonical class is well defined
when the singular points are of ADE type.  Only \(A\)-type singularities
appear for toric surfaces, and in this case the canonical class can again be
represented by the negative of the sum of all boundary divisors.}
Topologically, this is the first Chern class of the cotangent bundle.  In the
toric case,
\[
K_X=-\sum_j [D_j].
\]
Equivalently, the logarithmic two-form
\[
d\log z_1\wedge d\log z_2
\]
on \((\C^*)^2\) extends to \(X\) as a meromorphic form with simple poles
along all boundary divisors.  For us, the self-intersection \(K_X^2\) appears
in two related ways: it is the negative of the residue at \(s=0\) of the
tropical zeta function of a polygon dual to the fan, and it is also the
time-derivative of the lattice perimeter of the tropical wave front between
critical times.

\subsection{Symplectic toric surfaces}\label{ssec:sympltor}

We now pass from fans to their dual polygons.  A choice of such a polygon,
up to translation, corresponds to a numerical class in the real second
cohomology of the toric surface which evaluates positively on each boundary
divisor.  These positive numbers are precisely the lattice lengths of the sides
of the polygon.  Each side is orthogonal to a ray of the fan and is the
moment-map image of the corresponding boundary divisor.

For example, the dual polygons of Example~\ref{ex_toric}(1) are the standard
unimodular triangles with vertices
\[
(0,0),\qquad (a,0),\qquad (0,a),
\qquad a>0.
\]
They form a ray in the one-dimensional second cohomology of \(\C P^2\).
For Example~\ref{ex_toric}(2), one obtains a fixed combinatorial type of
trapezoid for each \(d\), with two independent length parameters.  The
corresponding cone of cohomology classes is the K\"ahler cone.

Each class in the interior of the K\"ahler cone is represented by a symplectic
form \(\omega\), that is, a closed nondegenerate \(2\)-form.  In the toric
setting this form can be chosen invariant under the compact torus \(T^2\).
The resulting \(T^2\)-action is Hamiltonian and gives a momentum map
\[
\mu:X\longrightarrow(\mathfrak t^2)^*,
\]
where \((\mathfrak t^2)^*\cong\R^2\) is the dual of the Lie algebra of
\(T^2\).  The defining equation is
\[
d\langle\mu,\xi\rangle=-\iota_{V_\xi}\omega,
\]
where \(V_\xi\) is the vector field generated by
\(\xi\in\mathfrak t^2\).  For an introduction to symplectic toric manifolds,
including their algebro-geometric aspects, see \cite{da2003symplectic}.

\begin{figure}[t]
\centering
\begin{tikzpicture}[>=Latex, line join=round, line cap=round]

\definecolor{accentred}{RGB}{210,55,45}

\newcommand{\spherecomponent}[2]{%
  \pgfmathsetmacro{\midang}{(#1+#2)/2}
  \pgfmathsetmacro{\span}{#1-#2}
  \pgfmathsetmacro{\bulge}{0.36 + 0.0052*\span}
  \pgfmathsetmacro{\inset}{0.48 + 0.0028*\span}
  \pgfmathsetmacro{\hlx}{0.085 + 0.0007*\span}
  \pgfmathsetmacro{\hly}{0.13 + 0.0010*\span}
  \pgfmathsetmacro{\shx}{0.16 + 0.0011*\span}
  \pgfmathsetmacro{\shy}{0.24 + 0.0016*\span}
  \coordinate (P) at ({#1}:\R);
  \coordinate (Q) at ({#2}:\R);
  \coordinate (O) at ({\midang}:{\R+\bulge});
  \coordinate (I) at ({\midang}:{\R-\inset});

  \path[shade, inner color=gray!8, outer color=gray!52, shading angle={\midang+138}]
    (P) .. controls ($(P)!0.56!(O)$) and ($(Q)!0.56!(O)$) .. (Q)
        .. controls ($(Q)!0.58!(I)$) and ($(P)!0.58!(I)$) .. cycle;
  \draw[gray!58, line width=0.42pt]
    (P) .. controls ($(P)!0.56!(O)$) and ($(Q)!0.56!(O)$) .. (Q)
        .. controls ($(Q)!0.58!(I)$) and ($(P)!0.58!(I)$) .. cycle;

  \begin{scope}
    \clip
      (P) .. controls ($(P)!0.56!(O)$) and ($(Q)!0.56!(O)$) .. (Q)
          .. controls ($(Q)!0.58!(I)$) and ($(P)!0.58!(I)$) .. cycle;
    \path[fill=white, opacity=0.22]
      ({\midang+22}:{\R+0.16}) ellipse [x radius=\hlx, y radius=\hly, rotate={\midang-18}];
    \path[fill=black, opacity=0.08]
      ({\midang-150}:{\R+0.07}) ellipse [x radius=\shx, y radius=\shy, rotate={\midang-14}];
    \path[fill=white, opacity=0.07]
      ({\midang+2}:{\R+0.02}) ellipse
      [x radius={0.10 + 0.0008*\span}, y radius={0.30 + 0.0015*\span}, rotate={\midang-90}];
  \end{scope}
}

\begin{scope}[shift={(0,0)}]
  \def\R{1.72}
  \spherecomponent{90.0}{48.50}
  \spherecomponent{48.50}{-5.88}
  \spherecomponent{-5.88}{-61.26}
  \spherecomponent{-61.26}{-123.43}
  \spherecomponent{-123.43}{-181.27}
  \spherecomponent{-181.27}{-270.0}

  \fill[accentred, even odd rule, opacity=0.08]
    (0,0.01) ellipse (0.86 and 0.33)
    (0,0.01) ellipse (0.38 and 0.14);
  \draw[accentred, line width=1.35pt] (0,0.01) ellipse (0.86 and 0.33);
  \draw[accentred, line width=1.10pt] (0,0.01) ellipse (0.38 and 0.14);
\end{scope}

\begin{scope}[shift={(8.3,-0.18)}]
  \coordinate (A) at (-1.55,0.55);
  \coordinate (B) at (-0.90,1.55);
  \coordinate (C) at (0.65,1.75);
  \coordinate (D) at (1.75,0.60);
  \coordinate (E) at (1.20,-1.10);
  \coordinate (F) at (-0.35,-1.70);

  \path[shade, left color=gray!56, right color=gray!18, middle color=gray!34]
    (A)--(B)--(C)--(D)--(E)--(F)--cycle;
  \draw[draw=gray!58, line width=0.45pt] (A)--(B)--(C)--(D)--(E)--(F)--cycle;

  \draw[->, line width=0.9pt] ($ (A)!0.50!(B) $) -- ++(-0.84,0.55);
  \draw[->, line width=0.9pt] ($ (B)!0.50!(C) $) -- ++(-0.18,1.00);
  \draw[->, line width=0.9pt] ($ (C)!0.50!(D) $) -- ++(0.82,0.78);
  \draw[->, line width=0.9pt] ($ (D)!0.50!(E) $) -- ++(0.95,-0.30);
  \draw[->, line width=0.9pt] ($ (E)!0.50!(F) $) -- ++(0.35,-0.92);
  \draw[->, line width=0.9pt] ($ (F)!0.50!(A) $) -- ++(-0.92,-0.38);

  \fill[accentred] (0.10,0.16) circle (2.6pt);
  \node[above right] at (1.55,2.25) {$\R^{2}$};
\end{scope}

\draw[->, draw=accentred, line width=0.95pt]
  (0.92,-0.01) to[out=-16,in=172] node[midway, below, sloped, accentred] {$\mu$} (8.01,-0.03);

\end{tikzpicture}
\caption{\textbf{Momentum map geometry for a compact symplectic toric
four-manifold.}  Over an interior point of the Delzant polygon the fiber is a
Lagrangian torus \(T^2\).  Over the boundary the torus degenerates, and the
inverse image of the polygonal boundary is a cyclic chain of \(2\)-spheres,
one for each facet.  The outward arrows indicate the primitive facet normals.}
\label{fig:moment-map-toric}
\end{figure}

By Delzant's theorem \cite{delzant1988hamiltoniens}, compact symplectic
toric four-manifolds are classified by Delzant polygons.  Thus the dual
polygons of the fans above are precisely the images of the momentum maps
associated with invariant symplectic forms.  The lattice length of a side of
the polygon is the integral of \(\omega\) over the corresponding boundary
divisor; it is the symplectic area of that divisor.  The Euclidean area of the
moment polygon is the symplectic four-volume of the surface, namely the total
integral
\[
\int_X \frac{\omega^2}{2}.
\]

In symplectic geometry, blow-ups have sizes.  If one removes a standard open
ball of radius \(\sqrt R\) in a Darboux chart and collapses its boundary
along the Hopf fibration, the exceptional divisor has symplectic area \(R\).
For toric manifolds this operation is visible in the moment polygon as a
unimodular corner cut of size \(R\).

More concretely, a unimodular corner cut removes a triangle from a vertex of
the polygon by a line orthogonal to
\[
\nu_1+\nu_2,
\]
where \(\nu_1\) and \(\nu_2\) are the inward primitive normals to the two
sides adjacent to the corner.  Iterating such cuts is central in this paper:
the tropical zeta function is expressed as a Dirichlet generating series of
the sizes of these cuts.

The local model is the standard momentum map
\[
(z_1,z_2)\in\C^2
\longmapsto
\frac12(|z_1|^2,|z_2|^2)\in\R^2,
\]
corresponding to the invariant symplectic form
\[
\frac{i}{2}\,dz_1\wedge d\bar z_1+\frac{i}{2}\,dz_2\wedge d\bar z_2.
\]
Although \(\C^2\) is noncompact, this model explains the polygonal operation:
the ball
\[
|z_1|^2+|z_2|^2<R
\]
projects to the triangle with vertices
\[
(0,0),\qquad (R,0),\qquad (0,R).
\]
Thus a symplectic blow-up of size \(R\) corresponds exactly to cutting off a
unimodular triangle of size \(R\) in the moment polygon.

\subsection{Essentials of tropical optics}\label{ssec:tropop}

We now recall the basic facts about planar tropical wave fronts and caustics
used in the paper. For general background on tropical geometry, see
\cite{BriefIntroTropGeom}; for a systematic treatment of planar tropical wave
fronts and caustics, see
\cite{MikhalkinShkolnikov2023WaveFrontsCaustics}.

The tropical wave front of a compact convex domain \(\Omega\) at time
\(t\ge0\) is
\[
\Omega_t=\rho_\Omega^{-1}[t,m_\Omega],
\]
where
\[
m_\Omega=\max_\Omega \rho_\Omega.
\]
For \(t\in(0,m_\Omega)\), the set \(\Omega_t\) is a polygon with
rational-slope sides.  Hence it has a dual fan, and therefore an associated
toric surface \(X_t\).  If \(\Omega_t\) has non-unimodular corners, the
surface \(X_t\) may have singularities; in the stable case these are of
\(A_n\)-type.  Such a singularity is resolved by a chain of \(n\) boundary
divisors, each of self-intersection \(-2\).  The canonical class \(K_t\) and
the symplectic form \(\omega_t\) still make sense in this setting.

For non-critical values of \(t\), the canonical evolution equation is
\begin{equation}\label{eq:canevol}
\frac{d}{dt}[\omega_t]=K_t.
\end{equation}
This identity is valid between critical times, and it has higher-dimensional
analogues.\footnote{In some conventions an additional factor of \(2\pi\)
appears.  This depends only on the normalization of symplectic area.  The
formula first appears in \cite{Shkolnikov17Theis}, and the singular case is
treated in detail in \cite{MikhalkinShkolnikov2023WaveFrontsCaustics}.}

A time is called critical if the combinatorics of \(\Omega_t\), equivalently
the dual fan, changes at that time.  The set of critical times has at most one
accumulation point, namely \(0\).  At a critical time the event is a
blow-down: a ray is erased from the fan, and the resulting point of the toric
surface is smooth.  The canonical evolution equation gives a quick derivation
of the basic surface-volume identity for tropical wave fronts.

\begin{proposition}\label{prop:twf_surfacevol}
Let \(P_\Omega(t)\) denote the lattice surface volume of \(\Omega_t\), as in
Definitions~\ref{def:surfacevolume} and~\ref{def:wfsurfacevolume}.  For every
non-critical value of \(t\),
\[
\frac{d}{dt}\operatorname{Vol}(\Omega_t)=-P_\Omega(t).
\]
Equivalently, the identity holds piecewise in \(t\), with the natural
one-sided interpretation at critical times.
\end{proposition}

\begin{proof}\label{proof:prop:twf_surfacevol}
On a symplectic manifold \(X\) with symplectic form \(\omega\), the
symplectic volume of a \(2k\)-dimensional homology class \(\beta\), with
Poincar\'e dual cohomology class \(\alpha\), is obtained by evaluating
\[
\frac{1}{k!}[\omega]^k\alpha
\]
on the fundamental class of \(X\).  This is the same as integrating
\(\frac{1}{k!}\omega^k\) over a representative cycle of \(\beta\).

The volume of the moment polytope \(\Omega_t\subset\R^n\) can be written as
\[
\int_X \frac{1}{n!}[\omega_t]^n.
\]
Taking the derivative and using \eqref{eq:canevol}, we get
\[
\frac{d}{dt}\operatorname{Vol}(\Omega_t)
=
\int_X \frac{1}{(n-1)!}[\omega_t]^{n-1}\cdot K_t.
\]
Since \(-K_t\) is represented by the sum of boundary divisors, the last
integral is minus the total lattice surface volume of the facets of
\(\Omega_t\).  Hence
\[
\frac{d}{dt}\operatorname{Vol}(\Omega_t)=-P_\Omega(t).
\]
\end{proof}

\begin{corollary}
In dimension two, differentiating the lattice perimeter \(P_\Omega(t)\)
between critical times gives
\[
\frac{d}{dt}P_\Omega(t)=-K_t^2.
\]
Thus negative self-intersection of the canonical class corresponds to
increasing lattice perimeter, zero self-intersection to constant lattice
perimeter, and positive self-intersection to decreasing lattice perimeter.
This interpretation is valid on intervals where the dual fan, and hence the
cohomology group containing \(K_t\), is fixed.
\end{corollary}

Critical events occur at non-maximal vertices of the tropical caustic
\[
\mathcal K_\Omega\subset\Omega^\circ,
\]
that is, at vertices outside
\[
M_\Omega=\rho_\Omega^{-1}(m_\Omega).
\]
The caustic \(\mathcal K_\Omega\) is the corner locus of the tropical distance
function \(\rho_\Omega\): it is the set where \(\rho_\Omega\) is not locally
linear.  All non-maximal vertices of the caustic are trivalent, and the value
of \(\rho_\Omega\) at such a vertex is the corresponding critical time.

Starting from a vertex on \(M_\Omega\) and following a shortest path along
the caustic toward the boundary, one travels precisely \(m_\Omega\) units of
lattice distance.  Conversely, if one starts from \(M_\Omega\) and moves with
primitive velocity along a non-maximal edge of the caustic for time
\(m_\Omega\), the endpoint is a vertex of the minimal model
\(\widehat\Omega\).  In this way the caustic gives a canonical recipe for
carving \(\Omega\) out of \(\widehat\Omega\) by a sequence of corner cuts, or
equivalently by symplectic blow-ups.  Each cut has size equal to the
corresponding critical time.  This is the same procedure used in
Subsection~\ref{subsec:integral-boundary}, now viewed directly through the
evolution of the caustic.

We now spell out the stable non-unimodular case.  Suppose that \(\Omega\)
has only \(A_n\)-type corners.  The caustic \(\mathcal K_\Omega\) is traced
by the vertices of the wave fronts \(\Omega_t\) as \(t\) varies.  In dimension
two these vertices move with primitive velocity.  If the final locus
\(M_\Omega\) is a segment, that segment is also included in the caustic and
is assigned weight \(2\), because the two gradients of \(\rho_\Omega\) on
opposite sides differ by twice a primitive vector.

Similarly, an edge of \(\mathcal K_\Omega\) traced by an \(A_n\)-type vertex
of \(\Omega_t\) has weight \(n+1\).  These weights make the balancing
condition hold at every vertex of the caustic.  An \(A_0\)-vertex is a
unimodular vertex; it corresponds to a smooth point of the toric surface and
traces a weight-one edge of the caustic.

At a critical time, a disjoint collection of edges of \(\Omega_{t-\varepsilon}\)
may be contracted when passing to \(\Omega_{t+\varepsilon}\).  Only one of
the two endpoints of a contracting edge can be non-unimodular.  If the edge
connects an \(A_0\)-vertex to an \(A_n\)-vertex, there are two equivalent
geometric descriptions.  One may view the event as a single weighted
symplectic blow-up, replacing a round Darboux ball by an ellipsoid with ratio
\(1:(n+1)\).  For the tropical zeta function, however, the relevant
description is the unimodular one: the event is represented by \(n+1\)
successive unimodular cuts of the same size \(c>0\), always proceeding in
one cyclic direction.  Hence the corresponding contribution to the boundary
Dirichlet series is
\[
(n+1)c^s.
\]

\subsection{Asymptotics of lattice perimeter of tropical wave fronts}
\label{ssec:pertwf}

Let \(\Omega\subset\R^2\) be a compact convex domain with nonempty interior.
For \(t\ge0\), recall that
\[
\Omega_t:=\{x\in\Omega:\rho_\Omega(x)\ge t\}
\]
is the tropical wave front at time \(t\).  In this subsection we prepare the
proof of the asymptotic behavior of its lattice perimeter as \(t\to0^+\).  We
use the same smoothness assumptions as in the residue theorem at
\(s=\frac23\): the boundary is \(C^1\)-smooth, and every arc of
\(\partial\Omega\setminus\partial\widehat\Omega\) is \(C^3\) with
nonvanishing curvature.

The tropical caustic of \(\Omega\) is
\[
\mathcal K_\Omega\subset \Omega^\circ.
\]
It is the corner locus of the tropical distance function \(\rho_\Omega\).
Equivalently, it is the support of the distributional Laplacian of
\(\rho_\Omega\).  This distribution is the leading term of the push-forward
of symplectic area inside the algebraic torus under the tropicalization of a
holomorphic curve \cite{KalininShkolnikov2018IntroductionTropicalSeries}.
Each edge of the caustic carries a natural weight, namely the lattice length
between the two gradients of \(\rho_\Omega\) on the two sides of that edge.
The balancing condition at vertices follows from this gradient interpretation.

Let
\[
m_\Omega:=\max_{\Omega}\rho_\Omega
\]
and
\[
M_\Omega:=\{x\in \Omega^\circ:\rho_\Omega(x)=m_\Omega\}.
\]
Then \(M_\Omega\subset \mathcal K_\Omega\).  The set \(M_\Omega\) is either a
point or a segment; in the segment case it has weight \(2\).  We say that
\(\mathcal K_\Omega\) is \emph{unbranched} if it has no vertices outside
\(M_\Omega\).

Equivalently, the minimal model \(\widehat\Omega\supset\Omega\) can be
characterized by the following three conditions:
\begin{itemize}
\item the tropical caustic of \(\widehat\Omega\) is unbranched;
\item \(m_{\widehat\Omega}=m_\Omega\) and
\(M_{\widehat\Omega}=M_\Omega\);
\item \(\rho_{\widehat\Omega}\) and \(\rho_\Omega\) coincide in a neighborhood
of \(M_\Omega\).
\end{itemize}
This agrees with Definition~\ref{def:minimalmodel}.  Indeed, the second and
third conditions are immediate from that definition.  The first follows
because any branching outside \(M_\Omega\) would come from a supporting term
which is irrelevant near \(M_\Omega\), and hence removable from the minimal
description.

A planar tropical caustic of a compact convex domain has no cycles.\footnote{One
may nevertheless have fake cycles if the boundary of \(\Omega\) is excluded
from the caustic and the caustic passes through non-\(A_n\)-type
non-unimodular corners.  This is precisely the situation in which more than
one caustic edge may end at the same boundary corner.}
Thus it is a tree.  Its branching profile is governed by Farey pairs along
each weight-one edge of the caustic of the minimal model.  For example, the
minimal models of the unit disk and the special domain \(L\) from
Subsection~\ref{ss:specialdomain} are both the square \([-1,1]^2\), while the
minimal model of the triangle with vertices
\[
(0,0),\qquad (1,0),\qquad (0,1)
\]
is the triangle itself.

The minimal model \(\widehat\Omega\) is uniquely determined by \(\Omega\).
Conversely, \(\Omega\) may be obtained from \(\widehat\Omega\) by an infinite
sequence of unimodular corner cuts, equivalently by symplectic blow-ups.  These
cut-off triangles are in one-to-one correspondence with the vertices of
\[
\mathcal K_\Omega\setminus M_\Omega.
\]
If \(v\) is the caustic vertex corresponding to such a triangle \(\Delta\),
then
\[
\Area(\Delta)=\frac12\,\rho_\Omega(v)^2.
\]

The complement
\[
\widehat\Omega\setminus\Omega^\circ
\]
is a finite union of \(\Gamma\)-triangles.  Each such region is bounded by two
straight segments and one convex arc
\[
\Gamma\subset\partial\Omega.
\]
Each \(\Gamma\)-triangle is decomposed into a family \(T(\Gamma)\) of
unimodular triangles, each corresponding to one corner cut along the arc
\(\Gamma\).  The associated boundary series is
\[
F_\Gamma(s):=
\sum_{\Delta\in T(\Gamma)}
\bigl(2\,\Area(\Delta)\bigr)^{s/2}.
\]

\begin{proposition}[Blow-up formula]
\label{prop:blowup}
One has
\[
Z_\Omega(s)
=
Z_{\widehat\Omega}(s)
-
\frac{1}{s(s-1)}
\sum_{\Gamma}
F_\Gamma(s),
\]
where the finite summation is taken over the connected components
\(\Gamma\) of \(\partial\Omega\setminus\partial\widehat\Omega\).
\end{proposition}

It is useful to examine the summation protocol defining \(F_\Gamma(s)\) when
\(\Gamma\) is smooth and convex.  After an affine unimodular change of
coordinates, the two straight sides of the corresponding \(\Gamma\)-triangle
may be placed on the positive coordinate axes, meeting at the origin.  In
this normalization, the unimodular triangles in \(T(\Gamma)\) are indexed by
the free monoid
\[
\mathrm{SL}_2^+(\Z)
:=
\left\{
A\in\mathrm{SL}(2,\Z): A_{ij}\ge0
\right\}.
\]
The rule is the Stern--Brocot rule: whenever an edge with neighboring
gradients \(v_1\) and \(v_2\) of \(\rho_\Omega\) branches, the new gradient
visible from it is
\[
v_1+v_2.
\]
Thus the two new edges adjacent to the new vertex have neighboring gradients
\[
v_1,\ v_1+v_2
\qquad\text{and}\qquad
v_1+v_2,\ v_2.
\]
Equivalently, the corresponding \(2\times2\) matrix \((v_1,v_2)\) is
multiplied by
\[
\begin{pmatrix}
1&0\\1&1
\end{pmatrix}
\qquad\text{or}\qquad
\begin{pmatrix}
1&1\\0&1
\end{pmatrix},
\]
the free generators of \(\mathrm{SL}_2^+(\Z)\).

Recall that \(\Omega\) is obtained from its minimal model \(\widehat\Omega\)
by a sequence of unimodular corner cuts; see
Subsection~\ref{subsec:integral-boundary}.  At each step one removes a
unimodular triangle of size \(r>0\), meaning a triangle congruent, under an
\(\mathrm{SL}(2,\Z)\) transformation and a translation, to the triangle with
vertices
\[
(0,0),\qquad (r,0),\qquad (0,r).
\]
Invariantly, for a triangle \(\Delta\) we define
\[
\operatorname{size}(\Delta):=\sqrt{2\,\Area(\Delta)}.
\]
This is consistent with the geometric interpretation of the boundary
summands in Subsection~\ref{sseq:geommeaning}.

We use superscripts for partial-cut models and subscripts for wave fronts:
\(\Omega^t\) is obtained by performing all cuts of size at least \(t\), while
\[
\Omega_t=\{\rho_\Omega\ge t\}
\]
is the actual tropical wave front.  The multiset of sizes of triangles in the
cutting process
\[
\widehat\Omega\rightsquigarrow\Omega
\]
has at most one accumulation point, namely \(0\), and all sizes are at most
\(m_\Omega\).  Thus
\[
\Omega^t=\widehat\Omega
\qquad\text{for } t>m_\Omega.
\]

\begin{lemma}
For \(0\le u<t\le m_\Omega\), the dual fan of \((\Omega^{\,t})_u\) is
constant.
\end{lemma}

Here \((\Omega^{\,t})_u\) denotes the tropical wave front at time \(u\) of the
partial-cut domain \(\Omega^t\).  Let \(K_t\) be the canonical class and
\(\omega_t\) the symplectic form of the toric surface with moment domain
\(\Omega^t\).  By the canonical evolution equation \eqref{eq:canevol}, the
class of the symplectic form corresponding to \((\Omega^{\,t})_u\), for
\(0\le u<t\), is
\[
[\omega_t]+uK_t.
\]
Therefore
\[
\Length_{\Z}\bigl(\partial(\Omega^{\,t})_u\bigr)
=
\Length_{\Z}(\partial\Omega^{\,t})-uK_t^2.
\]

We also need the compatibility relation between partial-cut models and
actual wave fronts.

\begin{lemma}
For every \(u>0\),
\[
(\Omega^{\,u})_u=\Omega_u.
\]
\end{lemma}

Now define the triangle-size counting function by
\[
A_\Omega(t):=
\#\{\Delta:\operatorname{size}(\Delta)\ge t\},
\]
where the count ranges over the triangles in the cutting process
\(\widehat\Omega\rightsquigarrow\Omega\).  In words, \(\Omega^t\) is obtained
from \(\widehat\Omega\) by performing \(A_\Omega(t)\) symplectic blow-ups.
Since each blow-up decreases \(K^2\) by \(1\), we have
\[
K_t^2=K_{\widehat\Omega}^2-A_\Omega(t).
\]

In the lattice-perimeter convention used here, only rational-slope linear
facets contribute.  Thus a smooth strictly convex boundary contributes no
lattice perimeter at time \(0\).  Consequently, the lattice perimeter of the
partial-cut model \(\Omega^t\) is the tail of the sum of triangle sizes:
\[
\Length_{\Z}(\partial\Omega^{\,t})
=
-\int_0^t u\,dA_\Omega(u).
\]
Using the canonical evolution formula \eqref{eq:canevol} at \(u=t\), we get
\[
\Length_{\Z}(\partial\Omega_t)
=
\Length_{\Z}(\partial\Omega^{\,t})
+
tA_\Omega(t)
+
o(t^{1/3}).
\]
The error term accounts for the finitely many critical events and for the
fact that the comparison is made at the same height \(t\) while the partial
model \(\Omega^t\) changes with \(t\).

For notational clarity, in the analytic completion below we write
\[
N^{\mathrm{cut}}_\Omega(t):=A_\Omega(t).
\]

\subsection{From the boundary Dirichlet series to the wave-front asymptotic}
\label{ssec:analyticcompletion}

We now complete the proof of the asymptotic lattice perimeter of tropical
wave fronts.  The first step is not applied directly to the interior Mellin
integral, but to the boundary Dirichlet series
\[
F_{\partial\Omega}(s),
\]
whose terms are indexed by the support triangles appearing in the cutting
process
\[
\widehat\Omega\rightsquigarrow\Omega.
\]
The analytic continuation of this boundary series is converted, by a
Tauberian argument, into an asymptotic for the counting function of
support-triangle sizes.  We then return to the tropical wave fronts and use
the geometry of the partial-cut models \(\Omega^t\) developed above.

It is important to distinguish the two Mellin-type objects in the paper.  On
the one hand, the interior tropical zeta function is
\[
Z_\Omega(s)=\int_0^{m_\Omega} t^{s-2}P_\Omega(t)\,dt,
\]
where
\[
P_\Omega(t)=\Length_{\Z}(\partial\Omega_t)
\]
is the lattice perimeter of the tropical wave front
\[
\Omega_t=\{x\in\Omega:\rho_\Omega(x)\ge t\}.
\]
On the other hand, the boundary Dirichlet series is
\[
F_{\partial\Omega}(s)=\sum_{\Delta}\operatorname{size}(\Delta)^s,
\]
where the sum runs over the support triangles \(\Delta\) in the cutting
process \(\widehat\Omega\rightsquigarrow\Omega\).  The Tauberian step below
is applied first to \(F_{\partial\Omega}(s)\), not to \(Z_\Omega(s)\).

Thus the logic is as follows.  We rewrite \(F_{\partial\Omega}(s)\) as a
Mellin--Laplace transform of the support-triangle counting function.  The
meromorphic continuation of \(F_{\partial\Omega}(s)\) near \(s=\frac23\) then
gives the asymptotic distribution of triangle sizes.  Finally, the comparison
between \(\Omega^t\) and \(\Omega_t\) converts that counting asymptotic into
the asymptotic law for \(P_\Omega(t)\), and hence into the corresponding
residue formula for \(Z_\Omega(s)\).

Let
\[
N^{\mathrm{cut}}_\Omega(t)
:=
\#\{\Delta:\operatorname{size}(\Delta)\ge t\},
\qquad t>0,
\]
where the count runs over all support triangles in
\(\widehat\Omega\rightsquigarrow\Omega\).  Since these sizes lie in
\((0,m_\Omega]\), define
\[
M_\Omega(x):=N^{\mathrm{cut}}_\Omega(e^{-x}),
\qquad
x\ge -\log m_\Omega.
\]

\begin{lemma}[Mellin--Laplace identity]
\label{lem:AOmega-mellin}
In the half-plane of absolute convergence,
\[\mkern-12mu
F_{\partial\Omega}(s)
=
s\int_0^\infty N^{\mathrm{cut}}_\Omega(t)t^{s-1}\,dt
=
s\int_0^{m_\Omega} N^{\mathrm{cut}}_\Omega(t)t^{s-1}\,dt
=
s\int_{-\log m_\Omega}^{\infty} M_\Omega(x)e^{-sx}\,dx.
\]
Equivalently,
\[
\frac1s\,F_{\partial\Omega}(s)
=
\int_{-\log m_\Omega}^{\infty} M_\Omega(x)e^{-sx}\,dx.
\]
\end{lemma}

\begin{proof}
For every \(\lambda>0\) and \(\Re(s)>0\),
\[
\lambda^s=s\int_0^\lambda t^{s-1}\,dt.
\]
Applying this identity termwise to the absolutely convergent series
\[
F_{\partial\Omega}(s)
=
\sum_\Delta \operatorname{size}(\Delta)^s
\]
gives
\[
F_{\partial\Omega}(s)
=
s\int_0^\infty N^{\mathrm{cut}}_\Omega(t)t^{s-1}\,dt.
\]
The counting function vanishes for \(t>m_\Omega\), so
\[
F_{\partial\Omega}(s)
=
s\int_0^{m_\Omega} N^{\mathrm{cut}}_\Omega(t)t^{s-1}\,dt.
\]
Now substitute \(t=e^{-x}\).  Since \(dt=-e^{-x}\,dx\), and
\[
t\in(0,m_\Omega]
\quad\Longleftrightarrow\quad
x\in[-\log m_\Omega,\infty),
\]
we obtain
\[
F_{\partial\Omega}(s)
=
s\int_{-\log m_\Omega}^{\infty}M_\Omega(x)e^{-sx}\,dx.
\]
Dividing by \(s\) gives the equivalent form.
\end{proof}

We use the following standard Laplace--Stieltjes form of the
Wiener--Ikehara theorem.

\begin{theorem}[Laplace--Tauberian theorem]
\label{thm:laplace-tauberian}
Let \(x_0\in\R\), and let
\[
B:[x_0,\infty)\to[0,\infty)
\]
be nondecreasing.  Suppose that for some \(c>\rho>0\),
\[
G(s):=\int_{x_0}^{\infty} B(x)e^{-sx}\,dx
\]
converges for \(\Re(s)>c\).  Assume that \(G(s)\) admits a meromorphic
continuation to the half-plane
\[
\Re(s)>\rho-\delta
\]
for some \(\delta>0\), that it is holomorphic there except for a simple pole
at \(s=\rho\), and that
\[
sG(s)-\frac{\rho A}{s-\rho}
\]
extends holomorphically to a neighborhood of the closed half-plane
\[
\Re(s)\ge \rho.
\]
Then
\[
B(x)\sim A e^{\rho x}
\qquad (x\to+\infty).
\]
\end{theorem}

\begin{proof}
Since \(B\) is nondecreasing, it has bounded variation on every compact
interval.  Consider its Laplace--Stieltjes transform
\[
H(s):=\int_{x_0}^{\infty}e^{-sx}\,dB(x),
\qquad \Re(s)>c.
\]
Integration by parts gives
\[
G(s)
=
\left[-\frac{B(x)e^{-sx}}{s}\right]_{x=x_0}^{\infty}
+
\frac1s\int_{x_0}^{\infty}e^{-sx}\,dB(x)
=
\frac{B(x_0)e^{-sx_0}}{s}+\frac{H(s)}{s}.
\]
Thus
\[
H(s)=sG(s)-B(x_0)e^{-sx_0}.
\]
Therefore
\[
H(s)-\frac{\rho A}{s-\rho}
=
s\left(G(s)-\frac{A}{s-\rho}\right)
+
A-B(x_0)e^{-sx_0}.
\]
By hypothesis, the right-hand side is holomorphic in a neighborhood of the
closed half-plane \(\Re(s)\ge\rho\).  The Wiener--Ikehara theorem in Laplace--Stieltjes form
then gives
\[
B(x)\sim \frac{\rho A}{\rho}e^{\rho x}=A e^{\rho x}
\qquad (x\to+\infty).
\]
See, for example, \cite[Chapter~III]{Korevaar2004}, building on the classical
works \cite{Ikehara1931,Wiener1932}.
\end{proof}

\begin{proposition}[Half-plane continuation of the full boundary series]
\label{prop:Fboundary-halfplane-app}
Under the standing smoothness and curvature assumptions on the arcs of
\(\partial\Omega\setminus\partial\widehat\Omega\), the full boundary series
\[
F_{\partial\Omega}(s)
\]
admits a meromorphic continuation to the half-plane
\[
\Re(s)>\frac35.
\]
It is holomorphic there except for a simple pole at
\[
s=\frac23.
\]
Moreover,
\[
\Res_{s=2/3}F_{\partial\Omega}(s)
=
\frac{\sqrt3\,\Gamma(1/3)^3}{2^{2/3}\pi^3}
\Length_{\mathrm{equiaffine}}(\partial\Omega).
\]
\end{proposition}

\begin{proof}
Let \(\widetilde g\) be the Legendre dual associated with a boundary arc
\(\Gamma\). By Theorem~\ref{thm:res-halfplane} from
Appendix~\ref{app3}, Subsection~\ref{ssec:striprefinement},
the series \(Z_{\widetilde g}(s)\) admits a meromorphic continuation to
\[
\Re(s)>\frac35,
\]
holomorphic there except for a simple pole at \(s=\frac23\). By
Corollary~\ref{cor:Fgamma-is-Zg},
\[
F_\Gamma(s)-Z_{\widetilde g}(s)
\]
is entire. Hence \(F_\Gamma(s)\) has the same half-plane continuation, the
same pole, and the same residue. Using
Lemma~\ref{lem:Legendre-ea}, we obtain
\[
\Res_{s=2/3}F_\Gamma(s)
=
\frac{\sqrt3\,\Gamma(1/3)^3}{2^{2/3}\pi^3}
\Length_{\mathrm{equiaffine}}(\Gamma).
\]
Since
\[
F_{\partial\Omega}(s)=\sum_\Gamma F_\Gamma(s)
\]
is a finite sum over the arcs in
\(\partial\Omega\setminus\partial\widehat\Omega\), the same continuation
statement holds for \(F_{\partial\Omega}(s)\).  The residue is the sum of the
arc residues.  By additivity of equiaffine length,
\[
\sum_\Gamma \Length_{\mathrm{equiaffine}}(\Gamma)
=
\Length_{\mathrm{equiaffine}}(\partial\Omega).
\]
\end{proof}

\begin{proposition}[Asymptotic for the support-triangle counting function]
\label{prop:Aomega-tauberian}
Let
\[
r_\Omega:=\Res_{s=2/3}F_{\partial\Omega}(s).
\]
Then, as \(x\to+\infty\),
\[
M_\Omega(x)\sim \frac32\,r_\Omega\,e^{2x/3}.
\]
Equivalently, as \(t\to0^+\),
\[
N^{\mathrm{cut}}_\Omega(t)\sim \frac32\,r_\Omega\,t^{-2/3}.
\]
\end{proposition}

\begin{proof}
By Lemma~\ref{lem:AOmega-mellin},
\[
G_\Omega(s):=\frac1sF_{\partial\Omega}(s)
=
\int_{-\log m_\Omega}^{\infty}M_\Omega(x)e^{-sx}\,dx
\]
in the half-plane of absolute convergence.  By
Proposition~\ref{prop:Fboundary-halfplane-app}, the function
\(F_{\partial\Omega}(s)\) is meromorphic on
\[
\Re(s)>\frac35,
\]
with a single simple pole at \(s=\frac23\).  Since \(1/s\) is holomorphic
there, the same is true for \(G_\Omega(s)\).  Moreover,
\[
\Res_{s=2/3}G_\Omega(s)
=
\Res_{s=2/3}\frac1sF_{\partial\Omega}(s)
=
\frac{1}{2/3}\Res_{s=2/3}F_{\partial\Omega}(s)
=
\frac32\,r_\Omega.
\]
Therefore
\[
G_\Omega(s)-\frac{\frac32\,r_\Omega}{s-\frac23}
\]
is holomorphic in a neighborhood of the closed half-plane
\(\Re(s)\ge\frac23\).

Applying Theorem~\ref{thm:laplace-tauberian} with
\[
B=M_\Omega,\qquad
\rho=\frac23,\qquad
A=\frac32\,r_\Omega,
\]
we obtain
\[
M_\Omega(x)\sim \frac32\,r_\Omega e^{2x/3}.
\]
Substituting \(x=-\log t\) gives
\[
N^{\mathrm{cut}}_\Omega(t)
=
M_\Omega(-\log t)
\sim
\frac32\,r_\Omega t^{-2/3}.
\]
\end{proof}

The asymptotic just obtained concerns the tail distribution of the
support-triangle sizes in the cutting process
\[
\widehat\Omega\rightsquigarrow\Omega.
\]
To convert it into an asymptotic for the lattice perimeter of the tropical
wave fronts, we return to the partial-cut models \(\Omega^t\).  The passage
from boundary triangles to wave fronts uses the compatibility between the
truncated domains \(\Omega^t\) and the actual wave fronts \(\Omega_t\), the
canonical evolution of the intermediate toric models through the classes
\(K_t\), and the comparison between the lattice perimeters of
\(\partial\Omega^t\) and \(\partial\Omega_t\).  Thus the Tauberian asymptotic
for \(N^{\mathrm{cut}}_\Omega(t)\) is transported from the boundary
Dirichlet series to the interior wave-front profile governed by \(Z_\Omega(s)\).

\begin{theorem}[Asymptotic lattice perimeter of tropical wave fronts]
\label{thm:twfper-rigorous}
Let \(\Omega\subset\R^2\) be a compact convex domain with \(C^3\)-smooth
boundary, and assume that each arc of
\(\partial\Omega\setminus\partial\widehat\Omega\) has everywhere nonvanishing
curvature.  Then, as \(t\to0^+\),
\[
\Length_{\Z}(\partial\Omega_t)
=
\frac92\,r_\Omega\,t^{1/3}+o(t^{1/3}).
\]
Equivalently,
\[
\Length_{\Z}(\partial\Omega_t)
=
\Res_{s=2/3}Z_\Omega(s)\,t^{1/3}+o(t^{1/3}).
\]
\end{theorem}

\begin{proof}
By Proposition~\ref{prop:Aomega-tauberian},
\[
N^{\mathrm{cut}}_\Omega(t)
=
\frac32\,r_\Omega\,t^{-2/3}+o(t^{-2/3})
\qquad (t\to0^+).
\]
Using the Stieltjes identity established above,
\[
\Length_{\Z}(\partial\Omega^{\,t})
=
-\int_0^t u\,dN^{\mathrm{cut}}_\Omega(u)
=
-tN^{\mathrm{cut}}_\Omega(t)
+
\int_0^t N^{\mathrm{cut}}_\Omega(u)\,du,
\]
we get
\[
-tN^{\mathrm{cut}}_\Omega(t)
=
-\frac32\,r_\Omega\,t^{1/3}+o(t^{1/3}),
\]
and
\[
\int_0^t N^{\mathrm{cut}}_\Omega(u)\,du
=
\frac32\,r_\Omega\int_0^t u^{-2/3}\,du+o(t^{1/3})
=
\frac92\,r_\Omega\,t^{1/3}+o(t^{1/3}).
\]
Hence
\[
\Length_{\Z}(\partial\Omega^{\,t})
=
3\,r_\Omega\,t^{1/3}+o(t^{1/3}).
\]

\clearpage

\begin{figure}[p]
\centering
\vspace*{-4mm}

\begin{tikzpicture}[
  >=Latex,
  font=\scriptsize,
  box/.style={
    draw,
    rounded corners=2.5pt,
    align=center,
    inner xsep=4pt,
    inner ysep=3pt,
    text width=0.73\textwidth
  },
  final/.style={
    draw,
    very thick,
    rounded corners=3pt,
    align=center,
    inner xsep=5pt,
    inner ysep=4pt,
    text width=0.75\textwidth
  },
  arrow/.style={
    ->,
    thick
  },
  node distance=2.8mm
]

\node[box] (F) {
\textbf{Boundary Dirichlet series}\\[-1pt]
The support triangles in the cutting process
\(\widehat\Omega\rightsquigarrow\Omega\) define
\[
F_{\partial\Omega}(s)
=
\sum_{\Delta}\operatorname{size}(\Delta)^s,
\qquad
r_\Omega
:=
\Res_{s=2/3}F_{\partial\Omega}(s).
\]
};

\node[box, below=of F] (cont) {
\textbf{Analytic continuation}\\[-1pt]
The boundary series is meromorphic for
\[
\Re(s)>\frac35
\]
and holomorphic there except for a simple pole at \(s=\frac23\).
};

\node[box, below=of cont] (laplace) {
\textbf{Mellin--Laplace transform}\\[-1pt]
Let
\[
N^{\mathrm{cut}}_\Omega(t)
:=
\#\{\Delta:\operatorname{size}(\Delta)\ge t\},
\qquad
M_\Omega(x)
:=
N^{\mathrm{cut}}_\Omega(e^{-x}).
\]
Then
\[
G_\Omega(s)
:=
\frac1sF_{\partial\Omega}(s)
=
\int_{-\log m_\Omega}^{\infty}
M_\Omega(x)e^{-sx}\,dx.
\]
};

\node[box, below=of laplace] (taub) {
\textbf{Laplace--Tauberian step}\\[-1pt]
Since
\[
\Res_{s=2/3}G_\Omega(s)
=
\frac{1}{2/3}\Res_{s=2/3}F_{\partial\Omega}(s)
=
\frac32r_\Omega,
\]
the Laplace--Tauberian theorem gives
\[
M_\Omega(x)
\sim
\frac32r_\Omega e^{2x/3},
\qquad
N^{\mathrm{cut}}_\Omega(t)
\sim
\frac32r_\Omega t^{-2/3}.
\]
};

\node[box, below=of taub] (stieltjes) {
\textbf{Partial-cut perimeter}\\[-1pt]
For the partial-cut model \(\Omega^{\,t}\),
\[
\Length_{\Z}(\partial\Omega^{\,t})
=
-\int_0^t u\,dN^{\mathrm{cut}}_\Omega(u)
=
-tN^{\mathrm{cut}}_\Omega(t)
+
\int_0^tN^{\mathrm{cut}}_\Omega(u)\,du.
\]
Therefore
\[
\Length_{\Z}(\partial\Omega^{\,t})
=
3r_\Omega t^{1/3}
+
o(t^{1/3}).
\]
};

\node[box, below=of stieltjes] (transport) {
\textbf{Transport to the actual wave front}\\[-1pt]
The comparison with the actual tropical wave front gives
\[
\Length_{\Z}(\partial\Omega_t)
=
\Length_{\Z}(\partial\Omega^{\,t})
+
tN^{\mathrm{cut}}_\Omega(t)
+
o(t^{1/3}).
\]
Since
\[
tN^{\mathrm{cut}}_\Omega(t)
=
\frac32r_\Omega t^{1/3}
+
o(t^{1/3}),
\]
we obtain
\[
\Length_{\Z}(\partial\Omega_t)
=
\frac92r_\Omega t^{1/3}
+
o(t^{1/3}).
\]
};

\node[final, below=of transport] (final) {
\textbf{Tropical wave-front asymptotic}\\[-1pt]
The integral--boundary identity gives
\[
\Res_{s=2/3}Z_\Omega(s)
=
\frac92r_\Omega.
\]
Hence, as \(t\to0^+\),
\[
\boxed{
\Length_{\Z}(\partial\Omega_t)
=
\Res_{s=2/3}Z_\Omega(s)\,t^{1/3}
+
o(t^{1/3})
}.
\]
};

\draw[arrow] (F) -- (cont);
\draw[arrow] (cont) -- (laplace);
\draw[arrow] (laplace) -- (taub);
\draw[arrow] (taub) -- (stieltjes);
\draw[arrow] (stieltjes) -- (transport);
\draw[arrow] (transport) -- (final);

\end{tikzpicture}

\vspace{2mm}

\caption{
Dependency structure of the tropical wave-front lattice-perimeter
asymptotic, under the hypotheses of the wave-front theorem.
The boundary Dirichlet series is rewritten as a Mellin--Laplace transform
of the support-triangle counting function. Its pole at \(s=\frac23\)
gives the triangle-size tail asymptotic by a Laplace--Tauberian theorem.
Stieltjes integration yields the lattice perimeter of the partial-cut model,
and comparison with the actual tropical wave front gives
\(\frac92r_\Omega=\Res_{s=2/3}Z_\Omega(s)\).
}
\label{fig:twf-tauberian-dependency}

\end{figure}

\clearpage

Now use the comparison between the partial-cut models and the actual wave
fronts:
\[
\Length_{\Z}(\partial\Omega_t)
=
\Length_{\Z}(\partial\Omega^{\,t})
+
tN^{\mathrm{cut}}_\Omega(t)
+
o(t^{1/3}).
\]
Since
\[
tN^{\mathrm{cut}}_\Omega(t)
=
\frac32\,r_\Omega\,t^{1/3}+o(t^{1/3}),
\]
we obtain
\[
\Length_{\Z}(\partial\Omega_t)
=
\left(3+\frac32\right)r_\Omega\,t^{1/3}+o(t^{1/3})
=
\frac92\,r_\Omega\,t^{1/3}+o(t^{1/3}).
\]

Finally, the integral--boundary identity gives
\[
\Res_{s=2/3}Z_\Omega(s)=\frac92\,r_\Omega,
\]
and the equivalent form follows.
\end{proof}

\begin{corollary}[Area deficit]
\label{cor:area-deficit-rigorous}
Under the hypotheses of Theorem~\ref{thm:twfper-rigorous},
\[
\Area(\Omega)-\Area(\Omega_t)
=
\frac34\,\Res_{s=2/3}Z_\Omega(s)\,t^{4/3}+o(t^{4/3})
\qquad (t\to0^+).
\]
\end{corollary}

\begin{proof}
By Proposition~\ref{prop:twf_surfacevol}, in dimension two,
\[
\frac{d}{dt}\Area(\Omega_t)
=
-\Length_{\Z}(\partial\Omega_t).
\]
Therefore
\[
\Area(\Omega)-\Area(\Omega_t)
=
\int_0^t \Length_{\Z}(\partial\Omega_u)\,du.
\]
Applying Theorem~\ref{thm:twfper-rigorous}, we get
\[
\Length_{\Z}(\partial\Omega_u)
=
\Res_{s=2/3}Z_\Omega(s)\,u^{1/3}+o(u^{1/3}).
\]
Hence
\[
\Area(\Omega)-\Area(\Omega_t)
=
\Res_{s=2/3}Z_\Omega(s)
\int_0^t u^{1/3}\,du
+
o(t^{4/3})
=
\frac34\,\Res_{s=2/3}Z_\Omega(s)\,t^{4/3}
+
o(t^{4/3}).
\]
\end{proof}

\section{Zeta functions for convex domains: a survey}\label{sec_survey}

Several classical zeta functions may be attached to a convex domain, depending on which geometric profile one chooses to encode analytically. The main examples relevant here are:

\begin{enumerate}
\item gauge or lattice zeta functions, attached to lattice-point counting in dilates of the body;
\item distance and tube zeta functions, attached to the small-scale growth of Euclidean neighborhoods;
\item spectral zeta functions, attached to eigenvalue asymptotics of natural operators on the domain or on its boundary.
\end{enumerate}

In each case, the zeta function is a Dirichlet series or Mellin transform whose first singularity is determined by the leading asymptotic term of the underlying counting or scale problem; see, for instance, \cite{Montoro2018,LapidusRadunovicZubrinic2017,Seeley1967}. The goal of this section is to recall these constructions in a form suited to comparison with the zeta function studied in the main text.

\subsection{Gauge zetas, Hlawka zeta, and lattice-point asymptotics}

Let \(K\subset\mathbb R^n\) be a convex body with \(0\in\operatorname{int}K\), and let
\[
\|x\|_K:=\inf\{t>0:x\in tK\}
\]
be its Minkowski functional. The associated gauge zeta function is
\[
Z_K(s):=\sum_{m\in\mathbb Z^n\setminus\{0\}} \|m\|_K^{-s},
\qquad \Re(s)>n.
\]
Its natural counting function is
\[
N_K(X):=\#\{m\in\mathbb Z^n:\|m\|_K\le X\}
       =\#(XK\cap\mathbb Z^n).
\]
Thus \(Z_K(s)\) is the Mellin--Dirichlet transform naturally associated with lattice-point counting in dilates of \(K\). In particular, its first singularity is expected at \(s=n\), in accordance with the leading term \(\operatorname{vol}(K)X^n\) in lattice-point asymptotics \cite{KratzelNowak1991,Montoro2018}.

Restricting to primitive lattice points gives
\[
Z_K^{\mathrm{prim}}(s):=
\sum_{\substack{m\in\mathbb Z^n\setminus\{0\}\\ \gcd(m_1,\dots,m_n)=1}}
\|m\|_K^{-s},
\qquad \Re(s)>n,
\]
and homogeneity implies
\[
Z_K(s)=\zeta(s)\,Z_K^{\mathrm{prim}}(s)
\]
in the common half-plane of absolute convergence.

In dimension \(2\), for a star-shaped domain \(D\subset\mathbb R^2\) about the origin, one often writes \(t(m,n)\) for the least \(t>0\) such that \((m,n)\in tD\), and defines the Hlawka zeta by
\[
Z_D(s):=\sum_{(m,n)\neq(0,0)} t(m,n)^{-2s},
\qquad \Re(s)>1.
\]
With this normalization, the natural pole is at \(s=1\), corresponding to the planar lattice-counting problem
\[
A_D(x):=\#(xD\cap\mathbb Z^2),
\]
and this is the form in which the construction is usually discussed in the modern literature \cite{Montoro2018}.

For the Euclidean disk one recovers
\[
Z_{\mathrm{disk}}(s)=\sum_{(m,n)\neq(0,0)} (m^2+n^2)^{-s},
\]
the Epstein zeta function of the quadratic form \(x^2+y^2\). More generally, positive-definite binary quadratic forms lead to Epstein zeta functions, for which analytic continuation and functional equations are classical \cite{Epstein1903}. Primitive restrictions and boundary effects in convex lattice counting lead to finer arithmetic refinements \cite{KratzelNowak2001}.

\subsection{Distance zetas, tube zetas, and Euclidean boundary geometry}

Let \(A\subset\mathbb R^N\) be bounded and fix \(\delta>0\). The distance zeta function of \(A\) is
\[
\zeta_A(s):=\int_{A_\delta} \operatorname{dist}(x,A)^{\,s-N}\,dx,
\]
initially defined for \(\Re(s)\) sufficiently large. A basic theorem states that the abscissa of absolute convergence of \(\zeta_A\) is the upper box dimension \(\overline{\dim}_B A\) \cite{LapidusRadunovicZubrinic2017JMAA,LapidusRadunovicZubrinic2017}.

A closely related object is the tube zeta function
\[
\widetilde\zeta_A(s):=\int_0^\delta t^{\,s-N-1}|A_t|\,dt,
\]
where \(A_t=\{x:\operatorname{dist}(x,A)<t\}\). The two zeta functions are linked by an explicit identity, and under suitable meromorphic continuation hypotheses their poles encode the small-\(t\) asymptotics of \(|A_t|\) \cite{LapidusRadunovicZubrinic2017JMAA,LapidusRadunovicZubrinic2017}.

For convex geometry the most natural choice is \(A=\partial\Omega\), where \(\Omega\subset\mathbb R^N\) is a bounded convex domain. If \(\partial\Omega\) is sufficiently regular (for example \(C^{1,1}\), or more generally of positive reach), then the tube expansion is classical, and the first singularity occurs at \(s=N-1\), the dimension of the boundary. Under the usual Minkowski measurability hypotheses, the residue at this first pole is proportional to \(\mathcal H^{N-1}(\partial\Omega)\). For smooth convex bodies, the higher coefficients are the classical curvature coefficients appearing in the Steiner--Federer theory \cite{Federer1959,LapidusRadunovicZubrinic2017}.

\subsection{Spectral zetas of the interior and of the boundary}

Let \(\Omega\subset\mathbb R^N\) be a bounded domain with smooth boundary, and let
\[
0<\lambda_1\le \lambda_2\le\cdots\to\infty
\]
be the Dirichlet eigenvalues of the Laplacian. The spectral zeta function is
\[
\zeta_\Delta(s):=\sum_{j=1}^\infty \lambda_j^{-s},
\qquad \Re(s)>\frac N2.
\]
Its meromorphic continuation is governed by the heat trace expansion
\[
\operatorname{Tr}(e^{-t\Delta})\sim \sum_{j\ge0} c_j\, t^{(j-N)/2},
\qquad t\to0^+,
\]
and therefore \(\zeta_\Delta(s)\) has at most simple poles at
\[
s=\frac{N-j}{2},\qquad j=0,1,2,\dots
\]
\cite{MinakshisundaramPleijel1949,Seeley1967}. The leading pole at \(s=N/2\) is proportional to \(|\Omega|\), while the next coefficient involves \(\mathcal H^{N-1}(\partial\Omega)\) \cite{MinakshisundaramPleijel1949,Seeley1967}.

A boundary spectral analogue is provided by the Dirichlet-to-Neumann operator \(\Lambda_\Omega\). In dimension \(2\), if \(\Omega\) is a smooth simply connected planar domain and
\[
0=\sigma_0<\sigma_1\le\sigma_2\le\cdots
\]
are the Steklov eigenvalues, one defines
\[
\zeta_{\mathrm{St}}(s):=\sum_{j=1}^\infty \sigma_j^{-s}.
\]
Here the leading asymptotic is already one-dimensional: the Weyl law is governed by the boundary length \(L(\partial\Omega)\), and correspondingly the principal singular behavior of \(\zeta_{\mathrm{St}}(s)\) is that of
\[
2\Bigl(\frac{L(\partial\Omega)}{2\pi}\Bigr)^s \zeta_R(s)
\]
\cite{GirouardParnovskiPolterovichSher2014,JollivetSharafutdinov2018,JollivetSharafutdinov2021}. Thus the Dirichlet Laplacian and the Steklov problem produce different zeta functions on the same domain, reflecting, respectively, an interior asymptotic law and a boundary asymptotic law.

\subsection{Comparison with the tropical zeta function of the present paper}

We now place the tropical zeta function studied in this paper alongside the preceding constructions.

Let \(\Omega\subset\RR^2\) be a compact convex domain. The zeta function introduced in the main text is
\[
Z_\Omega(s)=\int_\Omega \rho_\Omega(x)^{\,s-2}\,dx,
\]
where
\[
\rho_\Omega(x)=\min_{u\in\ZZ^2_{\mathrm{prim}}}\bigl(\langle u,x\rangle-h_\Omega(u)\bigr)
\]
is the tropical distance-to-the-boundary function defined by primitive lattice supporting lines. By Proposition~\ref{prop:perimeter-mellin}, this admits the Mellin representation
\[
Z_\Omega(s)=\int_0^{m_\Omega} t^{s-2}P_\Omega(t)\,dt,
\]
where \(P_\Omega(t)\) is the lattice perimeter of the tropical wave front
\[
\Omega_t=\{x\in\Omega:\rho_\Omega(x)\ge t\}.
\]
Thus, like distance, tube, and spectral zeta functions, the tropical zeta function is a Mellin transform of a geometric profile. The essential difference lies in the profile to which the Mellin transform is applied: here it is not Euclidean neighborhood growth, nor the spectrum of a differential operator, but the lattice perimeter of the wave fronts generated by primitive lattice support data.

This places the tropical zeta function closest in spirit to the gauge and Hlawka zeta functions, since all of them are built from the interaction of the convex body with the ambient lattice. However, the role of the lattice is different. In the gauge zeta
\[
Z_K(s)=\sum_{m\in\ZZ^n\setminus\{0\}}\|m\|_K^{-s},
\]
the lattice enters through the distribution of lattice points in dilates of \(K\), so the first pole is governed by the main term in lattice-point counting, namely volume \cite{KratzelNowak1991,Montoro2018}. By contrast, in the tropical zeta function the lattice enters through the \emph{primitive supporting directions} of the boundary. The resulting zeta is still \(\operatorname{SL}(2,\ZZ)\)-invariant, but its singularities are governed by the boundary evolution of the tropical wave fronts rather than by the counting of lattice points inside dilates\footnote{We will elaborate on the relation between these two problems in future work. Note also that the Hlawka zeta function is sensitive to translations of the domain, even by integer vectors, whereas the tropical zeta function is invariant under all real translations. This difference is in fact responsible for an obstruction to expressing the lattice-point error term solely in terms of the residues of the tropical zeta function.}.

This difference is already visible in the polygonal case. For rational polygons, the first pole of the tropical zeta function occurs at
\[
s=1,
\]
and its residue is the lattice perimeter of \(\partial\Omega\); see Theorem~\ref{thm:integral-boundary} together with the computations of minimal models in Appendix~\ref{app1}. Thus the leading singularity records a boundary quantity, not the area term coming from the interior. In this respect, the tropical zeta function is more boundary-sensitive than the gauge zeta and more arithmetic than the Euclidean distance and tube zetas.

Compared with distance and tube zeta functions, the contrast is even sharper. For
\[
\zeta_A(s)=\int_{A_\delta}\operatorname{dist}(x,A)^{\,s-N}\,dx
\]
and
\[
\widetilde\zeta_A(s)=\int_0^\delta t^{s-N-1}|A_t|\,dt,
\]
the relevant small-scale profile is Euclidean, and for smooth convex boundaries the first pole is at the boundary dimension \(N-1\), with residue proportional to Euclidean boundary measure \cite{LapidusRadunovicZubrinic2017JMAA,LapidusRadunovicZubrinic2017,Federer1959}. In the tropical case, one again has a Mellin transform of a boundary profile, but the profile is measured in the geometry of primitive lattice support planes. As a consequence, the first nontrivial pole in the smooth convex planar case with everywhere nonvanishing curvature is not at \(s=1\) but at
\[
s=\frac23,
\]
and the residue is proportional to equiaffine arc length:
\[
\operatorname*{Res}_{s=2/3}Z_\Omega(s)
=
\frac{3^{5/2}}{2^{5/3}\pi^3}\Gamma\!\left(\frac13\right)^3
\operatorname{Length}_{\mathrm{equiaffine}}(\partial\Omega).
\]
Thus, whereas Euclidean tube zeta functions recover Euclidean curvature data, the tropical zeta function recovers an affine boundary invariant.

The comparison with spectral zeta functions is also instructive. For the Dirichlet Laplacian, the poles of
\[
\zeta_\Delta(s)=\sum_{j\ge1}\lambda_j^{-s}
\]
are governed by the heat trace, and the leading pole reflects the interior Weyl law, hence volume \cite{MinakshisundaramPleijel1949,Seeley1967}. For the Steklov zeta, the leading asymptotic is already boundary-based and determined by the boundary length \cite{GirouardParnovskiPolterovichSher2014,JollivetSharafutdinov2018}. The tropical zeta function resembles the Steklov zeta in that its leading nontrivial asymptotics live on the boundary rather than in the interior. But it differs from both spectral examples in that it does not arise from the spectrum of an operator. Its analytic continuation is controlled instead by a boundary Dirichlet series over Farey neighbors and, ultimately, by incomplete Kloosterman-sum estimates and Legendre duality.

A second distinguishing feature of the tropical zeta function is its exact reduction to a boundary series. In dimension \(2\), Theorem~\ref{thm:integral-boundary} shows that
\[
s(s-1)Z_\Omega(s)=-F_{\partial\Omega}(s)+H_{\widehat\Omega}(s),
\]
where \(H_{\widehat\Omega}(s)\) is an explicit holomorphic correction term coming from the minimal model (see Proposition~\ref{prop:minmodzeta}). Thus all nontrivial singular behavior is already encoded in the boundary Dirichlet series \(F_{\partial\Omega}(s)\). This kind of exact interior-to-boundary reduction does not occur in the same form for the classical gauge, tube, or spectral zeta functions recalled above.

Finally, from the arithmetic point of view, the tropical zeta function appears to be the construction among these examples that is most directly adapted to the study of rational and lattice points near convex curves. In the smooth convex case with everywhere nonvanishing curvature, its leading residue is equiaffine arc length, which is precisely the affine invariant that governs several quantitative problems on rational points and lattice points on convex arcs \cite{Petrov2006,HowardTrifonov2022,Howard2023}. In this sense, the tropical zeta function is neither just Euclidean nor spectral: it is an \(\operatorname{SL}(2,\ZZ)\)-arithmetic zeta whose first genuinely smooth residue is \(\operatorname{SL}(2,\RR)\)-affine.

To summarize, the tropical zeta function shares with the classical constructions the general Dirichlet--Mellin principle that poles encode asymptotic geometry, but it differs from them in three essential respects:
\begin{enumerate}
\item its defining geometric input is the tropical distance function built from primitive lattice support directions;
\item in dimension \(2\), its singularities are governed by an exact boundary Dirichlet series over Farey neighbors;
\item for smooth convex domains with everywhere nonvanishing curvature, its first nontrivial pole detects equiaffine arc length rather than Euclidean boundary measure, volume, or spectral Weyl data.
\end{enumerate}

This combination of lattice symmetry, boundary reduction, and affine residue appears to be specific to the tropical zeta function introduced in this paper.

The tropical zeta function thus places the geometry of convex domains into a single analytic framework in which discrete lattice data and smooth affine geometry appear as different singular regimes of the same object. In the polygonal case, the first pole detects the lattice-visible boundary through the tropical perimeter; in the smooth strictly convex case, the leading nontrivial residue is the equiaffine perimeter. The mechanism behind this transition is the exact reduction of the interior Mellin integral to a boundary Dirichlet series, whose arithmetic structure is governed by Farey neighbors and whose local model is the primitive Mordell--Tornheim series, equivalently Witten's \(\mathrm{SU}(3)\) zeta function. From this perspective, the pole at \(s=\tfrac23\) is not an isolated accident but the first smooth affine manifestation of a tropical-lattice analytic theory of convexity. The broader continuation problem, the higher-dimensional case, and the relation to lattice-point error terms remain open, but the results obtained here already show that tropical optics provides a natural zeta-theoretic passage from arithmetic boundary data to equiaffine geometry.

\section{Analytic proof of Proposition~\ref{prop:spec}}\label{app3}

\subsubsection{Regularity hypotheses and analytic dependence on \(s\)}

For \(s\in\C\) and \(x\in[0,1]\), define
\[
B_s(x):=|f''(x)|^s=\exp\bigl(s\log|f''(x)|\bigr).
\]
For each fixed \(x\), the map \(s\mapsto B_s(x)\) is entire. Moreover, for \(s\) in any compact set \(K\subset\C\), we have the uniform bounds
\begin{equation}
\label{eq:Bs_bounds}
\sup_{s\in K}\|B_s\|_{\infty}\le M^{\sup_{s\in K}\Re(s)},\qquad
\sup_{s\in K}\|(B_s)'\|_{\infty}\le C_K<\infty,
\end{equation}
since
\((B_s)'(x)=s\,|f''(x)|^s\,\frac{f'''(x)}{f''(x)}\)
and $f'''/f''$ is bounded by \eqref{eq:ass}.

\subsubsection{A model endpoint-sampled series}

\begin{proof}[Proof of Lemma~\ref{lem:replace}]\label{proof:lem:replace}
By Lemma~\ref{lemma:exactT},
\(Z_f(s)=2^{-s}\sum_I |f''(\xi_I)|^s/(bd(b+d))^s\).
Fix a small disk $D:=\{s:|s-2/3|\le \eta\}$ with $0<\eta<1/30$.
Let $\sigma_-:=\min_{s\in D}\Re(s)=2/3-\eta$.

For \(s\in D\) and \(u,v\in[0,1]\), the mean-value theorem and \eqref{eq:Bs_bounds} give
\[
\bigl||f''(u)|^s-|f''(v)|^s\bigr|\le \sup_{s\in D}\|(B_s)'\|_{\infty}\,|u-v|\le C_D\,|u-v|.
\]
For a Farey interval $I=[c/d,a/b]$, we have $|\xi_I-a/b|\le|a/b-c/d|=1/(bd)$.
Hence for $s\in D$,
\[\!\!\!
\left|\frac{|f''(\xi_I)|^s-|f''(a/b)|^s}{(bd(b+d))^s}\right|
\le
\frac{C_D}{(bd)\,(bd(b+d))^{\sigma_-}}
=
C_D\, b^{-(\sigma_-+1)}d^{-(\sigma_-+1)}(b+d)^{-\sigma_-}.
\]
The right-hand side is independent of $s$ and is summable over all $b,d\ge1$ because, for \(\sigma_-=2/3-\eta\), the exponent of each of \(b\) and \(d\) is \(\sigma_-+1>1\).
Therefore the series defining $Z_f(s)-Z_f^{\mathrm{end}}(s)$ converges uniformly on $D$ by the Weierstrass $M$-test. Since each term is holomorphic in $s$, the sum is holomorphic on the interior of $D$.
\end{proof}


\subsection{Auxiliary bounds for $H_s$ and reduced-residue sampling}

\begin{lemma}[\texorpdfstring{Pointwise bound for \(H_s\)}{Pointwise bound for Hs}]
\label{lem:Hsbound}
For \(1/2<\sigma:=\Re(s)<1\) and \(u\in(0,1]\), one has
\[
|H_s(u)|\le C_{s}\,u^{-\sigma},
\]
with $C_s$ depending continuously on $s$ on compact subsets of $\{1/2<\Re(s)<1\}$.
\end{lemma}

\begin{proof}
Write
\(H_s(u)=u^{-s}(1+u)^{-s}+\sum_{k\ge1}(k+u)^{-s}(k+1+u)^{-s}\).
The first term has modulus at most $u^{-\sigma}$. For the tail,
\(\sum_{k\ge1}(k+u)^{-\sigma}(k+1+u)^{-\sigma}\le \sum_{k\ge1}k^{-2\sigma}<\infty\).
Since $u^{-\sigma}\ge1$ on $(0,1]$, the tail is $\ll_s u^{-\sigma}$.
\end{proof}

\begin{lemma}[Derivative bound for $H_s$]
\label{lem:Hsderiv}
Fix a compact set $K\Subset\{\tfrac12<\Re(s)<1\}$. Then there exists $C_K>0$ such that for all $s\in K$ and all $u\in(0,1]$,
\[
\bigl|\partial_u H_s(u)\bigr|\le C_K\,u^{-\Re(s)-1}.
\]
\end{lemma}

\begin{proof}
Differentiate termwise (absolute convergence holds uniformly for $u\in[\delta,1]$, and the $k=0$ term controls the singularity as $u\downarrow0$):
\[
\partial_u H_s(u)
=
-s\sum_{k\ge0}(k+u)^{-s-1}(k+1+u)^{-s}
-s\sum_{k\ge0}(k+u)^{-s}(k+1+u)^{-s-1}.
\]
For $u\in(0,1]$ the $k=0$ contributions satisfy
\(
|(u)^{-s-1}(1+u)^{-s}|+|u^{-s}(1+u)^{-s-1}|\ll_K u^{-\Re(s)-1}.
\)
For the tails $k\ge1$ we bound each summand by
\(
|s|\,(k+u)^{-\sigma-1}(k+1+u)^{-\sigma}+|s|\,(k+u)^{-\sigma}(k+1+u)^{-\sigma-1}
\ll_K k^{-2\sigma-1},
\)
where \(\sigma=\Re(s)>1/2\), and \(\sum_{k\ge1}k^{-2\sigma-1}<\infty$ uniformly on $K$.
Since $u^{-\sigma-1}\ge1$ on $(0,1]$, the tail is also $\ll_K u^{-\sigma-1}$.
\end{proof}

\begin{lemma}[Reduced residues: $H_s$ averages to its integral]
\label{lem:unitavg}
Fix a compact set $K\Subset\{\tfrac12<\Re(s)<1\}$. Then for $s\in K$ and $b\ge1$,
\[
\sum_{\substack{1\le r\le b\\(r,b)=1}} H_s(r/b)
=
\varphi(b)\int_0^1 H_s(u)\,du
+O_K\bigl(b^{\Re(s)}\tau(b)\bigr),
\]
where $\tau(b)$ is the divisor function and the implied constant is uniform for $s\in K$.
\end{lemma}

\begin{proof}
By inclusion--exclusion,
\[
\sum_{\substack{1\le r\le b\\(r,b)=1}} H_s(r/b)
=
\sum_{d\mid b}\mu(d)\sum_{1\le k\le b/d} H_s\!\left(\frac{k}{b/d}\right).
\]
Write $N=b/d$ and denote $S_s(N):=\sum_{k=1}^N H_s(k/N)$.
We prove
\begin{equation}
\label{eq:Rs}
S_s(N)=N\int_0^1 H_s(u)\,du+O_K(N^{\Re(s)}),
\end{equation}
uniformly for $s\in K$.

Decompose $H_s(u)=u^{-s}(1+u)^{-s}+R_s(u)$, where
\(R_s(u):=\sum_{k\ge1}(k+u)^{-s}(k+1+u)^{-s}\).

Let \(\sigma_-:=\inf_{s\in K}\Re(s)>\tfrac12\). Then for \(u\in[0,1]\),
\[
|(k+u)^{-s}(k+1+u)^{-s}|\le k^{-2\Re(s)}\le k^{-2\sigma_-},
\]
and
\[
\bigl|\partial_u\bigl((k+u)^{-s}(k+1+u)^{-s}\bigr)\bigr|
\ll_K k^{-2\sigma_--1}.
\]
Since \(\sum_{k\ge1}k^{-2\sigma_-}<\infty\) and
\(\sum_{k\ge1}k^{-2\sigma_--1}<\infty\), both the series for \(R_s\) and for
\(\partial_u R_s\) converge uniformly on \([0,1]\), uniformly in \(s\in K\).
Hence \(R_s\in C^1([0,1])\) and \(\sup_{s\in K}\|R_s'\|_\infty<\infty\).

A standard Riemann-sum estimate for $C^1$ functions yields
\[
\sum_{j=1}^N R_s(j/N)=N\int_0^1 R_s(u)\,du+O_K(1).
\]

It remains to treat the singular part \(u^{-s}(1+u)^{-s}=u^{-s}+u^{1-s}Q_s(u)\), where
\(Q_s(u):=\frac{(1+u)^{-s}-1}{u}\)
extends continuously to $u=0$ with $Q_s(0)=-s$.
For $s\in K$, $Q_s$ is $C^1$ on $[0,1]$ with norms bounded uniformly in $s$.

For
\[
f_s(u):=u^{1-s}Q_s(u),
\]
we have
\[
f_s'(u)=(1-s)u^{-s}Q_s(u)+u^{1-s}Q_s'(u).
\]
Since \(K\Subset\{\tfrac12<\Re(s)<1\}\), there exists \(\sigma_+<1\) such that
\(\Re(s)\le \sigma_+\) for all \(s\in K\). As \(Q_s,Q_s'\) are uniformly bounded on
\([0,1]\) for \(s\in K\), it follows that
\[
|f_s'(u)|\ll_K u^{-\sigma_+}+1,
\]
hence
\[
\int_0^1 |f_s'(u)|\,du \ll_K 1.
\]
Therefore the standard Riemann-sum estimate for absolutely continuous functions gives
\[
\sum_{j=1}^N f_s(j/N)
=
N\int_0^1 f_s(u)\,du+O_K(1).
\]

So the contribution of $u^{1-s}Q_s(u)$ to \eqref{eq:Rs} is $O_K(1)$.

Finally, for $u^{-s}$ we use a direct Riemann-sum estimate.
For $s\in K$ (hence $\Re(s)>\tfrac12$ and $s\neq1$), the function $x\mapsto x^{-s}$ is $C^1$ on $[1,\infty)$ and
\[
\sum_{j=1}^N j^{-s}=\int_1^N x^{-s}\,dx+O_K(1).
\]
Indeed,
\[
\sum_{j=1}^N j^{-s}-\int_1^N x^{-s}dx
=\sum_{j=1}^N\int_j^{j+1}\bigl(j^{-s}-x^{-s}\bigr)\,dx+O_K(1),
\]
and the sum is bounded in modulus by
\[\sum_{j\ge1}\int_j^{j+1}|(x^{-s})'|\,dx\ll_K\int_1^{\infty} x^{-\Re(s)-1}dx<\infty.\]
Multiplying by $N^{s}$ yields
\[
\sum_{j=1}^N (j/N)^{-s}=N^{s}\int_1^N x^{-s}dx+O_K(N^{\Re(s)}).
\]
But
\(N^{s}\int_1^N x^{-s}dx=\frac{N-N^{s}}{1-s}\), so
\[
\sum_{j=1}^N (j/N)^{-s}-N\int_0^1 u^{-s}du
=-\frac{N^{s}}{1-s}+O_K(N^{\Re(s)})
=O_K(N^{\Re(s)}),
\]
since $|N^{s}|=N^{\Re(s)}$.
Combining the three pieces proves \eqref{eq:Rs}.

Returning to inclusion--exclusion,
\[
\sum_{d\mid b}\mu(d)S_s(b/d)
=
\left(\sum_{d\mid b}\mu(d)\frac{b}{d}\right)\int_0^1 H_s
+O_K\left(\sum_{d\mid b}(b/d)^{\Re(s)}\right).
\]
The main coefficient equals $\varphi(b)$. The error satisfies
\(\sum_{d\mid b}(b/d)^{\Re(s)}\le b^{\Re(s)}\tau(b)\).
\end{proof}

\subsection{Incomplete Kloosterman sums (Weil bound)}

For integers $n$, $b\ge1$, and $1\le R\le b$, define the incomplete Kloosterman sum
\[
K_b(n;R):=\sum_{\substack{1\le r\le R\\(r,b)=1}} e\!\left(\frac{n\overline r}{b}\right),
\qquad e(x):=e^{2\pi i x}.
\]
\begin{lemma}[Completion bound for \(K_b(n;R)\)]
\label{lem:inckloo}
For all integers \(n\), all \(b\ge 1\), and all \(1\le R\le b\),
\[
K_b(n;R)
\ll
(n,b)^{1/2}\,b^{1/2}\,\tau(b)^2\log(2b),
\]
with an absolute implied constant.

In particular, if \((n,b)=1\), then for every \(\varepsilon>0\),
\[
\max_{1\le R\le b}|K_b(n;R)|\ll_{\varepsilon} b^{1/2+\varepsilon}.
\]
\end{lemma}

\begin{proof}
For \(x\in \Z/b\Z\), define
\[
w_R(x):=\sum_{m=1}^R \ind_{x\equiv m\;(\mathrm{mod}\;b)}.
\]

Its discrete Fourier expansion is
\[
w_R(r)=\frac1b\sum_{h\bmod b}\widehat w_R(h)\,e\!\left(\frac{hr}{b}\right),
\qquad
\widehat w_R(h):=\sum_{m=1}^R e\!\left(-\frac{hm}{b}\right).
\]
Hence
\[
K_b(n;R)
=
\sum_{\substack{r\bmod b\\(r,b)=1}} w_R(r)e\!\left(\frac{n\overline r}{b}\right)
=
\frac1b\sum_{h\bmod b}\widehat w_R(h)\,S(n,h;b),
\]
where
\[
S(n,h;b):=\sum_{\substack{r\bmod b\\(r,b)=1}}
e\!\left(\frac{n\overline r+hr}{b}\right)
\]
is the classical Kloosterman sum.

We separate the zero frequency \(h=0\). Since \(\widehat w_R(0)=R\), we get
\[
\frac1b\,\widehat w_R(0)\,S(n,0;b)
=
\frac{R}{b}\,c_b(n),
\]
where \(c_b(n)=S(n,0;b)\) is the Ramanujan sum. 

Let \(g=(n,b)\). Since
\[
c_b(n)=\mu\!\left(\frac{b}{g}\right)\frac{\varphi(b)}{\varphi(b/g)},
\]
we have
\[
|c_b(n)|\leq \frac{\varphi(b)}{\varphi(b/g)}\le g=(n,b).
\]
Hence
\[
\left|\frac1b\,\widehat w_R(0)\,S(n,0;b)\right|
=
\left|\frac{R}{b}c_b(n)\right|
\le (n,b).
\]
Since \((n,b)\le (n,b)^{1/2}b^{1/2}\), this term is acceptable.

Now consider the nonzero frequencies \(h\neq 0\). For representatives
\[
-\frac b2\le h\le \frac b2,\qquad h\neq 0,
\]
the geometric-series estimate gives
\[
|\widehat w_R(h)|
\ll
\min\!\left(R,\frac{b}{|h|}\right).
\]
Also, by the Weil bound for Kloosterman sums,
\[
|S(n,h;b)|\ll \tau(b)\,(n,h,b)^{1/2}\,b^{1/2}.
\]

Group the nonzero frequencies according to \(d=(h,b)\). Then
\[
(n,h,b)=(n,d).
\]
Therefore
\[
\sum_{\substack{h\bmod b\\ h\neq 0}}
|\widehat w_R(h)|\,|S(n,h;b)|
\ll
\tau(b)b^{1/2}
\sum_{d\mid b}(n,d)^{1/2}
\sum_{\substack{h\bmod b\\ h\neq 0\\ (h,b)=d}}
|\widehat w_R(h)|.
\]
So it remains to bound the inner sum. Write \(h=d h_1\). Then \(h_1\) runs through a subset of nonzero residue classes modulo \(b/d\), and hence
\[
\sum_{\substack{h\bmod b\\ h\neq 0\\ (h,b)=d}}
|\widehat w_R(h)|
\le
\sum_{1\le |m|\le b/(2d)}
\min\!\left(R,\frac{b}{d|m|}\right)
\ll
\frac{b}{d}\log(2b).
\]
Thus
\[
\sum_{\substack{h\bmod b\\ h\neq 0}}
|\widehat w_R(h)|\,|S(n,h;b)|
\ll
\tau(b)b^{1/2}\log(2b)
\sum_{d\mid b}(n,d)^{1/2}\frac{b}{d}.
\]
Dividing by \(b\), we obtain
\[
\frac1b\sum_{\substack{h\bmod b\\ h\neq 0}}
|\widehat w_R(h)|\,|S(n,h;b)|
\ll
\tau(b)b^{1/2}\log(2b)
\sum_{d\mid b}\frac{(n,d)^{1/2}}{d}.
\]
Now
\[
(n,d)^{1/2}\le (n,b)^{1/2},
\]
so
\[
\sum_{d\mid b}\frac{(n,d)^{1/2}}{d}
\le
(n,b)^{1/2}\sum_{d\mid b}\frac1d
\le
(n,b)^{1/2}\tau(b).
\]
Hence the contribution of the nonzero frequencies is
\[
\ll
(n,b)^{1/2}b^{1/2}\tau(b)^2\log(2b).
\]

Combining the zero and nonzero frequencies gives
\[
K_b(n;R)\ll (n,b)^{1/2}b^{1/2}\tau(b)^2\log(2b),
\]
as claimed.

Finally, if \((n,b)=1\), then
\[
K_b(n;R)\ll b^{1/2}\tau(b)^2\log(2b)\ll_\varepsilon b^{1/2+\varepsilon},
\]
since \(\tau(b)^2\log(2b)\ll_\varepsilon b^\varepsilon\).
\end{proof}

\subsection{A Fej\'er approximation bound for periodic Lipschitz functions}

Let \(F_N\) denote the Fej\'er kernel on the circle,
\[
F_N(t):=\frac{1}{N}\left(\frac{\sin(\pi N t)}{\sin(\pi t)}\right)^2,\qquad t\in[-1/2,1/2],
\]
extended \(1\)-periodically.

\begin{lemma}[First moment of the Fej\'er kernel]
\label{lem:fejer_moment}
There exists an absolute constant \(C>0\) such that for all \(N\ge2\),
\[
\int_{-1/2}^{1/2} |t|\,F_N(t)\,dt\le C\,\frac{\log N}{N}.
\]
Consequently, if \(G\) is \(1\)-periodic and Lipschitz, then its Fej\'er mean \(G*F_N\) satisfies
\[
\|G-G*F_N\|_{\infty}\le C\,\frac{\log N}{N}\,\operatorname{Lip}(G).
\]
\end{lemma}

\begin{proof}
Since \(F_N\ge0\) and \(\int_{-1/2}^{1/2}F_N=1\),
\[
G(x)-(G*F_N)(x)=\int_{-1/2}^{1/2}\bigl(G(x)-G(x-t)\bigr)F_N(t)\,dt.
\]
By Lipschitz continuity,
\[
|G(x)-G(x-t)|\le \operatorname{Lip}(G)\,|t|,
\]
which gives the second inequality once the moment bound is proved.

For the moment bound, split the integral into \(|t|\le 1/N\) and \(1/N<|t|\le 1/2\).
For \(|t|\le 1/N\), use \(F_N(t)\le N\) to get
\[
\int_{|t|\le 1/N}|t|F_N(t)\,dt
\ll
N\int_0^{1/N} t\,dt
\ll
\frac1N.
\]
For \(1/N<|t|\le 1/2\), use \(|\sin(\pi t)|\gg |t|\) and \(|\sin(\pi Nt)|\le1\) to obtain
\[
F_N(t)\ll \frac{1}{Nt^2}.
\]
Hence
\[
\int_{1/N<|t|\le1/2}|t|F_N(t)\,dt
\ll
\frac1N\int_{1/N}^{1/2}\frac{dt}{t}
\ll
\frac{\log N}{N}.
\]
\end{proof}

\subsection{Proof of Proposition~\ref{prop:spec}}

\begin{proof}[Proof of Proposition~\ref{prop:spec}]\label{proof:prop:spec}
Fix \(\varepsilon>0\), and write \(\sigma:=\Re(s)\). Let
\[
B_s(x):=|f''(x)|^s,
\qquad
A_r:=H_s(r/b).
\]
We may assume \(b\ge4\), because the finitely many smaller values of \(b\) can be absorbed into the implied constant. Set
\[
N:=R:=\lfloor b^{1/2}\rfloor,
\qquad
\eta:=\varepsilon/4.
\]

Define the affine endpoint interpolant
\[
\ell_s(x):=(1-x)B_s(0)+xB_s(1),
\]
and the remainder
\[
G_s(x):=B_s(x)-\ell_s(x).
\]
Then
\[
G_s(0)=G_s(1)=0.
\]
By \eqref{eq:Bs_bounds}, the family \(B_s\) is \(C^1\) with norms bounded uniformly for \(s\in K\). Hence
\[
|B_s(1)-B_s(0)|\ll_{f,K}1,
\qquad
\operatorname{Lip}(G_s)\ll_{f,K}1,
\]
uniformly for \(s\in K\). In particular, \(G_s\) admits a \(1\)-periodic Lipschitz extension to the circle.

\emph{Step 1: replace the periodic remainder \(G_s\) by its Fej\'er mean.}
By Lemma~\ref{lem:fejer_moment},
\[
\|G_s-(G_s)_N\|_\infty \ll_{f,K} \frac{\log(2N)}{N},
\]
uniformly for \(s\in K\). Therefore
\[\!
\Sigma_b(s)
=
\sum_{(r,b)=1} A_r\,(G_s)_N(\overline r/b)
+
\sum_{(r,b)=1} A_r\,\ell_s(\overline r/b)
+
O_{f,K}\!\left(\frac{\log(2N)}{N}\sum_{(r,b)=1}|A_r|\right).
\]
Now Lemma~\ref{lem:Hsbound} yields
\[
\sum_{(r,b)=1}|A_r|
\le
\sum_{r=1}^b |H_s(r/b)|
\ll_K
\sum_{r=1}^b (r/b)^{-\sigma}
=
b^\sigma\sum_{r=1}^b r^{-\sigma}
\ll_K b,
\]
uniformly for \(s\in K\), since \(K\Subset\{\tfrac12<\Re(s)<1\}\). Thus
\begin{equation}
\label{eq:step1}
\Sigma_b(s)
=
\sum_{(r,b)=1} A_r\,(G_s)_N(\overline r/b)
+
\sum_{(r,b)=1} A_r\,\ell_s(\overline r/b)
+
O_{f,K}\bigl(b^{1/2}\log(2b)\bigr).
\end{equation}

\emph{Step 2: expand the periodic part in Fourier modes.}
Write
\[
(G_s)_N(x)=\widehat G_s(0)+\sum_{0<|n|<N}\widehat{(G_s)_N}(n)e(nx),
\qquad
\widehat G_s(0)=\int_0^1 G_s(x)\,dx.
\]
Since \(G_s\) is Lipschitz with \(\operatorname{Lip}(G_s)\ll_{f,K}1\), its Fourier coefficients satisfy
\[
|\widehat G_s(n)|\ll_{f,K}\frac1{|n|}\qquad (n\neq0),
\]
uniformly for \(s\in K\), and hence
\[
|\widehat{(G_s)_N}(n)|\le |\widehat G_s(n)|\ll_{f,K}\frac1{|n|}
\qquad (0<|n|<N).
\]
Therefore
\[
\sum_{(r,b)=1} A_r\,(G_s)_N(\overline r/b)
=
\widehat G_s(0)\sum_{(r,b)=1}A_r
+
\sum_{0<|n|<N}\widehat{(G_s)_N}(n)\,S_b(n),
\]
where
\[
S_b(n):=\sum_{(r,b)=1}A_r\,e\!\left(\frac{n\overline r}{b}\right).
\]

\emph{Step 3: expand the affine part exactly on the residue grid.}
Let
\[
\ell_s(x)=\alpha_s+\beta_s x,
\qquad
\beta_s=B_s(1)-B_s(0).
\]
For \(m=0,\dots,b-1\), the values \(\ell_s(m/b)\) admit the discrete Fourier expansion
\[
\ell_s(m/b)=c_{b,s}(0)+\sum_{0<|n|\le b/2} c_{b,s}(n)e(nm/b),
\]
where the sum is taken over any fixed symmetric system of representatives of the nonzero residue classes modulo \(b\). The coefficients satisfy
\[
c_{b,s}(0)
=
\frac1b\sum_{m=0}^{b-1}\ell_s(m/b)
=
\int_0^1 \ell_s(x)\,dx+O_{f,K}(b^{-1}),
\]
and, for \(0<|n|\le b/2\),
\[
c_{b,s}(n)
=
-\frac{\beta_s}{b\bigl(1-e(-n/b)\bigr)},
\qquad
|c_{b,s}(n)|\ll_{f,K}\frac1{|n|},
\]
uniformly for \(s\in K\). Hence
\[
\sum_{(r,b)=1} A_r\,\ell_s(\overline r/b)
=
c_{b,s}(0)\sum_{(r,b)=1}A_r
+
\sum_{0<|n|\le b/2} c_{b,s}(n)\,S_b(n).
\]

\emph{Step 4: the main term.}
By Lemma~\ref{lem:unitavg},
\[
\sum_{(r,b)=1}A_r
=
\varphi(b)\int_0^1 H_s(u)\,du
+
O_K\bigl(b^\sigma\tau(b)\bigr),
\]
uniformly for \(s\in K\). Also,
\[\mkern-26mu
\widehat G_s(0)+c_{b,s}(0)
=
\int_0^1 G_s(x)\,dx+\int_0^1 \ell_s(x)\,dx+O_{f,K}(b^{-1})
=
\int_0^1 B_s(x)\,dx+O_{f,K}(b^{-1}).
\]
Therefore the contribution of the zero Fourier mode is
\[
\varphi(b)\left(\int_0^1 H_s(u)\,du\right)\left(\int_0^1 B_s(x)\,dx\right)
+
O_{f,K}\bigl(b^\sigma\tau(b)\bigr)
+
O_{f,K}(1).
\]
Since \(\sigma<1\), the error \(O_{f,K}(b^\sigma\tau(b))\) is dominated by the final error term.

\emph{Step 5: bound the oscillatory sums \(S_b(n)\).}
Let \(n\neq0\). Split
\[
S_b(n)=S_b^{\le R}(n)+S_b^{>R}(n),
\]
where
\[
S_b^{\le R}(n):=\sum_{\substack{1\le r\le R\\(r,b)=1}}A_r\,e\!\left(\frac{n\overline r}{b}\right),
\qquad
S_b^{>R}(n):=\sum_{\substack{R<r\le b\\(r,b)=1}}A_r\,e\!\left(\frac{n\overline r}{b}\right).
\]

For the initial segment, Lemma~\ref{lem:Hsbound} gives
\[
|S_b^{\le R}(n)|
\le \sum_{r\le R}|A_r|
\ll_K b^\sigma\sum_{r\le R}r^{-\sigma}
\ll_K b^\sigma R^{1-\sigma}
\ll_K b^{\frac{1+\sigma}{2}}.
\]

For the tail, define
\[
a_r:=\ind_{(r,b)=1}\,e\!\left(\frac{n\overline r}{b}\right),
\qquad
T_n(t):=\sum_{1\le r\le t} a_r.
\]
Then Lemma~\ref{lem:inckloo} gives
\[
\max_{1\le t\le b}|T_n(t)|
\ll_\eta (n,b)^{1/2}b^{1/2+\eta}.
\]
By Abel summation,
\[\mkern-10mu
S_b^{>R}(n)
=
H_s(1)\,T_n(b)-H_s(R/b)\,T_n(R)
-\sum_{t=R}^{b-1}T_n(t)\Bigl(H_s\!\left(\frac{t+1}{b}\right)-H_s\!\left(\frac{t}{b}\right)\Bigr).
\]
Using Lemma~\ref{lem:Hsbound},
\[
|H_s(1)|\ll_K1,
\qquad
|H_s(R/b)|\ll_K (R/b)^{-\sigma}.
\]
Also, by the mean value theorem and Lemma~\ref{lem:Hsderiv},
\[\mkern-22mu
\left|H_s\!\left(\frac{t+1}{b}\right)-H_s\!\left(\frac{t}{b}\right)\right|
\le
\frac1b\sup_{u\in[t/b,(t+1)/b]}|\partial_u H_s(u)|
\ll_K
\frac1b\left(\frac{t}{b}\right)^{-\sigma-1}
\mkern-12mu\ll_K b^\sigma t^{-\sigma-1}.
\]
Hence
\[
\sum_{t=R}^{b-1}
\left|H_s\!\left(\frac{t+1}{b}\right)-H_s\!\left(\frac{t}{b}\right)\right|
\ll_K
b^\sigma\sum_{t\ge R}t^{-\sigma-1}
\ll_K b^\sigma R^{-\sigma}.
\]
Therefore
\[
|S_b^{>R}(n)|
\ll_{K,\eta}
(n,b)^{1/2}b^{1/2+\eta}\Bigl(1+(R/b)^{-\sigma}+b^\sigma R^{-\sigma}\Bigr)
\ll_{K,\eta}
(n,b)^{1/2}b^{\frac{1+\sigma}{2}+\eta},
\]
since \(R\asymp b^{1/2}\). Together with the initial-segment bound, this gives
\begin{equation}
\label{eq:Sbn-corrected}
S_b(n)\ll_{K,\eta}(n,b)^{1/2}b^{\frac{1+\sigma}{2}+\eta}
\qquad (n\neq0).
\end{equation}

\emph{Step 6: sum the nonzero Fourier modes.}
For the periodic part, using \eqref{eq:Sbn-corrected},
\[
\sum_{0<|n|<N}|\widehat{(G_s)_N}(n)\,S_b(n)|
\ll_{f,K,\eta}
b^{\frac{1+\sigma}{2}+\eta}
\sum_{1\le n<N}\frac{(n,b)^{1/2}}{n}.
\]
Grouping by \(d=(n,b)\), we obtain
\[
\sum_{1\le n<N}\frac{(n,b)^{1/2}}{n}
\le
\sum_{d\mid b}d^{-1/2}\sum_{m<N/d}\frac1m
\ll
\tau(b)\log(2N).
\]
Hence
\[
\sum_{0<|n|<N}\widehat{(G_s)_N}(n)\,S_b(n)
=
O_{f,K,\eta}\!\left(
b^{\frac{1+\sigma}{2}+\eta}\tau(b)\log(2N)
\right).
\]

Similarly, for the affine part,
\[
\sum_{0<|n|\le b/2}|c_{b,s}(n)\,S_b(n)|
\ll_{f,K,\eta}
b^{\frac{1+\sigma}{2}+\eta}
\sum_{1\le n\le b/2}\frac{(n,b)^{1/2}}{n}
\ll
b^{\frac{1+\sigma}{2}+\eta}\tau(b)\log(2b).
\]

Since \(\tau(b)\log(2b)\ll_\eta b^{2\eta}\), both nonzero-mode contributions are
\[
O_{f,K,\varepsilon}\!\left(b^{\frac{1+\sigma}{2}+\varepsilon}\right).
\]
The truncation error in \eqref{eq:step1} also satisfies
\[
b^{1/2}\log(2b)\ll_{f,K,\varepsilon} b^{\frac{1+\sigma}{2}+\varepsilon},
\]
because \(\sigma>\tfrac12\).

Putting everything together, and recalling that
\[
\int_0^1 B_s(x)\,dx=\int_0^1 |f''(x)|^s\,dx,
\]
we conclude that
\[
\Sigma_b(s)
=
\varphi(b)\left(\int_0^1 H_s(u)\,du\right)\left(\int_0^1 |f''(x)|^s\,dx\right)
+
O_{f,K,\varepsilon}\!\left(b^{\frac{1+\Re(s)}{2}+\varepsilon}\right),
\]
uniformly for \(s\in K\). This proves the proposition.
\end{proof}


\subsection{Strip-uniform refinement and half-plane continuation}
\label{ssec:striprefinement}

In this subsection we strengthen the local continuation statement near \(s=\tfrac23\)
to meromorphic continuation in the half-plane \(\Re(s)>\tfrac35\), and we make the
dependence on the imaginary part quantitative on vertical strips. This is the analytic
input needed in Appendix~\ref{app2}, subsection~\ref{ssec:analyticcompletion}.

\begin{lemma}[Endpoint replacement on \(\Re(s)>\tfrac12\)]
\label{lem:replace-halfplane}
Assume \(f\in C^3([0,1])\) and \(0<m\le |f''(x)|\) on \([0,1]\).
For every compact set
\[
K\Subset \{\,s\in\C:\Re(s)>\tfrac12\,\},
\]
the function
\[
Z_f(s)-Z_f^{\mathrm{end}}(s)
\]
is holomorphic on a neighborhood of \(K\). In particular, \(Z_f-Z_f^{\mathrm{end}}\)
is holomorphic on the half-plane \(\Re(s)>\tfrac12\).
\end{lemma}

\begin{proof}
Choose a compact set
\[
K'\Subset \{\,s\in\C:\Re(s)>\tfrac12\,\}
\]
such that \(K\subset \operatorname{int}(K')\). Set
\[
\sigma_-:=\inf_{s\in K'}\Re(s)>\tfrac12.
\]

Recall that for \(s\in\C\) and \(x\in[0,1]\) we write
\[
B_s(x):=|f''(x)|^s.
\]
By the regularity discussion at the beginning of this appendix, the family \(\{B_s\}_{s\in K'}\)
is uniformly \(C^1\), so there exists \(C_{K'}>0\) such that
\[
\sup_{s\in K'}\|(B_s)'\|_\infty\le C_{K'}.
\]
Hence for all \(u,v\in[0,1]\) and all \(s\in K'\),
\[
\bigl||f''(u)|^s-|f''(v)|^s\bigr|
=
|B_s(u)-B_s(v)|
\le C_{K'}|u-v|.
\]

Now let \(I=[c/d,a/b]\) be a Farey interval. By Lemma~\ref{lemma:exactT},
\[
Z_f(s)=2^{-s}\sum_I \frac{|f''(\xi_I)|^s}{(bd(b+d))^s},
\]
while by definition
\[
Z_f^{\mathrm{end}}(s)=2^{-s}\sum_I \frac{|f''(a/b)|^s}{(bd(b+d))^s}.
\]
Since \(\xi_I\in(c/d,a/b)\), we have
\[
\left|\xi_I-\frac ab\right|
\le
\left|\frac ab-\frac cd\right|
=
\frac1{bd}.
\]
Therefore, uniformly for \(s\in K'\),
\[\mkern-6mu
\left|
\frac{|f''(\xi_I)|^s-|f''(a/b)|^s}{(bd(b+d))^s}
\right|
\le
\frac{C_{K'}}{bd\,(bd(b+d))^{\Re(s)}}
\le
C_{K'}\,b^{-(\sigma_-+1)}d^{-(\sigma_-+1)}(b+d)^{-\sigma_-}.
\]

Thus the series defining \(Z_f-Z_f^{\mathrm{end}}\) is dominated on \(K'\) by
\[
\sum_I b^{-(\sigma_-+1)}d^{-(\sigma_-+1)}(b+d)^{-\sigma_-}.
\]
Dropping the Farey restriction only enlarges the sum, so it is enough to consider
\[
\sum_{b,d\ge1} b^{-(\sigma_-+1)}d^{-(\sigma_-+1)}(b+d)^{-\sigma_-}.
\]
Using \(b+d\ge b\), we obtain
\[
b^{-(\sigma_-+1)}d^{-(\sigma_-+1)}(b+d)^{-\sigma_-}
\le
b^{-(2\sigma_-+1)}d^{-(\sigma_-+1)}.
\]
Hence
\[
\sum_{b,d\ge1} b^{-(\sigma_-+1)}d^{-(\sigma_-+1)}(b+d)^{-\sigma_-}
\le
\left(\sum_{b\ge1} b^{-(2\sigma_-+1)}\right)
\left(\sum_{d\ge1} d^{-(\sigma_-+1)}\right).
\]
Both series converge because \(\sigma_->\tfrac12\). Therefore the defining series for
\(Z_f-Z_f^{\mathrm{end}}\) converges uniformly on \(K'\) by the Weierstrass \(M\)-test.
Since each term is holomorphic in \(s\), the sum is holomorphic on \(\operatorname{int}(K')\),
hence on a neighborhood of \(K\).
\end{proof}

\begin{lemma}[Strip-uniform bounds for \(H_s\) and \(\partial_u H_s\)]
\label{lem:Hs-strip}
Fix numbers
\[
\frac12<\alpha<\beta<1.
\]
Then there exist constants \(C_1=C_1(\alpha)\) and \(C_2=C_2(\alpha,\beta)\) such that for every
\[
s=\sigma+i\tau,\qquad \alpha\le \sigma\le \beta,
\]
and every \(u\in(0,1]\), one has
\[
|H_s(u)|\le C_1\,u^{-\sigma},
\]
and
\[
|\partial_u H_s(u)|\le C_2\,(1+|\tau|)\,u^{-\sigma-1}.
\]
\end{lemma}

\begin{proof}
Recall
\[
H_s(u)=\sum_{k\ge0}(k+u)^{-s}(k+1+u)^{-s},
\qquad \Re(s)>\frac12,\quad u\in(0,1].
\]

We first prove the bound for \(H_s(u)\). Split off the \(k=0\) term:
\[
H_s(u)=u^{-s}(1+u)^{-s}+\sum_{k\ge1}(k+u)^{-s}(k+1+u)^{-s}.
\]
Since \(k+u>0\) and \(k+1+u>0\), we have
\[
|(k+u)^{-s}(k+1+u)^{-s}|=(k+u)^{-\sigma}(k+1+u)^{-\sigma}.
\]
Hence
\[
|u^{-s}(1+u)^{-s}|
=
u^{-\sigma}(1+u)^{-\sigma}
\le u^{-\sigma}.
\]
Also, for \(k\ge1\),
\[
(k+u)^{-\sigma}(k+1+u)^{-\sigma}\le k^{-2\sigma}\le k^{-2\alpha}.
\]
Therefore
\[
\sum_{k\ge1}(k+u)^{-\sigma}(k+1+u)^{-\sigma}
\le \sum_{k\ge1}k^{-2\alpha}
=:C_\alpha<\infty.
\]
Since \(u\in(0,1]\), we have \(u^{-\sigma}\ge1\), so
\[
C_\alpha\le C_\alpha\,u^{-\sigma}.
\]
Combining the two estimates gives
\[
|H_s(u)|\le (1+C_\alpha)\,u^{-\sigma},
\]
which proves the first bound.

We now turn to the derivative. On every interval \(u\in[\delta,1]\) with \(\delta>0\), the defining
series and the differentiated series converge uniformly, so termwise differentiation is justified.
We obtain
\[
\partial_u H_s(u)
=
-s\sum_{k\ge0}(k+u)^{-s-1}(k+1+u)^{-s}
-s\sum_{k\ge0}(k+u)^{-s}(k+1+u)^{-s-1}.
\]

For the \(k=0\) term, we have
\[
|s|\Bigl(u^{-\sigma-1}(1+u)^{-\sigma}+u^{-\sigma}(1+u)^{-\sigma-1}\Bigr)
\le 2|s|\,u^{-\sigma-1}.
\]
For the tail \(k\ge1\), each summand is bounded by
\[\mkern-22mu
|s|\,(k+u)^{-\sigma-1}(k+1+u)^{-\sigma}
+
|s|\,(k+u)^{-\sigma}(k+1+u)^{-\sigma-1}
\le 2|s|\,k^{-2\sigma-1}\le 2|s|\,k^{-2\alpha-1}.
\]
Since \(\sum_{k\ge1}k^{-2\alpha-1}<\infty\), the full tail is bounded by \(C'_\alpha|s|\), and
again \(u^{-\sigma-1}\ge1\) on \((0,1]\). Therefore
\[
|\partial_u H_s(u)|\le C''_\alpha |s|\,u^{-\sigma-1}.
\]
On the strip \(\alpha\le \sigma\le \beta\),
\[
|s|=|\sigma+i\tau|\le \beta+|\tau|\le (1+\beta)(1+|\tau|),
\]
so
\[
|\partial_u H_s(u)|\le C_2(\alpha,\beta)\,(1+|\tau|)\,u^{-\sigma-1}.
\]
\end{proof}

\begin{lemma}[Strip-uniform bounds for \(B_s(x)=|f''(x)|^s\)]
\label{lem:Bs-strip}
Fix numbers
\[
\frac12<\alpha<\beta<1.
\]
Assume
\[
f\in C^3([0,1]),
\qquad
0<m\le |f''(x)|\le M
\quad (x\in[0,1]).
\]
Then there exist constants \(C_0,C_1,C_2>0\), depending only on \(\alpha,\beta,f\), such that for all
\[
s=\sigma+i\tau,\qquad \alpha\le \sigma\le \beta,
\]
and all \(x\in[0,1]\), one has
\[
|B_s(x)|\le C_0,\qquad \|B_s\|_\infty\le C_1,\qquad \|B_s'\|_\infty\le C_2(1+|\tau|).
\]
In particular, the affine interpolant/remainder decomposition used in the proof of
Proposition~\ref{prop:spec} remains valid on strips, with polynomial dependence on \(|\tau|\).
\end{lemma}

\begin{proof}
Since \(0<m\le |f''(x)|\le M\) on \([0,1]\), for \(s=\sigma+i\tau\) and \(x\in[0,1]\),
\[
|B_s(x)|
=
\bigl||f''(x)|^s\bigr|
=
|f''(x)|^\sigma
\le M^\beta.
\]
Thus we may take
\[
C_0=C_1=M^\beta.
\]

Now differentiate. Since \(f''\) has constant sign on \([0,1]\) and is bounded away from \(0\),
the function \(\log|f''(x)|\) is \(C^1\), and
\[
B_s'(x)
=
s\,|f''(x)|^s\,\frac{f'''(x)}{f''(x)}.
\]
Therefore
\[
|B_s'(x)|
\le
|s|\,|f''(x)|^\sigma\left|\frac{f'''(x)}{f''(x)}\right|
\le
|s|\,M^\beta\,\frac{\|f'''\|_\infty}{m}.
\]
On the strip \(\alpha\le \sigma\le \beta\),
\[
|s|=|\sigma+i\tau|\le \beta+|\tau|\le (1+\beta)(1+|\tau|),
\]
so
\[
\|B_s'\|_\infty
\le
(1+\beta)\,M^\beta\,\frac{\|f'''\|_\infty}{m}\,(1+|\tau|).
\]

Finally, in the proof of Proposition~\ref{prop:spec}, one writes
\[
\ell_s(x):=(1-x)B_s(0)+xB_s(1),
\qquad
G_s(x):=B_s(x)-\ell_s(x).
\]
The bounds above imply
\[
|B_s(1)-B_s(0)|\le 2\|B_s\|_\infty\ll_{\alpha,\beta,f}1,
\]
hence
\[
\operatorname{Lip}(\ell_s)\ll_{\alpha,\beta,f}1,
\]
while
\[
\operatorname{Lip}(G_s)\le \|B_s'\|_\infty+\operatorname{Lip}(\ell_s)
\ll_{\alpha,\beta,f}(1+|\tau|).
\]
\end{proof}

\begin{proposition}[Strip-uniform equidistribution estimate]
\label{prop:spec-strip}
Fix numbers
\[
\frac12<\alpha<\beta<1.
\]
Assume
\[
f\in C^3([0,1]),
\qquad
0<m\le |f''(x)|\le M
\quad (x\in[0,1]).
\]
Then for every \(\varepsilon>0\) there exists \(C=C(\alpha,\beta,f,\varepsilon)>0\) such that for every
\[
s=\sigma+i\tau,\qquad \alpha\le \sigma\le \beta,
\]
and every \(b\ge1\), one has
\[
\Sigma_b(s)
=
\varphi(b)\left(\int_0^1 H_s(u)\,du\right)\left(\int_0^1 |f''(v)|^s\,dv\right)
+
O\!\left((1+|\tau|)^2\,b^{\frac{1+\sigma}{2}+\varepsilon}\right),
\]
where
\[
\Sigma_b(s):=\sum_{\substack{1\le r\le b\\(r,b)=1}} H_s(r/b)\,|f''(\overline r/b)|^s.
\]
\end{proposition}

\begin{proof}
Write
\[
B_s(x):=|f''(x)|^s,
\qquad
A_r:=H_s(r/b).
\]

We first record the strip-uniform analogue of Lemma~\ref{lem:unitavg}:
\begin{equation}
\label{eq:unitavg-strip}
\sum_{\substack{1\le r\le b\\(r,b)=1}} H_s(r/b)
=
\varphi(b)\int_0^1 H_s(u)\,du
+
O_{\alpha,\beta}\!\left((1+|\tau|)\,b^\sigma \tau(b)\right).
\end{equation}
Indeed, the proof of Lemma~\ref{lem:unitavg} carries over verbatim once one replaces the
compact-uniform bounds used there by the strip-uniform bounds of Lemma~\ref{lem:Hs-strip}.
The only place where \(|\tau|\) appears is through the derivative bound for \(H_s\), and this yields
the factor \((1+|\tau|)\).

Now follow the proof of Proposition~\ref{prop:spec}. Let
\[
\ell_s(x):=(1-x)B_s(0)+xB_s(1),
\qquad
G_s(x):=B_s(x)-\ell_s(x).
\]
By Lemma~\ref{lem:Bs-strip},
\[
\|B_s\|_\infty\ll_{\alpha,\beta,f}1,
\qquad
\operatorname{Lip}(G_s)\ll_{\alpha,\beta,f}(1+|\tau|).
\]

Set
\[
N:=R:=\lfloor b^{1/2}\rfloor.
\]
By the Fej\'er-kernel estimate (Lemma~\ref{lem:fejer_moment}),
\[
\|G_s-(G_s)_N\|_\infty
\ll
\frac{\log(2N)}{N}\operatorname{Lip}(G_s)
\ll_{\alpha,\beta,f}
(1+|\tau|)\frac{\log(2N)}{N}.
\]
Therefore
\[\mkern-25mu
\Sigma_b(s)
=
\sum_{(r,b)=1}A_r\,(G_s)_N(\overline r/b)
+
\sum_{(r,b)=1}A_r\,\ell_s(\overline r/b)
+
O\!\left((1+|\tau|)\frac{\log(2N)}{N}\sum_{(r,b)=1}|A_r|\right).
\]
By Lemma~\ref{lem:Hs-strip},
\[
|A_r|=|H_s(r/b)|\ll_\alpha (r/b)^{-\sigma},
\]
so
\[
\sum_{(r,b)=1}|A_r|
\le
\sum_{r=1}^b |H_s(r/b)|
\ll_\alpha
b^\sigma\sum_{r=1}^b r^{-\sigma}
\ll_\alpha b.
\]
Hence the truncation error is
\begin{equation}
\label{eq:trunc-strip}
O\!\left((1+|\tau|)\,b^{1/2}\log(2b)\right).
\end{equation}

Next expand the periodic part:
\[
(G_s)_N(x)=\widehat G_s(0)+\sum_{0<|n|<N}\widehat{(G_s)_N}(n)e(nx).
\]
Since \(G_s\) is Lipschitz with
\[
\operatorname{Lip}(G_s)\ll (1+|\tau|),
\]
its Fourier coefficients satisfy
\[
|\widehat G_s(n)|\ll_{\alpha,\beta,f}\frac{1+|\tau|}{|n|}
\qquad (n\neq0),
\]
hence
\[
|\widehat{(G_s)_N}(n)|\le |\widehat G_s(n)|
\ll_{\alpha,\beta,f}\frac{1+|\tau|}{|n|}.
\]
So
\[
\sum_{(r,b)=1}A_r\,(G_s)_N(\overline r/b)
=
\widehat G_s(0)\sum_{(r,b)=1}A_r
+
\sum_{0<|n|<N}\widehat{(G_s)_N}(n)\,S_b(n),
\]
where
\[
S_b(n):=\sum_{(r,b)=1}A_r\,e\!\left(\frac{n\overline r}{b}\right).
\]

For the affine part, write
\[
\ell_s(x)=\alpha_s+\beta_s x,
\qquad \beta_s=B_s(1)-B_s(0).
\]
Since \(\|B_s\|_\infty\ll1\), we have \(|\beta_s|\ll1\). The discrete Fourier expansion on the \(b\)-grid gives
\[
\ell_s(m/b)=c_{b,s}(0)+\sum_{0<|n|\le b/2}c_{b,s}(n)e(nm/b),
\]
with
\[
c_{b,s}(0)=\int_0^1 \ell_s(x)\,dx+O(b^{-1}),
\qquad
|c_{b,s}(n)|\ll_{\alpha,\beta,f}\frac{1}{|n|}
\quad (0<|n|\le b/2).
\]
Therefore
\[
\sum_{(r,b)=1}A_r\,\ell_s(\overline r/b)
=
c_{b,s}(0)\sum_{(r,b)=1}A_r
+
\sum_{0<|n|\le b/2}c_{b,s}(n)S_b(n).
\]

We now bound \(S_b(n)\). Split
\[
S_b(n)=S_b^{\le R}(n)+S_b^{>R}(n),
\]
where
\[
S_b^{\le R}(n):=\sum_{\substack{1\le r\le R\\(r,b)=1}}A_r e\!\left(\frac{n\overline r}{b}\right),
\qquad
S_b^{>R}(n):=\sum_{\substack{R<r\le b\\(r,b)=1}}A_r e\!\left(\frac{n\overline r}{b}\right).
\]

For the initial segment,
\[
|S_b^{\le R}(n)|
\le
\sum_{r\le R}|A_r|
\ll_\alpha
b^\sigma\sum_{r\le R}r^{-\sigma}
\ll_\alpha
b^\sigma R^{1-\sigma}
\ll_\alpha
b^{\frac{1+\sigma}{2}}.
\]

For the tail, define
\[
a_r:=\ind_{(r,b)=1}e\!\left(\frac{n\overline r}{b}\right),
\qquad
T_n(T):=\sum_{1\le r\le T}a_r.
\]
By Lemma~\ref{lem:inckloo},
\[
\max_{1\le T\le b}|T_n(T)|
\ll_\eta
(n,b)^{1/2}b^{1/2+\eta}
\]
for any \(\eta>0\). Abel summation gives
\[
S_b^{>R}(n)
=
H_s(1)\,T_n(b)-H_s(R/b)\,T_n(R)
-\sum_{T=R}^{b-1}T_n(T)\Bigl(H_s((T+1)/b)-H_s(T/b)\Bigr).
\]
By Lemma~\ref{lem:Hs-strip},
\[
|H_s(1)|\ll_\alpha 1,
\qquad
|H_s(R/b)|\ll_\alpha (R/b)^{-\sigma}.
\]
Also, by the mean value theorem and Lemma~\ref{lem:Hs-strip},
\[
|H_s((T+1)/b)-H_s(T/b)|
\le
\frac1b\sup_{u\in[T/b,(T+1)/b]}|\partial_u H_s(u)|
\ll_{\alpha,\beta}
(1+|\tau|)\,b^\sigma T^{-\sigma-1}.
\]
Hence
\[
\sum_{T=R}^{b-1}|H_s((T+1)/b)-H_s(T/b)|
\ll_{\alpha,\beta}
(1+|\tau|)\,b^\sigma\sum_{T\ge R}T^{-\sigma-1}
\ll_{\alpha,\beta}
(1+|\tau|)\,b^\sigma R^{-\sigma}.
\]
Therefore
\[
|S_b^{>R}(n)|
\ll_{\alpha,\beta,\eta}
(n,b)^{1/2}b^{1/2+\eta}
\left(
1+(R/b)^{-\sigma}+(1+|\tau|)b^\sigma R^{-\sigma}
\right).
\]
Since \(R=\lfloor b^{1/2}\rfloor\), both \((R/b)^{-\sigma}\) and \(b^\sigma R^{-\sigma}\) are
\(\ll b^{\sigma/2}\). Thus
\[
S_b(n)\ll_{\alpha,\beta,\eta}
(1+|\tau|)\,(n,b)^{1/2}\,b^{\frac{1+\sigma}{2}+\eta}.
\]

Now sum the nonzero Fourier modes. For the periodic part,
\[
\sum_{0<|n|<N}|\widehat{(G_s)_N}(n)S_b(n)|
\ll_{\alpha,\beta,f,\eta}
(1+|\tau|)^2 b^{\frac{1+\sigma}{2}+\eta}
\sum_{1\le n<N}\frac{(n,b)^{1/2}}{n}.
\]
Grouping by \(d=(n,b)\) exactly as in the proof of Proposition~\ref{prop:spec},
\[
\sum_{1\le n<N}\frac{(n,b)^{1/2}}{n}
\ll \tau(b)\log(2N).
\]
Therefore
\[
\sum_{0<|n|<N}\widehat{(G_s)_N}(n)S_b(n)
=
O\!\left((1+|\tau|)^2 b^{\frac{1+\sigma}{2}+\eta}\tau(b)\log(2N)\right).
\]

Similarly, for the affine part,
\[
\sum_{0<|n|\le b/2}|c_{b,s}(n)S_b(n)|
\ll_{\alpha,\beta,f,\eta}
(1+|\tau|)\,b^{\frac{1+\sigma}{2}+\eta}
\sum_{1\le n\le b/2}\frac{(n,b)^{1/2}}{n},
\]
hence
\[
\sum_{0<|n|\le b/2}c_{b,s}(n)S_b(n)
=
O\!\left((1+|\tau|)\,b^{\frac{1+\sigma}{2}+\eta}\tau(b)\log(2b)\right).
\]

For the zero mode,
\[
\widehat G_s(0)+c_{b,s}(0)
=
\int_0^1 B_s(x)\,dx+O(b^{-1})
=
\int_0^1 |f''(x)|^s\,dx+O(b^{-1}),
\]
so by \eqref{eq:unitavg-strip},
\[
\bigl(\widehat G_s(0)+c_{b,s}(0)\bigr)\sum_{(r,b)=1}A_r
=
\varphi(b)\left(\int_0^1 H_s(u)\,du\right)\left(\int_0^1 |f''(x)|^s\,dx\right)
+
O\!\left((1+|\tau|)\,b^\sigma\tau(b)\right)
+
O(1).
\]

Since \(\sigma<1\), we have
\[
b^\sigma\tau(b)\ll_\varepsilon b^{\frac{1+\sigma}{2}+\varepsilon},
\]
and because \(\sigma>\tfrac12\),
\[
b^{1/2}\log(2b)\ll_\varepsilon b^{\frac{1+\sigma}{2}+\varepsilon}.
\]
Absorbing the factors \(\tau(b)\log(2b)\) into \(b^\varepsilon\), and then choosing \(\eta>0\)
sufficiently small relative to the final \(\varepsilon\), all the error terms combine into
\[
O\!\left((1+|\tau|)^2\,b^{\frac{1+\sigma}{2}+\varepsilon}\right).
\]
This proves the proposition.
\end{proof}

\begin{proposition}[Strip-uniform holomorphy and growth of the error term]
\label{prop:E-strip}
Fix numbers
\[
\frac35<\alpha<\beta<1.
\]
Define
\[
A(s):=
2^{-s}\left(\int_0^1 H_s(u)\,du\right)\left(\int_0^1 |f''(v)|^s\,dv\right)
\]
and
\[
E(s):=
2^{-s}\sum_{b\ge1} b^{-3s}R_b(s),
\]
where
\[
\Sigma_b(s)
=
\varphi(b)\left(\int_0^1 H_s(u)\,du\right)\left(\int_0^1 |f''(v)|^s\,dv\right)
+
R_b(s).
\]
Then \(E(s)\) is holomorphic on the strip
\[
\alpha<\Re(s)<\beta,
\]
and
\[
E(\sigma+i\tau)\ll_{\alpha,\beta,f}(1+|\tau|)^2
\qquad (\alpha\le \sigma\le \beta).
\]

Consequently, the function
\[
\widetilde Z_f^{\mathrm{end}}(s)
:=
A(s)\frac{\zeta(3s-1)}{\zeta(3s)}+E(s)
\]
is meromorphic on the strip
\[
\alpha<\Re(s)<\beta.
\]
On the overlap
\[
\max\left\{\alpha,\frac23\right\}<\Re(s)<\beta,
\]
it agrees with the original absolutely convergent series
\(Z_f^{\mathrm{end}}(s)\). It therefore gives the meromorphic continuation
of \(Z_f^{\mathrm{end}}\) to the strip. The error term \(E\) is holomorphic
there and has polynomial growth on vertical lines.

\end{proposition}

\begin{proof}
Let
\[
s=\sigma+i\tau,\qquad \alpha\le \sigma\le \beta.
\]
By Proposition~\ref{prop:spec-strip},
\[
|R_b(s)|
\ll_{\alpha,\beta,f,\varepsilon}
(1+|\tau|)^2\,b^{\frac{1+\sigma}{2}+\varepsilon}.
\]
Therefore
\[
|b^{-3s}R_b(s)|
=
b^{-3\sigma}|R_b(s)|
\ll_{\alpha,\beta,f,\varepsilon}
(1+|\tau|)^2\,b^{-3\sigma+\frac{1+\sigma}{2}+\varepsilon}.
\]
The exponent simplifies to
\[
-3\sigma+\frac{1+\sigma}{2}+\varepsilon
=
-\frac{5\sigma-1}{2}+\varepsilon.
\]
Since \(\sigma\ge \alpha\), we obtain the uniform bound
\[
|b^{-3s}R_b(s)|
\ll_{\alpha,\beta,f,\varepsilon}
(1+|\tau|)^2\,b^{-\frac{5\alpha-1}{2}+\varepsilon}.
\]
Now \(\alpha>\tfrac35\), so
\[
\frac{5\alpha-1}{2}>1.
\]
Choose \(\varepsilon>0\) so small that
\[
-\frac{5\alpha-1}{2}+\varepsilon<-1.
\]
Then
\[
\sum_{b\ge1} b^{-\frac{5\alpha-1}{2}+\varepsilon}
\]
converges. Hence, for every fixed \(\tau\),
\[
\sum_{b\ge1} b^{-3s}R_b(s)
\]
converges absolutely and uniformly in \(\sigma\in[\alpha,\beta]\), with the bound
\[
\sum_{b\ge1}|b^{-3s}R_b(s)|
\ll_{\alpha,\beta,f}
(1+|\tau|)^2.
\]
Since
\[
|2^{-s}|=2^{-\sigma}\le 2^{-\alpha},
\]
the same bound holds for the full error term:
\[
E(\sigma+i\tau)\ll_{\alpha,\beta,f}(1+|\tau|)^2.
\]
This proves the strip-uniform growth estimate.

To prove holomorphy, let \(K\) be any compact subset of the open strip
\[
\alpha<\Re(s)<\beta.
\]
Then there exist \(\alpha',\beta'\) with
\[
\alpha<\alpha'<\Re(s)<\beta'<\beta
\qquad (s\in K).
\]
Repeating the same argument with \(\alpha'\) in place of \(\alpha\), we obtain a convergent
numerical majorant independent of \(s\in K\). Hence the series defining \(E(s)\) converges
uniformly on \(K\) by the Weierstrass \(M\)-test. Since each summand is holomorphic in \(s\),
it follows that \(E(s)\) is holomorphic on the open strip \(\alpha<\Re(s)<\beta\).

Finally, \(A(s)\) is holomorphic on
\[
\frac12<\Re(s)<1
\]
by Lemma~\ref{lem:intHs} and dominated convergence. Since
\[
\Re(3s)>3\alpha>\frac95>1
\]
throughout the strip, the denominator \(\zeta(3s)\) does not vanish there.
Thus
\[
\widetilde Z_f^{\mathrm{end}}(s)
:=
A(s)\frac{\zeta(3s-1)}{\zeta(3s)}+E(s)
\]
is meromorphic on the strip, with the only possible singularity occurring
at \(s=\frac23\).

On the overlap
\[
\max\left\{\alpha,\frac23\right\}<\Re(s)<\beta,
\]
the original series is absolutely convergent, and the identity
\[
\sum_{b\ge1}\frac{\varphi(b)}{b^{3s}}
=
\frac{\zeta(3s-1)}{\zeta(3s)}
\]
is valid. Hence
\[
\widetilde Z_f^{\mathrm{end}}(s)
=
Z_f^{\mathrm{end}}(s)
\]
on this overlap. This proves that
\(\widetilde Z_f^{\mathrm{end}}\) is the meromorphic continuation of
\(Z_f^{\mathrm{end}}\) to the strip.

\end{proof}

\begin{theorem}[Half-plane refinement of Theorem~\ref{thm:res}]
\label{thm:res-halfplane}
Assume \eqref{eq:ass}. Then the local meromorphic continuation from
Theorem~\ref{thm:res} extends to the half-plane
\[
\Re(s)>\frac35.
\]
It is holomorphic there except for a simple pole at
\[
s=\frac23.
\]
Moreover,
\[
\Res_{s=2/3} Z_f(s)
=
\frac{\sqrt3\,\Gamma(1/3)^3}{2^{2/3}\pi^3}
\int_0^1 |f''(v)|^{2/3}\,dv.
\]
\end{theorem}

\begin{proof}

The defining series \(Z_f(s)\) and \(Z_f^{\mathrm{end}}(s)\) converge
absolutely and are holomorphic for
\[
\Re(s)>\frac23.
\]

Fix
\[
\frac35<\alpha<\beta<1.
\]
By Proposition~\ref{prop:E-strip}, the function
\[
\widetilde Z_f^{\mathrm{end}}(s)
=
A(s)\frac{\zeta(3s-1)}{\zeta(3s)}+E(s)
\]
is meromorphic on the strip
\[
\alpha<\Re(s)<\beta
\]
and agrees with \(Z_f^{\mathrm{end}}(s)\) on the nonempty overlap
\[
\max\left\{\alpha,\frac23\right\}<\Re(s)<\beta.
\]
It is therefore the unique meromorphic continuation of
\(Z_f^{\mathrm{end}}\) to that strip.

If two such strips overlap, their continued functions agree on the part of
the overlap lying in \(\Re(s)>\frac23\), where both coincide with the
original absolutely convergent series. By the identity theorem, they agree
on the entire overlap. Letting \(\alpha\downarrow\frac35\) and
\(\beta\uparrow1\), these strip continuations therefore glue to a
meromorphic continuation of \(Z_f^{\mathrm{end}}\) to
\[
\frac35<\Re(s)<1.
\]

By Lemma~\ref{lem:replace-halfplane}, the difference
\[
Z_f(s)-Z_f^{\mathrm{end}}(s)
\]
is holomorphic on \(\Re(s)>\frac12\). Hence
\[
\widetilde Z_f(s)
:=
\widetilde Z_f^{\mathrm{end}}(s)
+
\bigl(Z_f(s)-Z_f^{\mathrm{end}}(s)\bigr)
\]
defines a meromorphic continuation of \(Z_f\) to
\[
\frac35<\Re(s)<1.
\]
On \(\frac23<\Re(s)<1\), this continuation agrees with the original
Dirichlet series. Gluing it to the original function on
\(\Re(s)>\frac23\) therefore gives a meromorphic continuation to the whole
half-plane
\[
\Re(s)>\frac35.
\]

Since \(\Re(3s)>\frac95>1\) in this half-plane, one has
\(\zeta(3s)\neq0\). The error term \(E\) and the endpoint-replacement
difference are holomorphic. Consequently, the only singularity is the
simple pole of \(\zeta(3s-1)\) at
\[
s=\frac23.
\]

It remains to compute the residue. From the decomposition above,
\[
\Res_{s=2/3}Z_f(s)
=
A(2/3)\cdot \frac{1}{3\zeta(2)}.
\]
Now
\[
A(2/3)
=
2^{-2/3}\left(\int_0^1 H_{2/3}(u)\,du\right)\left(\int_0^1 |f''(v)|^{2/3}\,dv\right),
\]
and by Lemma~\ref{lem:intHs},
\[
\int_0^1 H_{2/3}(u)\,du=\frac{\Gamma(1/3)^2}{\Gamma(2/3)}.
\]
Using \(\zeta(2)=\pi^2/6\) and \(\Gamma(1/3)\Gamma(2/3)=2\pi/\sqrt3\), this simplifies to
\[
\Res_{s=2/3} Z_f(s)
=
\frac{\sqrt3\,\Gamma(1/3)^3}{2^{2/3}\pi^3}
\int_0^1 |f''(v)|^{2/3}\,dv.
\]
\end{proof}

\begin{remark}
Theorem~\ref{thm:res-halfplane} strengthens the earlier disk-based local continuation
statement. It is the analytic input used in Appendix~\ref{app2},
Subsection~\ref{ssec:analyticcompletion}.
\end{remark}
\section{Proofs for Section~\ref{sec_6}}\label{app:section6}

This appendix contains the details omitted from the main presentation of
Section~\ref{sec_6}. The order follows the proof: tropical-zeta and bulk
bookkeeping, the exact local decomposition, the three second-layer packages,
and finally the balanced-complement estimates.

\subsection{Tropical zeta, the Farey counting function, and bulk bookkeeping}

For one parabolic arc the primitive support defects are indexed by Farey
neighbors.  If their heights are \(p\) and \(q\), the child has height
\(p+q\), and the corresponding reciprocal tropical size is
\[
 m_{p,q}=pq(p+q).
\]
Thus the boundary Dirichlet series is \eqref{eq:FGamma}, and its summatory
function is \eqref{eq:Bcount}.  The tropical-zeta calculation in
Subsection~\ref{ss:specialdomain} identifies this series with the
\(SU(3)\) Witten zeta quotient in \eqref{eq:ZL}.  In particular, the pole at
\(s=2/3\) translates, by standard Perron--Stieltjes summation, into the
leading bulk scale
\begin{equation}\label{eq:BExpansion}
 B(X)=\kappa_{2/3}X^{2/3}+\kappa_{1/2}X^{1/2}+O(X^\theta)
\end{equation}
for some \(\theta<1/2\), after the continuation input available in the main
paper.  The precise constants are not needed in the final coefficient, because
the corresponding bulk packages cancel.  What is needed is that all naive
head--tail Stieltjes transforms are governed by the same Farey counting data.

The bulk bookkeeping is as follows.  The raw decomposition is
\[
 E_L=T_1+T_2-T_0+H_0+H_{12}+1.
\]
The quadratic tail term is exactly the triangular-density replacement for
\(T_0\):
\[
 T_2(n)=T^{\rm dens}_0(n).
\]
Therefore
\[
 T_2(n)-T_0(n)=R_{T_0}(n).
\]
This removes the bulk density contribution before averaging.  The linear term
\(T_1\) does not stand alone either: it must be paired with the mean part of
\(H_0\), giving the weighted tail in \eqref{eq:weightedTailId} below.  The
two-dimensional bulk average of \(H_{12}\) vanishes; only edge regimes and
moving-prefix defects survive.  In short:
\[
\begin{array}{c|c|c}
\text{package} & \text{bulk source} & \text{fate}\\ \hline
T_2-T_0 & \text{triangular density of the threshold grid} & R_{T_0}\\
H_0+T_1 & \text{first-power head/tail mass} & \text{weighted tail}\\
H_{12} & \text{complete-period quadratic phase} & \text{edge plus prefix}\\
1 & \text{global Euler term} & O(1)\text{ in the average}
\end{array}
\]
This is the formal reason the pole at \(s=2/3\) is present in the calculation
but absent from the final averaged coefficient.

\subsection{The column formula}

In one corner of the square \([-n,n]^2\), put
\[
 i=n-X,
 \qquad
 j=n-Y.
\]
The removed corner is described by
\begin{equation}\label{eq:cornerIneq}
 \sqrt i+\sqrt j<\sqrt n,
 \qquad 0\le i,j\le n.
\end{equation}
For \(i=0\), inequality \eqref{eq:cornerIneq} gives \(0\le j<n\), hence
exactly \(n\) removed points.  For \(1\le i<n\), it is equivalent to
\[
 j<n+i-2\sqrt{ni}.
\]
Thus the number of removed \(j\)'s in this column is
\[
 \left\lceil n+i-2\sqrt{ni}\right\rceil=n+i-\floor{2\sqrt{ni}}.
\]
The column \(i=n\) contributes no removed point.  Therefore the exact removed
count in one corner is
\begin{equation}\label{eq:cornerCount}
 C(n)=n+\sum_{i=1}^{n-1}\left(n+i-\floor{2\sqrt{ni}}\right).
\end{equation}
The boundary strip \(i=0\) is essential; omitting it would lose an \(O(n)\)
contribution.

Using
\[
 \floor{x}=x-\psi(x)-\frac12,
 \qquad
 \psi(x)=\{x\}-\frac12,
\]
we get
\begin{align*}
 C(n)
 &=n+\sum_{i=1}^{n-1}\left(n+i-2\sqrt{ni}+\psi(2\sqrt{ni})+\frac12\right)\\
 &=\frac32n^2-\frac12
 -2\sqrt n\sum_{i=1}^{n-1}\sqrt i
 +S_{\mathrm{col}}(n),
\end{align*}
where
\[
 S_{\mathrm{col}}(n)=\sum_{i=1}^{n-1}\psi(2\sqrt{ni}).
\]
Euler--Maclaurin gives
\begin{equation}\label{eq:EMsqrt}
 \sum_{i=1}^{n-1}\sqrt i
 =\frac23 n^{3/2}-\frac12 n^{1/2}+\zeta\!\left(-\frac12\right)+\frac1{24}n^{-1/2}+O(n^{-3/2}).
\end{equation}
Substitution into \eqref{eq:cornerCount} gives
\[
 C(n)=\frac16n^2+n-2\zeta\!\left(-\frac12\right)\sqrt n+S_{\mathrm{col}}(n)-\frac7{12}+O(n^{-1}).
\]
Since
\[
 N_L(n)=(2n+1)^2-4C(n),
\]
we obtain
\[
 N_L(n)=\frac{10}{3}n^2-4S_{\mathrm{col}}(n)+8\zeta\!\left(-\frac12\right)\sqrt n+\frac{10}{3}+O(n^{-1}).
\]
This proves \eqref{eq:columnFormula}.  Averaging gives
\[
 A_L(N)=-4N^{-1}\sum_{n\le N}S_{\mathrm{col}}(n)+\frac{16}{3}\zeta\!\left(-\frac12\right)\sqrt N+O(1).
\]
Combining this with Theorem~\ref{thm:mainL} proves \eqref{eq:triSaw}.  Since
the square sum equals twice the triangular sum plus \(O(N)\) diagonal and
boundary-line terms, \eqref{eq:squareSaw} follows.

\begin{remark}
The cancellation of the linear term is transparent in this computation.  The
ambient square has \((2n+1)^2=4n^2+4n+1\) lattice points.  The four corrected
corner counts contain the same \(4n\) term.  Hence no lattice-perimeter term
survives.  This agrees with the geometric fact that \(\partial L\) has no
positive-length rational-slope side.  If the domain had genuine rational
polygonal sides, this cancellation would be replaced by the usual Ehrhart
lattice-perimeter contribution.
\end{remark}

\subsection{The local head--tail formula and the global Euler term}

We recall the elementary local calculation because it is the source of all
five terms in \eqref{eq:ledger}.  Let \(\lambda_1,\lambda_2\) be Farey
neighbors.  Since \(\det(\lambda_1,\lambda_2)=1\), the corresponding cell is
unimodular in the primitive edge coordinates.  Its reciprocal size is
\(m=m_{p,q}=pq(p+q)\), and at time \(n\) we put
\[
 u=\frac{n}{m}.
\]
Let
\[
 x=\{n\alpha_{\lambda_1}\},
 \qquad
 y=\{n\alpha_{\lambda_2}\}.
\]
The local cell is the half-open shifted triangle
\[
 \Delta_u=\{(X,Y)\in\R^2_{\ge0}:X+Y<u\}
\]
intersected with the shifted lattice \((x,y)+\Z^2\).  The following count is
the local generalized Pick formula.

\begin{lemma}[Shifted half-open triangle count]\label{lem:triCount}
Let \(u>0\) and \(0\le x,y<1\).  Define
\[
 N(u;x,y)=\#\left(((x,y)+\Z^2_{\ge0})\cap \Delta_u\right).
\]
Then
\begin{equation}\label{eq:NuK}
 N(u;x,y)=\frac{K(K+1)}2,
 \qquad
 K=\max\{0,\ceil{u-x-y}\}.
\end{equation}
If \(u\ge1\) and
\[
 \phi=u-K,
\]
then
\begin{equation}\label{eq:headLocal}
 \frac12u^2+\frac12u-N(u;x,y)
 =u\phi+\frac12\phi-\frac12\phi^2.
\end{equation}
If \(0<u<1\), then
\begin{equation}\label{eq:tailLocal}
 N(u;x,y)=\ind_{x+y<u},
\end{equation}
and hence
\begin{equation}\label{eq:tailLocalFormula}
 \frac12u^2+\frac12u-N(u;x,y)=\frac12u^2+\frac12u-\ind_{x+y<u}.
\end{equation}
\end{lemma}

\begin{proof}
A point of \(((x,y)+\Z^2_{\ge0})\cap\Delta_u\) has the form \((x+i,y+j)\),
where \(i,j\in\Z_{\ge0}\), and the condition is
\[
 i+j<u-x-y.
\]
The possible values of \(i+j\) are \(0,1,\ldots,K-1\), where
\(K=\max\{0,\ceil{u-x-y}\}\).  For fixed \(s=i+j\), there are \(s+1\) choices
of \((i,j)\).  Thus
\[
 N(u;x,y)=\sum_{s=0}^{K-1}(s+1)=\frac{K(K+1)}2.
\]
If \(u\ge1\), put \(\phi=u-K\).  Then
\[
 N(u;x,y)=\frac{(u-\phi)(u-\phi+1)}2,
\]
and a direct expansion gives \eqref{eq:headLocal}.  If \(0<u<1\), then
\(\Delta_u\) is contained in the unit square, and the only possible shifted
lattice point is \((x,y)\).  This gives \eqref{eq:tailLocal} and
\eqref{eq:tailLocalFormula}.
\end{proof}

By the phase identity \eqref{eq:phaseID},
\[
 \{n\alpha_{\lambda_1+\lambda_2}\}=\{x+y-u\}.
\]
Therefore
\begin{equation}\label{eq:phiXYU}
 \Phi_{p,q}(n)=x+y-\{x+y-u\}=u-\ceil{u-x-y}
\end{equation}
whenever \(u\ge1\).  In the notation of Lemma~\ref{lem:triCount}, this says
\(\phi=\Phi_{p,q}(n)\).

Expanding \eqref{eq:headLocal} with \(u=n/m_{p,q}\) gives the head
contribution
\[
 \frac{n}{m_{p,q}}\Phi_{p,q}(n)
 +\frac12\Phi_{p,q}(n)
 -\frac12\Phi_{p,q}(n)^2.
\]
When \(0<u<1\), formula \eqref{eq:tailLocalFormula} gives the tail
contribution
\[
 \frac12\left(\frac{n}{m_{p,q}}\right)
 +\frac12\left(\frac{n}{m_{p,q}}\right)^2
 -\ind_{\{n\alpha_{\lambda_1}\}+\{n\alpha_{\lambda_2}\}<n/m_{p,q}}.
\]
Summing over the four parabolic arcs gives \(H_0,H_{12},T_1,T_2,T_0\), as in
Proposition~\ref{prop:ledger}.

It remains to identify the global constant.  The half-open convention assigns
every internal edge and every internal vertex of the Farey dissection to
exactly one adjacent cell.  Hence no local Euler term accumulates along
internal seams.  The only remaining Euler contribution is the Euler
characteristic of the whole convex disk \(nL\), namely \(1\).  Equivalently,
it is the constant term already present in the ambient square count
\((2n+1)^2\).  This proves the exact global decomposition \eqref{eq:ledger}.

\subsection{The phase identity}

Let
\[
 \lambda_1=(a,b),\qquad \lambda_2=(c,d)
\]
be Farey neighbors, with
\[
 a+b=p,
 \qquad
 c+d=q,
 \qquad
 ad-bc=1.
\]
Then \(\lambda_{12}=\lambda_1+\lambda_2=(a+c,b+d)\) has height \(p+q\).
Using \eqref{eq:alpha},
\begin{align*}
 \alpha_{\lambda_1}+\alpha_{\lambda_2}-\alpha_{\lambda_{12}}
 &\equiv
 -\frac{ab}{p}-\frac{cd}{q}+\frac{(a+c)(b+d)}{p+q}
 \pmod1.
\end{align*}
Putting the right-hand side over the common denominator \(pq(p+q)\), the
numerator is
\[
 -abq(p+q)-cdp(p+q)+pq(a+c)(b+d).
\]
Using \(b=p-a\) and \(d=q-c\), this simplifies to
\[
 (aq-pc)^2.
\]
But
\[
 aq-pc=a(c+d)-(a+b)c=ad-bc=1.
\]
Thus
\[
 \alpha_{\lambda_1}+\alpha_{\lambda_2}-\alpha_{\lambda_{12}}
 \equiv \frac{1}{pq(p+q)}\pmod1.
\]
Equivalently,
\[
 \alpha_{\lambda_{12}}
 \equiv
 \alpha_{\lambda_1}+\alpha_{\lambda_2}-\frac1{m_{p,q}}
 \pmod1.
\]

\begin{remark}
If the rational representative \(-\frac{ab}{a+b}\) is chosen for
\(\alpha_{(a,b)}\), the formula above holds as an equality, without reduction
modulo \(1\).
\end{remark}

\subsection{The weighted-tail proof for \(H_0+T_1\)}

For a primitive direction \(\lambda\), define the signed side length
\begin{equation}\label{eq:Slam}
 S_\lambda(n)=4n\sum_{\substack{\tau=(\lambda_1,\lambda_2)\\m_\tau\le n}}
 \frac{\ind_{\lambda=\lambda_1}+\ind_{\lambda=\lambda_2}-\ind_{\lambda=\lambda_1+\lambda_2}}{m_\tau}.
\end{equation}
Collecting coefficients of the phases in \eqref{eq:H0} gives
\begin{equation}\label{eq:H0Slam}
 H_0(n)=\sum_\lambda S_\lambda(n)\{n\alpha_\lambda\}.
\end{equation}
Since \(\gcd(ab,a+b)=1\) for primitive \(\lambda=(a,b)\), the sequence
\(\{n\alpha_\lambda\}\) runs through the full grid of size \(|\lambda|\).
Hence
\begin{equation}\label{eq:thetaSplit}
 \{n\alpha_\lambda\}=\theta_\lambda(n)+\frac12-\frac{1}{2|\lambda|},
\end{equation}
where \(\theta_\lambda\) has period \(|\lambda|\) and mean zero.  This splits
\(H_0=H_0^{\rm mean}+\Omega_0\).

For the oscillatory part, expand \eqref{eq:Slam} back into triangle
contributions.  A fixed triangle with heights \(p,q,p+q\) contributes sums of
the form
\[
 \frac4m\sum_{m\le n\le N}n\theta_\lambda(n),
 \qquad
 |\lambda|\in\{p,q,p+q\}.
\]
The partial sums of \(\theta_\lambda\) are \(O(|\lambda|)\); Abel summation
gives \(O(N|\lambda|/m)\).  Since \(m=pq(p+q)\), this is \(O(N/(pq))\).
Summing over \(pq(p+q)\le N\) gives
\[
 \sum_{n\le N}\Omega_0(n)=O(N\log^2N)=o(N^{3/2}).
\]

For the mean part define
\begin{equation}\label{eq:Wpq}
 W_{p,q}=\frac{1}{pq(p+q)}\left(\frac1p+\frac1q-\frac{1}{p+q}\right).
\end{equation}
The Tornheim identities
\begin{equation}\label{eq:Tornheim}
 \sum_{\gcd(p,q)=1}\frac{1}{pq(p+q)}=2,
 \qquad
 \sum_{\gcd(p,q)=1}W_{p,q}=2
\end{equation}
follow by M\"obius inversion from the classical evaluations
\[
 \sum_{a,b\ge1}\frac1{ab(a+b)}=2\zeta(3),
\]
and
\[
 \sum_{a,b\ge1}\frac1{ab(a+b)}\left(\frac1a+\frac1b-\frac1{a+b}\right)=2\zeta(4).
\]
The mean contribution of a head triangle is
\[
 2nf_{p,q}\left(1-\frac1p-\frac1q+\frac1{p+q}\right).
\]
Adding
\[
 T_1(n)=2n\sum_{m_{p,q}>n}f_{p,q}
\]
and using \eqref{eq:Tornheim} gives the exact identity
\begin{equation}\label{eq:weightedTailId}
 H_0^{\rm mean}(n)+T_1(n)
 =2n\sum_{\substack{\gcd(p,q)=1\\m_{p,q}>n}}W_{p,q}.
\end{equation}

We now extract the edge coefficient with the multiplicity and the error terms
explicit.  Summing first over \(n\), one pair contributes
\[
 2W_{p,q}\sum_{n\le \min(N,m_{p,q}-1)} n
 =W_{p,q}\min(N,m_{p,q})^2+O(W_{p,q}N).
\]
The total contribution of the error term is \(O(N\log^2N)=o(N^{3/2})\).
Fix \(p\).  As \(q\to\infty\),
\begin{equation}\label{eq:Wasy}
 W_{p,q}=\frac{1}{p^2q^2}+O_p(q^{-3}),
 \qquad
 m_{p,q}=pq^2+p^2q.
\end{equation}
The coprime counting estimate
\begin{equation}\label{eq:coprimeDensity}
 \#\{q\le X:\gcd(q,p)=1\}=\frac{\varphi(p)}{p}X+O(\tau(p))
\end{equation}
is used uniformly on dyadic intervals.  Put \(Q_N=(N/p)^{1/2}\).  For
\(q\le Q_N\), \(m_{p,q}\le N\) to leading order, and \eqref{eq:Wasy} gives
\[
 W_{p,q}m_{p,q}^2=q^2+O_p(q).
\]
Thus
\[
 \sum_{\substack{q\le Q_N\\\gcd(p,q)=1}} W_{p,q}m_{p,q}^2
 \sim
 \frac{\varphi(p)}{p}\frac{Q_N^3}{3}
 =\frac{\varphi(p)}{3p^{5/2}}N^{3/2}.
\]
For \(q>Q_N\), the contribution is
\[
 W_{p,q}N^2=\frac{N^2}{p^2q^2}+O_p(N^2q^{-3}),
\]
and hence
\[
 \sum_{\substack{q>Q_N\\\gcd(p,q)=1}}W_{p,q}N^2
 \sim
 \frac{\varphi(p)}{p}\frac{N^2}{p^2Q_N}
 =\frac{\varphi(p)}{p^{5/2}}N^{3/2}.
\]
One edge therefore contributes
\[
 \frac43\frac{\varphi(p)}{p^{5/2}}.
\]
The opposite edge contributes the same amount.  Summing over \(p\) gives
\[
 N^{-3/2}\sum_{n\le N}(H_0(n)+T_1(n))
 \longrightarrow
 \frac83\sum_{p\ge1}\frac{\varphi(p)}{p^{5/2}}
 =\frac83\frac{\zeta(3/2)}{\zeta(5/2)}.
\]
For the complement after removing \(\min(p,q)\le P\), use
\(W_{p,q}\ll p^{-2}q^{-2}\) in the region \(P<p\le q\).  The part with
\(m_{p,q}\le N\) is bounded by
\[
 \sum_{p>P}\sum_{q\le (N/p)^{1/2}}\frac{(pq^2)^2}{p^2q^2}
 \ll
 \sum_{p>P}\sum_{q\le (N/p)^{1/2}}q^2
 \ll
 N^{3/2}\sum_{p>P}p^{-3/2}
 \ll
 N^{3/2}P^{-1/2}.
\]
The part with \(m_{p,q}>N\) is bounded by
\[
 N^2\sum_{p>P}\sum_{q>(N/p)^{1/2}}\frac{1}{p^2q^2}
 \ll
 N^{3/2}\sum_{p>P}p^{-3/2}
 \ll
 N^{3/2}P^{-1/2}.
\]
This proves Proposition~\ref{prop:H0T1}.

\subsection{The complete-period part of \(H_{12}\)}

Let
\[
 F_{p,q}(n)=\frac12\Phi_{p,q}(n)-\frac12\Phi_{p,q}(n)^2.
\]
The complete-period mean is
\begin{equation}\label{eq:Qpq}
 Q(p,q)=\langle F_{p,q}\rangle_{\rm per}
 =-
 \frac{(p^2+pq+q^2)^2}{12p^2q^2(p+q)^2}.
\end{equation}
Indeed, over a complete period the three phases have independent uniform
distributions on the grids of sizes \(p,q,p+q\).  If \(X_d\) is uniform on
\(\{0,1/d,\ldots,(d-1)/d\}\), then
\[
 \mathbb E X_d=\frac12-\frac{1}{2d},
 \qquad
 \operatorname{Var}(X_d)=\frac{d^2-1}{12d^2}.
\]
Substituting into
\[
 \mathbb E\left(\frac12\Phi-\frac12\Phi^2\right)
 =\frac12\mathbb E\Phi-\frac12(\mathbb E\Phi)^2-\frac12\operatorname{Var}(\Phi)
\]
gives \eqref{eq:Qpq}.  As \(q\to\infty\) with \(p\) fixed,
\[
 Q(p,q)\to -\frac{1}{12p^2}.
\]
Therefore the same edge extraction as above gives
\begin{equation}\label{eq:H12perLimit}
 N^{-3/2}\sum_{n\le N}4\sum_{m_{p,q}\le n}Q(p,q)
 \longrightarrow
 -\frac49\sum_{p\ge1}\frac{\varphi(p)}{p^{7/2}}
 =-\frac49\frac{\zeta(5/2)}{\zeta(7/2)}.
\end{equation}
This is the complete-period part of Proposition~\ref{prop:H12corr}.

\subsection{The moving-prefix correction for \(H_{12}\)}

The complete-period replacement is not the whole story.  After a triangle
becomes active at \(n=m_{p,q}\), all full periods cancel, but the last prefix
of the period remains.  This prefix has a finite-strip contribution.

Using \eqref{eq:phaseID}, put
\[
 u=\frac{n}{m_{p,q}},
 \qquad
 x=\{n\alpha_{\lambda_1}\},
 \qquad
 y=\{n\alpha_{\lambda_2}\}.
\]
Then
\[
 \Phi_{p,q}(n)=x+y-\{x+y-u\}=u+\floor{x+y-u}.
\]
For fixed \(p\) and \(q\to\infty\), \(x\) remains on the \(p\)-grid and
\(y\) becomes continuous.  Put
\[
 A_p(u)=\sum_{i=0}^{p-1}\left(u-\frac{i}{p}\right)_+.
\]
A direct integration over the continuous \(y\)-variable gives the strip mean
\begin{equation}\label{eq:fpstrip}
 f_p(u)=\frac{u^2}{2}+\frac{u}{2p}-\frac{A_p(u)}p.
\end{equation}
Since \(Q(p,q)\to -1/(12p^2)\), the centered strip density is
\[
 g_p(u)=f_p(u)+\frac{1}{12p^2},
 \qquad
 \int_0^1g_p(u)\dd u=0.
\]
Let
\begin{equation}\label{eq:Gpdef}
 G_p(v)=\int_0^v g_p(u)\dd u
 =\frac{v^3}{6}+\frac{v^2}{4p}
 -\frac{1}{2p}\sum_{i=0}^{p-1}\left(v-\frac{i}{p}\right)_+^2
 +\frac{v}{12p^2}.
\end{equation}
For one fixed strip, the prefix contribution is
\begin{equation}\label{eq:prefixPair}
 m_{p,q}G_p\left(\left\{\frac{N}{m_{p,q}}\right\}\right)+o(m_{p,q}),
\end{equation}
uniformly over the fixed-strip range.  With \(q=\sqrt{N/p}\,t\), one has
\(m_{p,q}=Nt^2+o(N)\) and \(\{N/m_{p,q}\}=\{t^{-2}\}+o(1)\).  Therefore
\begin{equation}\label{eq:IpIntegral}
 I_p=\int_0^1 t^2G_p(\{t^{-2}\})\dd t
\end{equation}
and, for fixed \(P\),
\begin{equation}\label{eq:FSH12}
 \lim_{N\to\infty}N^{-3/2}\sum_{n\le N}R_{12}^{\min(p,q)\le P}(n)
 =8\sum_{p\le P}\frac{\varphi(p)}{p^{3/2}}I_p.
\end{equation}
The overlap \(p,q\le P\) contains only finitely many pairs and contributes
\(o(N^{3/2})\) to the raw time sum.

The integral \eqref{eq:IpIntegral} may be evaluated explicitly.  Decomposing
the intervals on which \(\{t^{-2}\}\) equals \(v\) gives an expression through Hurwitz zeta function
\[
 I_p=\frac12\int_0^1G_p(v)\zeta\!\left(\frac52,1+v\right)\dd v.
\]
Integration by parts, together with
\[
 \partial_v\zeta(s,1+v)=-s\zeta(s+1,1+v),
\]
and the Hurwitz shift identity gives exactly \eqref{eq:Ipdef}.  Equivalently,
\[
 I_p=\frac{K_p}{p}-\frac{1}{3p}+\frac{1}{18p^2}+\frac{4\zeta(-1/2)}{3p^{3/2}}.
\]

\subsection{A product-discrepancy lemma for the balanced parts}

The balanced complements in the \(H_{12}\) and \(T_0\) analyses are controlled
by the same elementary two-dimensional Euler--Maclaurin estimate.  We state it
explicitly in the form used below.

Let
\[
 \mu_p=\frac1p\sum_{i=0}^{p-1}\delta_{i/p},
 \qquad
 \eta_p=\mu_p-dx.
\]
Define the half-open discrepancy primitive
\[
 D_p(x)=\eta_p([0,x)),
 \qquad 0\le x\le1.
\]
Then
\begin{equation}\label{eq:DpBound}
 |D_p(x)|\le \frac1p.
\end{equation}

\begin{lemma}[Product-discrepancy estimate]\label{lem:prodDisc}
Let \(K\) be a bounded polygonal kernel on \([0,1]^2\), piecewise \(C^1\),
whose singular set is contained in finitely many line segments of slope
\(-1\), and whose mixed distributional derivative \(\partial_x\partial_yK\)
is a finite signed measure.  Then
\begin{equation}\label{eq:prodDisc}
 \left|\iint_{[0,1]^2}K(x,y)\,d\eta_p(x)d\eta_q(y)\right|
 \le
 \frac{C(K)}{pq},
\end{equation}
where \(C(K)\) is controlled by the total variation of
\(\partial_x\partial_yK\) and by the total variations of the boundary traces
of \(K\).  Moreover, for the triangular kernels used below, the constants are
uniform in the threshold parameter.
\end{lemma}

\begin{proof}
We give the proof with the half-open convention fixed above.  In one variable,
Stieltjes integration by parts gives
\[
 \int_0^1 f(x)\,d\eta_p(x)
 =-\int_{[0,1]}D_p(x)\,df(x)+\mathcal B_p(f),
\]
where \(\mathcal B_p(f)\) is a finite sum of endpoint terms, each with
coefficient bounded by \(O(p^{-1})\).  Applying this first in \(x\) and then
in \(y\), we obtain
\begin{align*}
 \iint K\,d\eta_p\,d\eta_q
 &=\iint D_p(x)D_q(y)\,d(\partial_x\partial_yK)(x,y)\\
 &\quad +\text{boundary integrals on }x=0,1\text{ and }y=0,1.
\end{align*}
The interior term is bounded by
\[
 \|D_p\|_\infty\|D_q\|_\infty\|\partial_x\partial_yK\|_{\TV}
 \ll \frac1{pq}\|\partial_x\partial_yK\|_{\TV}.
\]
The boundary terms are estimated in the same way, using one discrepancy
primitive and the boundary variation of the corresponding trace; the endpoint
coefficients contribute at most the same order.  This proves
\eqref{eq:prodDisc}.  The kernels below are integrals of indicators of
diagonal half-planes or piecewise quadratic functions of \(x+y-u\); their
mixed distributional derivatives are supported on finitely many diagonal line
segments with uniformly bounded mass.  Hence the constants are uniform.
\end{proof}

\subsection{The balanced complement for \(H_{12}\)}

Let
\[
 \Phi_u(x,y)=u+\floor{x+y-u},
 \qquad
 \mathcal F_u(x,y)=\frac12\Phi_u(x,y)-\frac12\Phi_u(x,y)^2.
\]
The continuous average of \(\mathcal F_u\) over \([0,1]^2\) is zero.  With the
normalized grid measures \(\mu_p,\mu_q\), the full grid average decomposes as
\[
 \iint \mathcal F_u\,d\mu_p d\mu_q=f_p(u)+f_q(u)+e^{12}_{p,q}(u),
\]
where the strip averages \(f_p,f_q\) are the one-dimensional terms and
\begin{equation}\label{eq:e12}
 e^{12}_{p,q}(u)=\iint \mathcal F_u(x,y)\,d\eta_p(x)d\eta_q(y)
\end{equation}
is the balanced part.  After subtracting its full-period mean, define
\begin{equation}\label{eq:E12}
 \mathcal E^{12}_{p,q}(v)=
 \int_0^v\left(e^{12}_{p,q}(u)-\int_0^1e^{12}_{p,q}(w)\dd w\right)\dd u.
\end{equation}
The integrated kernel
\[
 K^{12}_v(x,y)=\int_0^v\left(\mathcal F_u(x,y)-\int_0^1\mathcal F_w(x,y)\dd w\right)\dd u
\]
is piecewise polynomial.  Its singularities lie on \(x+y=v\) and \(x+y=v+1\),
and its mixed distributional derivative has uniformly bounded total
variation.  Lemma~\ref{lem:prodDisc} therefore gives
\begin{equation}\label{eq:E12bound}
 |\mathcal E^{12}_{p,q}(v)|\ll \frac1{pq},
 \qquad 0\le v\le1.
\end{equation}

For one Farey pair, complete periods of \(F_{p,q}-Q(p,q)\) cancel exactly.
Hence the balanced part of the incomplete-period contribution is a single
prefix, bounded by
\[
 m_{p,q}\,|\mathcal E^{12}_{p,q}(\{N/m_{p,q}\})|+O(p+q)
 \ll p+q.
\]
The error term \(O(p+q)\) is the ordinary Riemann-sum error over the \(p+q\)
slow blocks of one period.  Summing over all head pairs gives
\[
 \sum_{pq(p+q)\le N}(p+q)
 \ll N\log N=o(N^{3/2}).
\]
The finite-strip tail is smaller still.  From \eqref{eq:Ipdef},
\[
 I_p\sim \frac{1}{1440p^4},
\]
and therefore
\[
 \sum_{p>P}\frac{\varphi(p)}{p^{3/2}}I_p\ll P^{-7/2}.
\]
Thus the finite-strip formula \eqref{eq:FSH12} exhausts the \(H_{12}\) prefix
correction.  Proposition~\ref{prop:H12corr} follows.

\subsection{The threshold residual \(R_{T_0}\)}

We prove Proposition~\ref{prop:RT0}.  For \(0\le u\le1\), define the grid
count
\begin{equation}\label{eq:IpqCount}
 I_{p,q}(u)=\#\left\{0\le i<p,
 0\le j<q:
 \frac{i}{p}+\frac{j}{q}<u\right\}.
\end{equation}
Over a complete residue system modulo \(pq\), the phase pair
\[
 n\mapsto (\{n\alpha_{\lambda_1}\},\{n\alpha_{\lambda_2}\})
\]
runs exactly once through the \(p\times q\) grid.  This follows because the
phase multipliers are invertible modulo \(p\) and \(q\), and the Chinese
remainder theorem combines the two congruences.

The product-discrepancy decomposition is exact.  Put
\[
 A_p(u)=\sum_{i=0}^{p-1}\left(u-\frac{i}{p}\right)_+,
 \qquad
 \Delta_p(u)=A_p(u)-\frac12pu^2.
\]
Define
\begin{equation}\label{eq:EpqT0}
 E_{p,q}(u)=I_{p,q}(u)-qA_p(u)-pA_q(u)+\frac12pqu^2.
\end{equation}
Then
\begin{equation}\label{eq:T0decomp}
 \frac12u^2-\frac{I_{p,q}(u)}{pq}
 =-\frac{\Delta_p(u)}p-\frac{\Delta_q(u)}q-\frac{E_{p,q}(u)}{pq}.
\end{equation}
The first two terms give the finite strips; the last term is genuinely
balanced.

For fixed \(p\) and \(q\to\infty\),
\[
 I_{p,q}(u)=qA_p(u)+O(p)
\]
uniformly in \(u\).  Hence the strip defect is \(q\Delta_p(u)+O(p)\).  The
coefficient is
\begin{equation}\label{eq:KpIntegral}
 K_p=\frac13\int_0^1\Delta_p(u)u^{-3/2}\dd u
 =\frac13\left(\frac{8p}{3}-1-\frac4{\sqrt p}\sum_{i=0}^{p-1}\sqrt i\right).
\end{equation}
The fixed-strip extraction gives, for every fixed \(P\),
\begin{equation}\label{eq:FST0}
 \lim_{N\to\infty}N^{-3/2}\sum_{n\le N}R^{\rm strip,\min(p,q)\le P}_{T_0}(n)
 =-8\sum_{p\le P}\frac{\varphi(p)}{p^{5/2}}K_p.
\end{equation}
The factor \(4\) comes from the four arcs and the factor \(2\) from the two
edge orientations.  The overlap \(p,q\le P\) has total lifetime
\[
 \sum_{p,q\le P}pq(p+q)=O(P^5),
\]
so it contributes \(O(P^5/N^{3/2})\) after normalization.  Since
\(K_p=1/3+O(p^{-1/2})\), the omitted strip tail is
\[
 \sum_{p>P}\frac{\varphi(p)}{p^{5/2}}K_p\ll P^{-1/2}.
\]

It remains to control the balanced term.  Let
\[
 \nu_p=\sum_{i=0}^{p-1}\delta_{i/p}-pdx,
 \qquad
 \nu_q=\sum_{j=0}^{q-1}\delta_{j/q}-qdy.
\]
Then
\[
 E_{p,q}(u)=\iint_{x+y<u}d\nu_p(x)d\nu_q(y),
\]
and its primitive is
\begin{equation}\label{eq:T0prim}
 \mathcal E_{p,q}(v)=\int_0^vE_{p,q}(u)\dd u
 =\iint (v-x-y)_+\,d\nu_p(x)d\nu_q(y).
\end{equation}
Let
\[
 B_p(x)=\nu_p([0,x)),
 \qquad
 B_q(y)=\nu_q([0,y)).
\]
These functions are uniformly bounded.  Applying the same integration-by-parts
argument as in Lemma~\ref{lem:prodDisc}, now with unnormalized measures, gives
\begin{equation}\label{eq:T0primBound}
 |\mathcal E_{p,q}(v)|\ll1,
 \qquad 0\le v\le1.
\end{equation}
For small \(v\), after separating the origin atom, the support consists of two
axis strips and a small triangle, giving the sharper bound
\begin{equation}\label{eq:smallv}
 |\mathcal E^\circ_{p,q}(v)|\ll (p+q)v^2+pqv^3.
\end{equation}

We now spell out the moving-threshold replacement.  Put
\[
 B=pq,
 \qquad
 r=p+q,
 \qquad
 m=Br,
 \qquad
 M\le m,
 \qquad
 V=M/m.
\]
Assume first \(B\le M\).  Write \(n=kB+a\), \(0\le a<B\).  In a complete
\(B\)-block the phase points run once through the \(p\times q\) grid.  Freezing
the threshold at \(k/r\) gives the block value \(E_{p,q}(k/r)\).  During the
block, the threshold moves through an interval of length \(1/r\).  The
difference between the moving and frozen block is the product discrepancy of
the diagonal strip
\[
 \frac{k}{r}\le x+y<\frac{k}{r}+\frac1r.
\]
After the two strip defects have been subtracted, this strip discrepancy is
estimated by applying \eqref{eq:T0primBound} to the interval
\([k/r,k/r+1/r]\).  Both discrepancy primitives are bounded, and the diagonal
strip has width \(1/r\) but is multiplied by \(r\) when the block is unfolded
from threshold to time.  Hence the moving-freezing error of one block is
\(O(1)\).

Consequently, summing complete blocks and then applying summation by parts to
the frozen sum gives
\begin{equation}\label{eq:movingReplacement}
 \sum_{n\le M}^{\rm mix}1
 =r\mathcal E_{p,q}(V)+O\left(1+\frac{M}{pq}\right),
\end{equation}
where the left-hand side denotes the time-summed mixed contribution of the
pair after subtracting the two strip defects.  In the full-lifetime range
\(m\le N\), \eqref{eq:movingReplacement} gives \(O(p+q)\) per pair, and
\[
 \sum_{pq(p+q)\le N}(p+q)
 \ll N\log N.
\]
In the partial many-block range \(m>N\) and \(pq\le N\), it gives
\[
 \sum_{pq\le N}\left(1+\frac{N}{pq}\right)
 \ll N\log^2N.
\]
Both are \(o(N^{3/2})\).

Finally consider the short-block range \(pq>N\).  For \(n\le N\),
\[
 u_n=\frac{n}{pq(p+q)}<\frac1{p+q}<\min\left\{\frac1p,\frac1q\right\}.
\]
Thus the threshold triangle can contain only the grid point \((0,0)\).  The
indicator in \eqref{eq:T0} can then be nonzero only if \(p\mid n\) and
\(q\mid n\), equivalently \(pq\mid n\).  Since \(pq>N\ge n\), this is
impossible.  Thus the indicator vanishes.  Moreover \(A_p(u_n)=A_q(u_n)=u_n\),
and the balanced short-block summand is
\[
 u_n\left(\frac1p+\frac1q\right)-\frac12u_n^2
 =\frac{n}{p^2q^2}-\frac12\frac{n^2}{p^2q^2(p+q)^2}.
\]
Therefore
\[
 \sum_{pq>N}\sum_{n\le N}\left|\frac{n}{p^2q^2}\right|
 \ll
 N^2\sum_{pq>N}\frac1{p^2q^2}
 \ll N\log N=o(N^{3/2}).
\]
This completes the proof of the balanced complement and hence
Proposition~\ref{prop:RT0}.

\subsection{Final algebra}

Combining Propositions~\ref{prop:H0T1}, \ref{prop:H12corr}, and
\ref{prop:RT0}, the coefficient of \(N^{3/2}\) in the time-summed error is
\[
 C=\frac83\frac{\zeta(3/2)}{\zeta(5/2)}
 -\frac49\frac{\zeta(5/2)}{\zeta(7/2)}
 +8\sum_{p\ge1}\frac{\varphi(p)}{p^{3/2}}I_p
 -8\sum_{p\ge1}\frac{\varphi(p)}{p^{5/2}}K_p.
\]
Using \eqref{eq:IpKp}, the \(K_p\)-series cancels and
\begin{align*}
 C&=\frac83\sum_p\frac{\varphi(p)}{p^{5/2}}
 -\frac49\sum_p\frac{\varphi(p)}{p^{7/2}}
 -\frac83\sum_p\frac{\varphi(p)}{p^{5/2}}
 +\frac49\sum_p\frac{\varphi(p)}{p^{7/2}}
 +\frac{32}{3}\zeta\!\left(-\frac12\right)\sum_p\frac{\varphi(p)}{p^3}\\
 &=\frac{32}{3}\zeta\!\left(-\frac12\right)\frac{\zeta(2)}{\zeta(3)}.
\end{align*}
This proves Theorem~\ref{thm:mainL} and Corollary~\ref{cor:saw}.

\subsection{Bibliographical perspective}

The calculation in this section lies at the meeting point of several classical
traditions, but its combination seems specific to the tropical setting.  The
Dirichlet series \(m_{p,q}^{-s}=p^{-s}q^{-s}(p+q)^{-s}\) is the diagonal
Mordell--Tornheim series, equivalently Witten's \(\mathrm{SU}(3)\) zeta series.
It belongs to a tradition extending from Tornheim's harmonic double series and
Mordell's multiple sums to Witten's two-dimensional gauge theory, Zagier's
Witten zeta functions, and the modern theory of Witten and root-system zeta
functions
\cite{Tornheim1950,Mordell1958,Witten1991QuantumGauge2D,Zagier1994ValuesZeta,MatsumotoTsumura2006,Komori2008MordellTornheim}.
The Farey-neighbor parametrization and primitive-visible constraint belong to
the arithmetic of Farey fractions and visible lattice points, where Hata,
Boca--Cobeli--Zaharescu, and Boca--Zaharescu's correlation theory provide close
analogues \cite{Hata1995Farey,BocaCobeliZaharescu2000,BocaZaharescu2005}.  On
the lattice-counting side, the result sits between the smooth convex-body
theory of Huxley, Kr\"atzel, Nowak, and collaborators
\cite{HuxleyNowak1996,KratzelNowak1992,Huxley2000,IvicKratzelKuehleitnerNowak2004}
and the rational-polyhedral Ehrhart theory surveyed by Beck--Robins
\cite{BeckRobins2015}.  The domain is neither a rational polygon with stable
Ehrhart quasi-polynomial and lattice-perimeter term, nor a smooth strictly
convex curve governed directly by curvature-based exponential sums.  The Farey
dissection replaces both mechanisms by exact local tropical enumeration.  The
local half-open Pick formula, the Mordell--Tornheim pole, complete-period phase
averages, and moving-prefix corrections are classical in flavor. The new point
is their interaction in a single global decomposition: the \(s=2/3\) bulk
residue is present analytically but cancels arithmetically, while the surviving
coefficient comes from edge strips and finite-prefix defects.

\section*{Acknowledgments}
M.S. thanks Grigory Mikhalkin for his enduring encouragement and Dmitrii Korshunov for several inspiring communications, in particular for pointing out that \(L\) is the ``typical shape'' of a convex lattice polygon. He is especially grateful to Stanislav Shkolnikov, whose extensive numerical experiments revealed a structural issue in the lattice-point-counting mechanism, led to the completion of Section~\ref{sec_6}, and substantially simplified the coefficient \(C_L\).

The work of E.L. and M.S. was supported by the Simons Foundation under grant SFI-MPS-T-Institutes-00007697 and by the Ministry of Education and Science of the Republic of Bulgaria under grant DO1-239/10.12.2024. E.L. gratefully acknowledges Cinvestav for a sabbatical leave during which this work was prepared and submitted, as well as the Institute for the Mathematical Sciences of the Americas (IMSA) at the University of Miami for its hospitality and support on several occasions during the preparation of this work. E.L. also thanks E.S. for assistance with the \LaTeX{} preparation of the manuscript and for valuable technical and practical support.

N.K. thanks Fedor Petrov for suggesting the use of equiaffine invariants in this problem. He also acknowledges the Young Russian Mathematics grant (2018--2020), under which the investigation of the residue problem began; its completion ultimately required six additional years.

\printbibliography 

\newpage
\thispagestyle{empty}
{\centering
\ 
\vspace{77pt}

Guangdong Technion-Israel Institute of Technology\\
241 Daxue Road, Jinping District\\ 
Shantou, Guangdong Province 515063, China\\
\vspace{5pt}

nikaanspb[at]gmail.com\vspace{81pt}

The Center for Research and Advanced Studies\\ of the National Polytechnic Institute (Cinvestav)\\
Av. Instituto Politécnico Nacional 2508\\
Col. San Pedro Zacatenco, Alcaldía Gustavo A. Madero\\
Mexico City 07360, Mexico\\\vspace{4pt}
and\\\vspace{4pt}
Institute of Mathematics and Informatics\\
Bulgarian Academy of Sciences\\
Akad. G. Bonchev, Sofia 1113, Bulgaria \\\vspace{5pt} 

elupercio[at]gmail.com\\\vspace{81pt}

Institute of Mathematics and Informatics\\
Bulgarian Academy of Sciences\\
Akad. G. Bonchev, Sofia 1113, Bulgaria \\\vspace{5pt}

m.shkolnikov[at]math.bas.bg\\

}

\end{document}